\documentclass[12pt]{article}
\usepackage[margin=3.8cm]{geometry}
\usepackage{graphicx}
\usepackage{epsfig}
\usepackage{amsmath}
\usepackage{amssymb}
\usepackage{caption}
\usepackage{color}
\usepackage{enumerate}
\newcommand{\no}{\nonumber}
\newcommand{\ra}{\rightarrow}
\newcommand{\eps}{\mbox{$\epsilon$}}
\newcommand{\be}{\begin{equation}}
\newcommand{\ee}{\end{equation}}
\newcommand{\ba}{\begin{eqnarray}}
\newcommand{\ea}{\end{eqnarray}}
\newcommand{\bpm}{\begin{pmatrix}}
\newcommand{\epm}{\end{pmatrix}}

\newcommand{\R}{\mathbb{R}}

\newcommand{\gd}{\mbox{gd}}

\renewcommand{\H}{{\mathcal H}}
\newcommand{\beq}{\begin{equation}}
\newcommand{\eeq}{\end{equation}}

\newcommand{\sech}{\mbox{sech}}

\newtheorem{theorem}{Theorem}[section]
\newtheorem{lemma}{Lemma}[section]
   \newcommand{\blem}{\begin{lemma}}
   \newcommand{\elem}{\end{lemma}}
\newtheorem{proposition}{Proposition}[section]
   \newcommand{\bprop}{\begin{proposition}}
   \newcommand{\eprop}{\end{proposition}}

\newtheorem{remark}{Remark}[section]
\newtheorem{conjecture}{Conjecture}

\begin{document}

\title{Ballistic Orbits and  Front Speed Enhancement for ABC Flows}

\author{Tyler McMillen\thanks{Department of Mathematics, California State University at Fullerton, 
Fullerton, CA 92834, USA. Corresponding author. Email: tmcmillen@fullerton.edu. }, 
 Jack Xin\thanks{ Department
of Mathematics, University of California at Irvine, Irvine, CA
92697, USA. Email: jxin@math.uci.edu, yyu1@math.uci.edu.}, 
  Yifeng Yu$^{\dagger}$,
and Andrej Zlato\v{s}\thanks{ Department of Mathematics, University of Wisconsin, 
Madison, WI 53706, USA.  Email: andrej@math.wisc.edu.}  }

\maketitle

\begin{abstract}
We study the two main types of trajectories of the ABC flow in the near-integrable regime:  spiral orbits  and edge orbits.
The former are helical orbits which are perturbations of similar orbits that exist in the integrable regime, while the latter exist only in the non-integrable regime. 
We prove existence of 
ballistic (i.e., linearly growing) spiral orbits by using the contraction mapping principle in the Hamiltonian formulation, and we also find and analyze ballistic edge orbits.
We discuss the relationship of  existence of these orbits with questions concerning front propagation in the presence of flows, in particular, the question of
linear (i.e., maximal possible) front speed enhancement rate for  ABC flows.
 
\end{abstract}
\bigskip

\bigskip


\tableofcontents

\section{Introduction}

Front propagation in complex fluid flows arises in many 
areas of science, including combustion (e.g., in internal combustion engines)~\cite{Ro95,W85}, growth of populations (such as plankton)
in the ocean~\cite{A}, and 
chemical reactions in stirred liquids~\cite{PS05,Xin_09}. 
A longstanding fundamental problem is to characterize and quantify enhanced 
transport (front propagation and particle diffusion)  
in fluid flows containing complex and turbulent 
streamlines (see \cite{BCVV95,CW91,ChSo,FP97,MajKr99,PS05,Pet00,Siv,Xin_09,Yak88} 
and references therein, and \cite{baso12, malibosomi15} for recent work on so-called ``burning invariant manifolds'' in the analysis of advection-reaction diffusion systems). In the last two decades, significant progress has been achieved
in the study of these questions for
prototype partial differential equation (PDE) 
models \cite{MajKr99,Xin_09}. For instance, various analytical results were obtained concerning 
effective diffusion \cite{FP94,H03} as well as turbulent front speeds in reaction-diffusion \cite{Aud,KR, NR,RZ_07,Xin_09,XY5,Z_10,Z_10b} 
and G-equation \cite{CNS,XY1,XY5} models with spatially periodic  
incompressible flows in the advection dominated regime.  

Currently available analytical results concerning
asymptotics of front speed enhancement in the limit of large amplitude periodic advection are either limited to, or much more applicable in, two space dimensions.
This is because a detailed 
understanding of the trajectories of the underlying advection-generated dynamical system
is necessary, and there is a dramatic difference between 
two (2D) and three (3D) space dimensions. 
In particular, a crucial role is played by trajectories that extend to infinity, if they exist.  The question of their existence and properties can be effectively addressed in 2D via phase portrait analysis and integrability, but it becomes much more challenging in 3D due to the loss of 
integrability and emergence of chaos.

\subsection{ABC flows}

Motivated by the above questions, in this paper we study transport properties of a prime example of a complex steady incompressible periodic flow in 3D, the classical   
Arnold-Beltrami-Childress (ABC) flow \cite{ar65,Dom_86,Fr93}.  (See the end of this introduction for applications of our analysis to front propagation.)   Its standard form is
\ba
x' & = &  A\, \sin z+   C \, \cos y \no \\
y' & = &  B\, \sin x +   A\, \cos z \label{abc1} \\
z' & = &   C \,  \sin y + B\, \cos x \no 
\ea 
and we note that the vector field on the right-hand side is also a steady solution of the 3D Euler equations.
This system can also be written in the 
form
\ba
\bpm x \\ y \epm' & = & \bpm 0 & 1 \\ -1 & 0 \epm \nabla H(x,y) + A \bpm \sin z \\ \cos z\epm
\label{e1}
\\[.1in]
z' & = & H(x,y)
\label{e2}
\ea
where 
\[H(x,y) = B \cos x + C \sin y.
\]  
(Obviously, similar forms for $(x,z)$ or $(y,z)$ in place of $(x,y)$ exist as well.)  
Thus, the system is integrable when $A=0$ (or, indeed, if any of the three parameters is zero \cite{Dom_86}).  

\begin{figure}[h!]
\begin{center}
          \includegraphics*[width=3.5in]{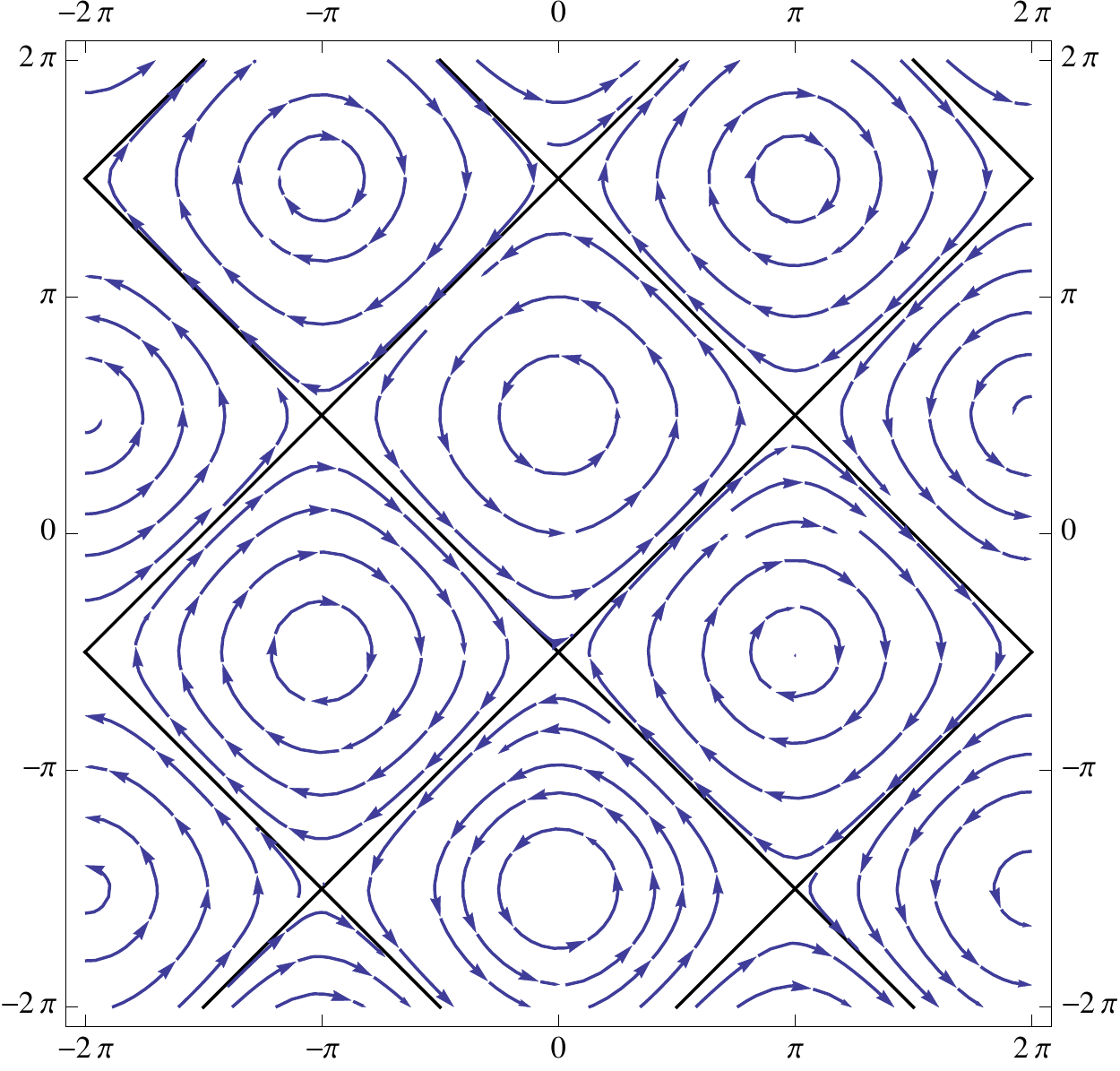}
\end{center}
\caption{Trajectories projected on the $xy-$plane in  the conservative case $A=0$ and $B=C=1$.}
\label{conservativeflow}
\end{figure}

The conservative case $A=0$ and $B=C=1$ is shown in Figure~\ref{conservativeflow}.  Heteroclinic orbits, which connect saddle points, are the contours $H(x,y)=0$.  The other trajectories are closed orbits or fixed points.  We will refer to the region enclosed by a heteroclinic cycle in this conservative case as a \emph{cell}.  Notice that $H(x,y)>0$ in those cells where the flow is counterclockwise, while $H(x,y)<0$ in those where the flow is clockwise.

\subsection{KAM regions and edge orbits}

In this paper we are interested in the near-integrable case $0<A\ll 1$, and most of the results will concern the case $B=C=1$.  (See \cite{XYZ} for a related analysis in the symmetric case $A=B=C=1$.)  In this case, two distinct types of trajectories exist, depending on their initial conditions:  
\begin{enumerate}[(i)]
\item 
\emph{spiral orbits}, in which $x,y$ oscillate within a single cell and $z$ grows monotonically; and 
\item 
\emph{edge orbits}, in which the projections of the trajectories  on the $xy$-plane repeatedly cross cell boundaries.  
\end{enumerate}
Due to our intended applications to front propagation, we are particularly interested in {\it ballistic} orbits, that is, those for which one or more coordinates ($z$ for spiral orbits and at least one of $x,y$ for edge orbits) grow linearly as $t\to\infty$.  
The following figures illustrate the types of behavior that can occur.

\begin{figure}[h!]
\begin{center}
          \includegraphics*[width=4in]{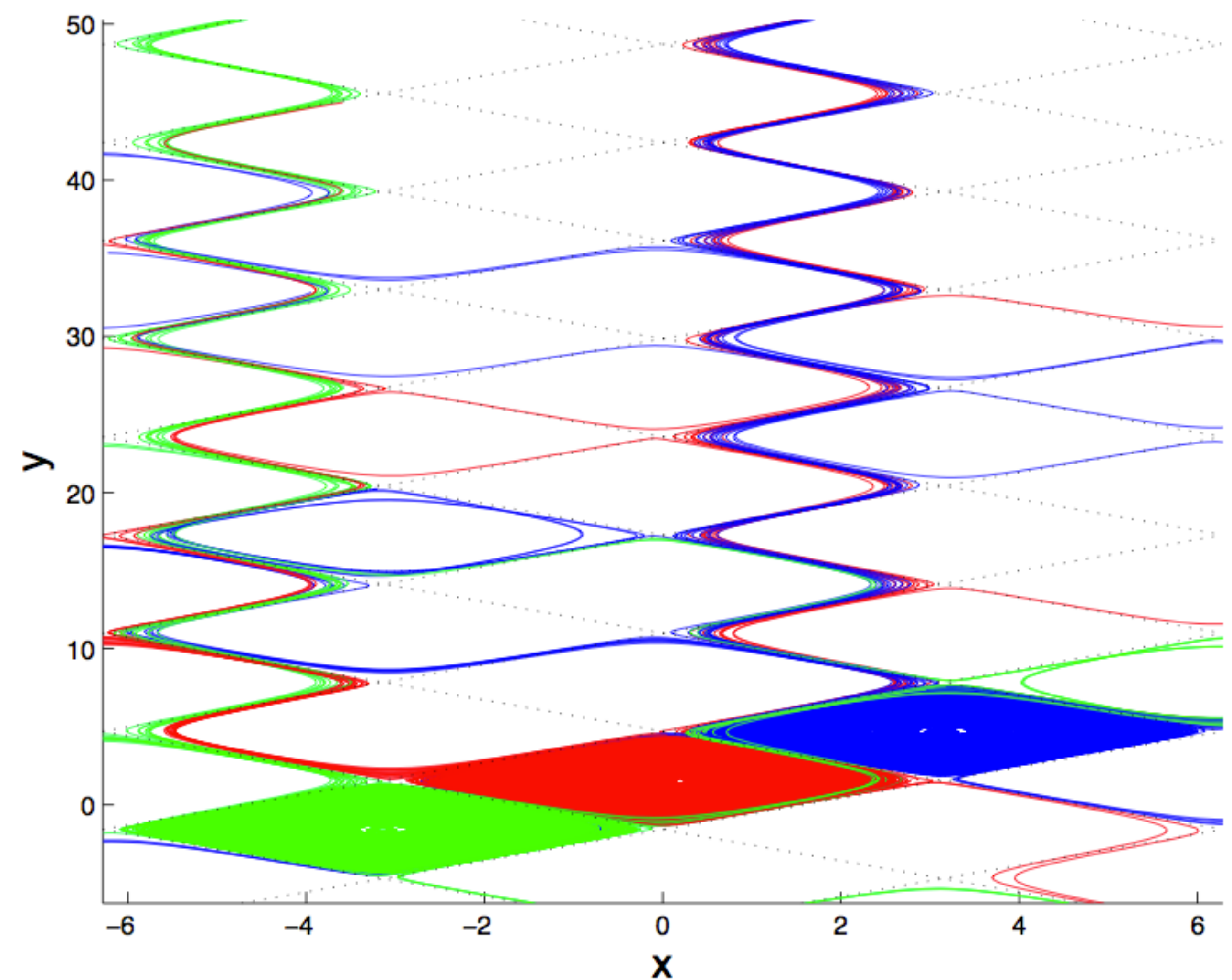}
\end{center}
\caption{Trajectories for $(A,B,C)=(.1,1,1)$ and $t\in[0,100]$, projected on the $xy$-plane.  Dotted lines represent cell boundaries.}
\label{trajectory1}
\end{figure}

Figure~\ref{trajectory1} shows the $xy$-plane projections of 720 trajectories with initial conditions uniformly distributed in 3 adjacent cells, and the  curves are color-coded according to in which cell they start.  For all trajectories, the initial $z$-value is $z(0)=0$.
Trajectories that start near the center of a cell are spiral orbits and remain in the cell;  a typical such trajectory is shown in Figure~\ref{spiral}.    
There is also a layer near the edges of the cells where trajectories cross into neighboring cells and we call these edge orbits by analogy with Rayleigh-B\'enard convection.

\begin{figure}[h!]
\begin{center}
          \includegraphics*[width=4.5in]{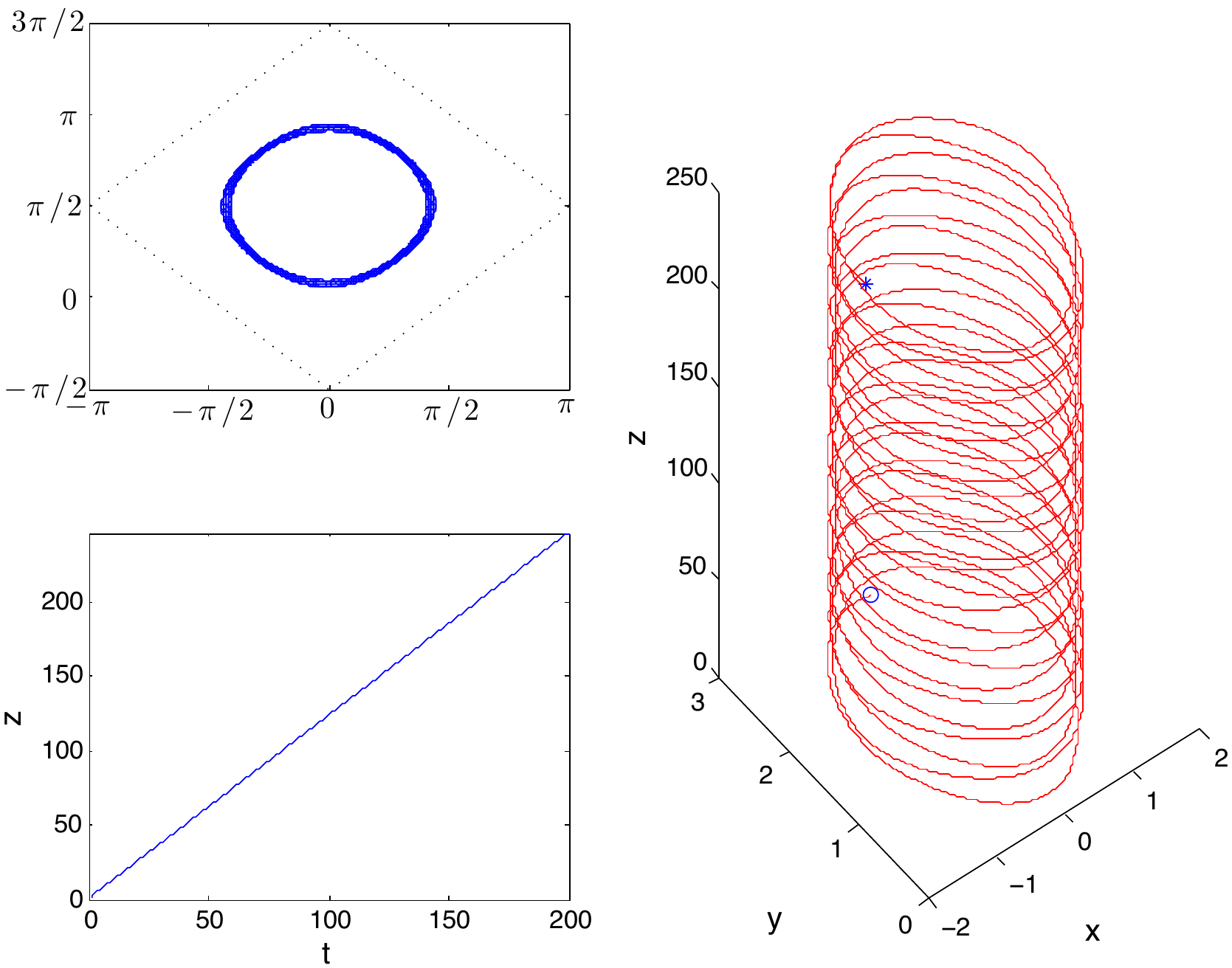}
\end{center}
\caption{A single trajectory with $(A,B,C)=(.1,1,1)$. The upper left panel is the projection on the $xy$-plane.  In the right panel, the initial condition is marked with a circle and the end position is marked with a star.}
\label{spiral}
\end{figure}

When $(x,y)$ stay within a cell, the sign of $z'$ does not change, so $z$ can be treated as a time variable and the system can be reduced to a 2D system.  This is the approach taken in \cite{huzhda98,zhkwlihu93}.  
Then \eqref{e1} can be written as
\be
H(x,y)\,\frac{d}{dz} \bpm x \\ y \epm = \bpm C\, \cos y \\ B\, \sin x\epm + A \bpm \sin z \\ \cos z \epm.
\ee
Note that $H$ has the same sign within a cell and when $A\ll 1$,  the $A\bpm \sin z \\ \cos z\epm$ term acts as a small periodic  forcing term on the conservative system.  Thus, one can expect to see trajectories that, when projected on the $xy$-plane, are small perturbations of the conservative flow (as in Figure~\ref{spiral}).  In fact, in \S\ref{sec.spiral} we will prove existence of ballistic spiral orbits where $x$ and $y$ are $2\pi$ periodic in $z$.  Note that while a $2\pi$ periodic solution was mentioned in \cite[cf. pp 377-8]{Dom_86}, the authors only discussed an approximate solution there.  Such a periodic solution does not follow from KAM type theorems for 3 dimensional flows (e.g.  \cite{MW1994}),  which only provide quasi-periodic solutions.  Nor does it follow from  the Melnikov method, since for the ABC flow it only leads to periodic orbits where $x$ and $y$ are $2m\pi$ periodic in $z$ for $m\geq 2$.  See also the remark after Theorem 3.2 in \cite{zhkwlihu93} and the paragraph after Theorem \ref{persol}.  Moreover,  in terms of the application to front propagation,  we believe that (\ref{limitformula})  attains maximum along  this type of  special spiral orbits for $p=(0,0,1)$.

To visualize the types of conditions that lead to spiral or edge orbits, we computed trajectories for initial conditions taken at 80000 points randomly distributed within each cell (with a fixed value of $z(0)$) and found those for which the trajectories never leave the cell for $t\in[0,50]$.  (The picture seems to stay the same for large $t$.)  These points are seen in Figure~\ref{cellflow}.   
Since $H$ has the same sign inside each cell,
trajectories that never leave their starting cell are spiral orbits with a monotone $z$ coordinate.   This is the 
Kolmogorov-Arnold-Moser (KAM) regime, and for this reason
we call the set of initial conditions within a cell for which trajectories never leave that cell a \emph{KAM region}.   Of course, trajectories may leave the KAM region while remaining inside the cell---they are trapped in the cell, not necessarily in the KAM region.  Indeed, trajectories starting from these initial conditions appear to fill the entire cell.
From Figure~\ref{cellflow} we see  that the KAM region shrinks with increasing $A$ (and disappears altogether for large $A$).  Note that the KAM region depends on which value $z(0)$ is chosen.

\begin{figure}[h!]
\begin{center}
          \includegraphics*[width=2.2in]{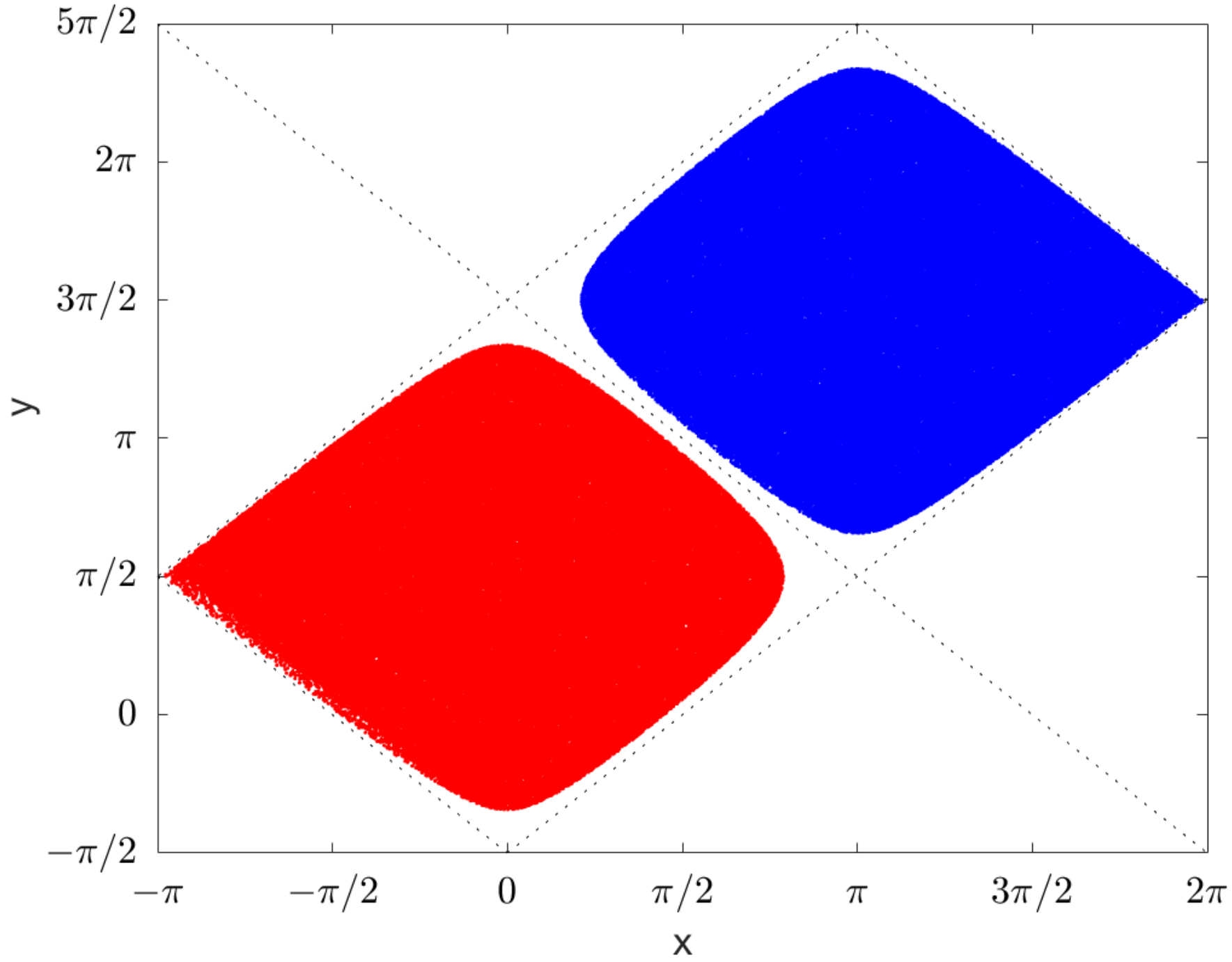}
          \hskip .1in
          \includegraphics*[width=2.2in]{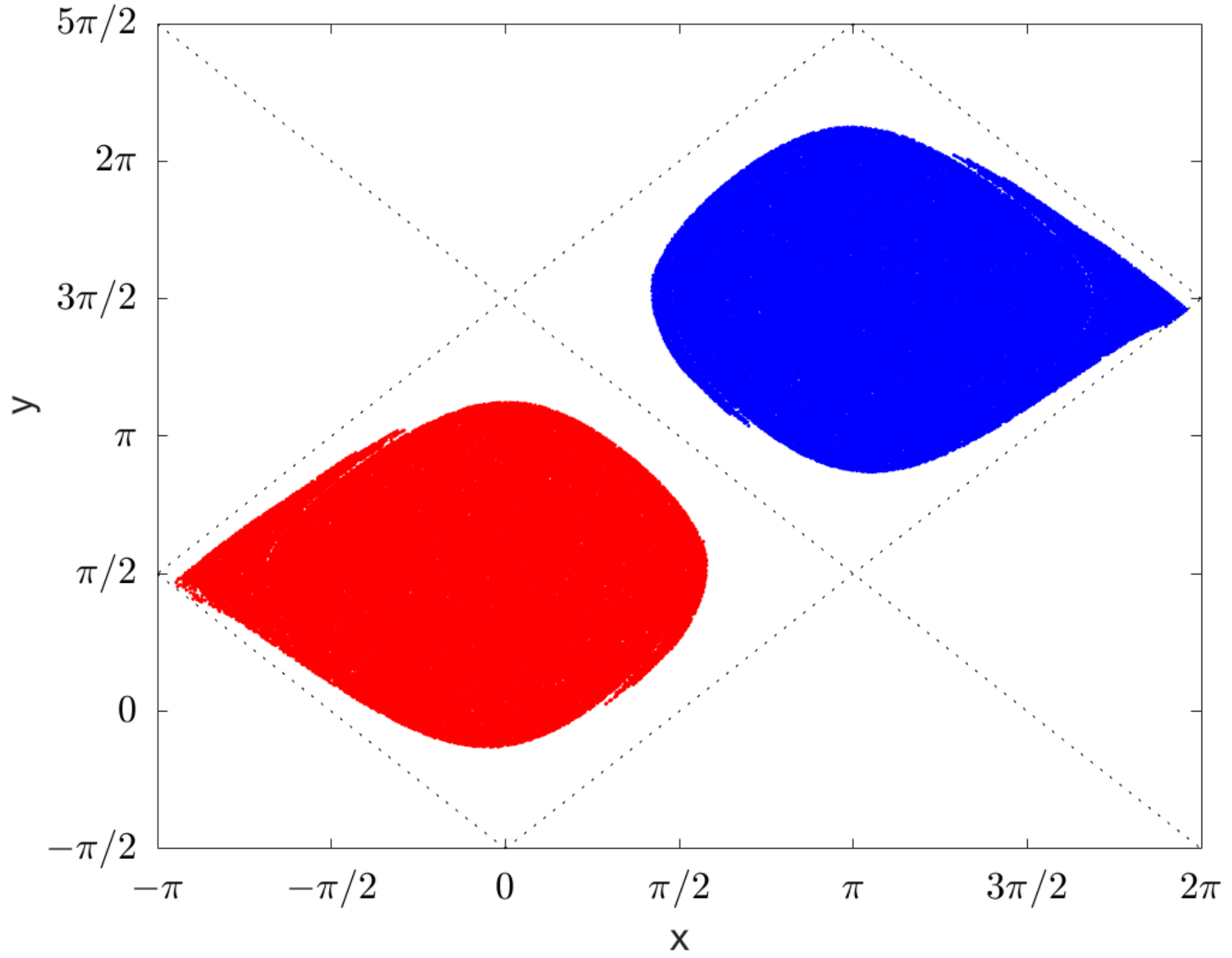} 

          \includegraphics*[width=2.2in]{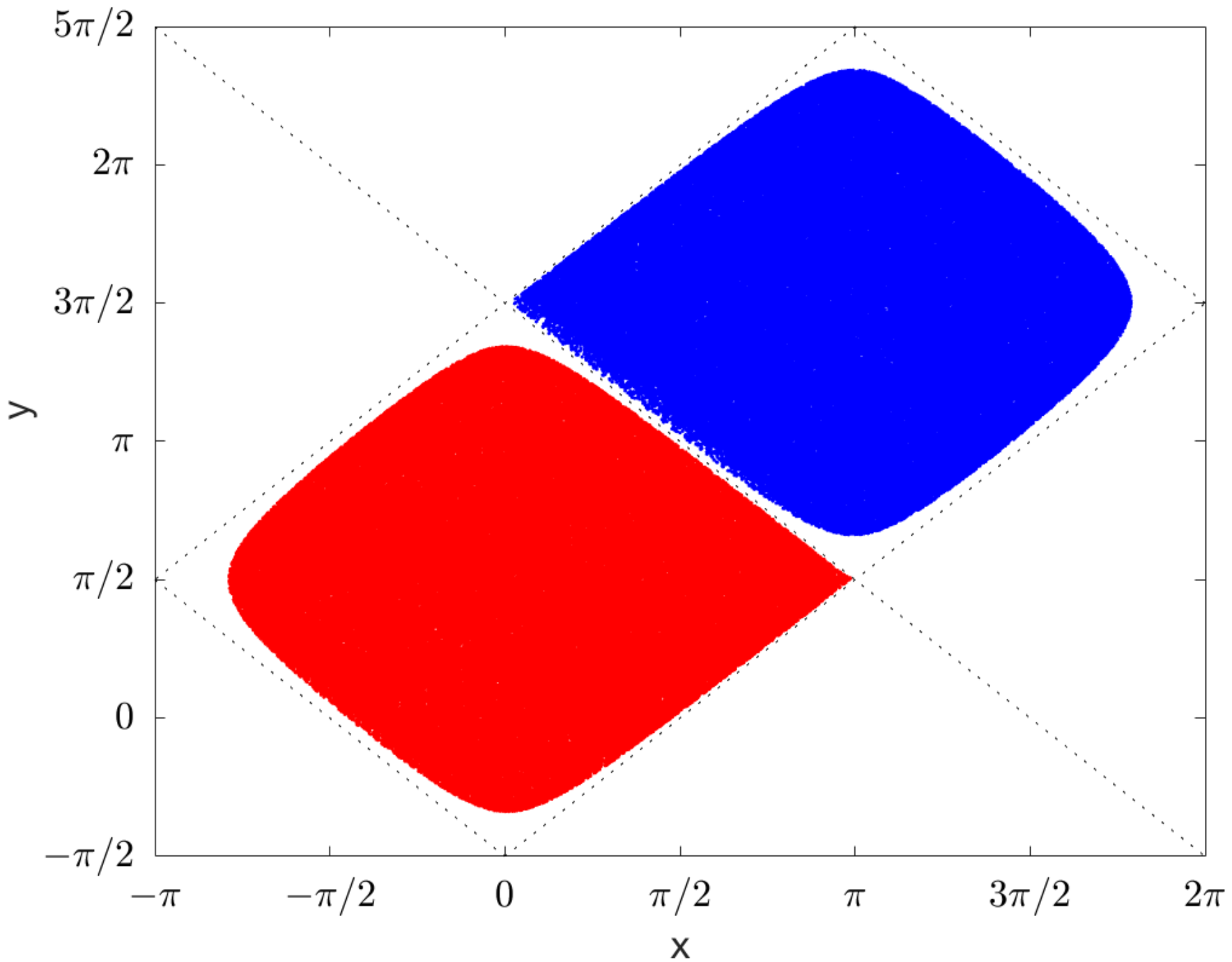}
          \hskip .1in
          \includegraphics*[width=2.2in]{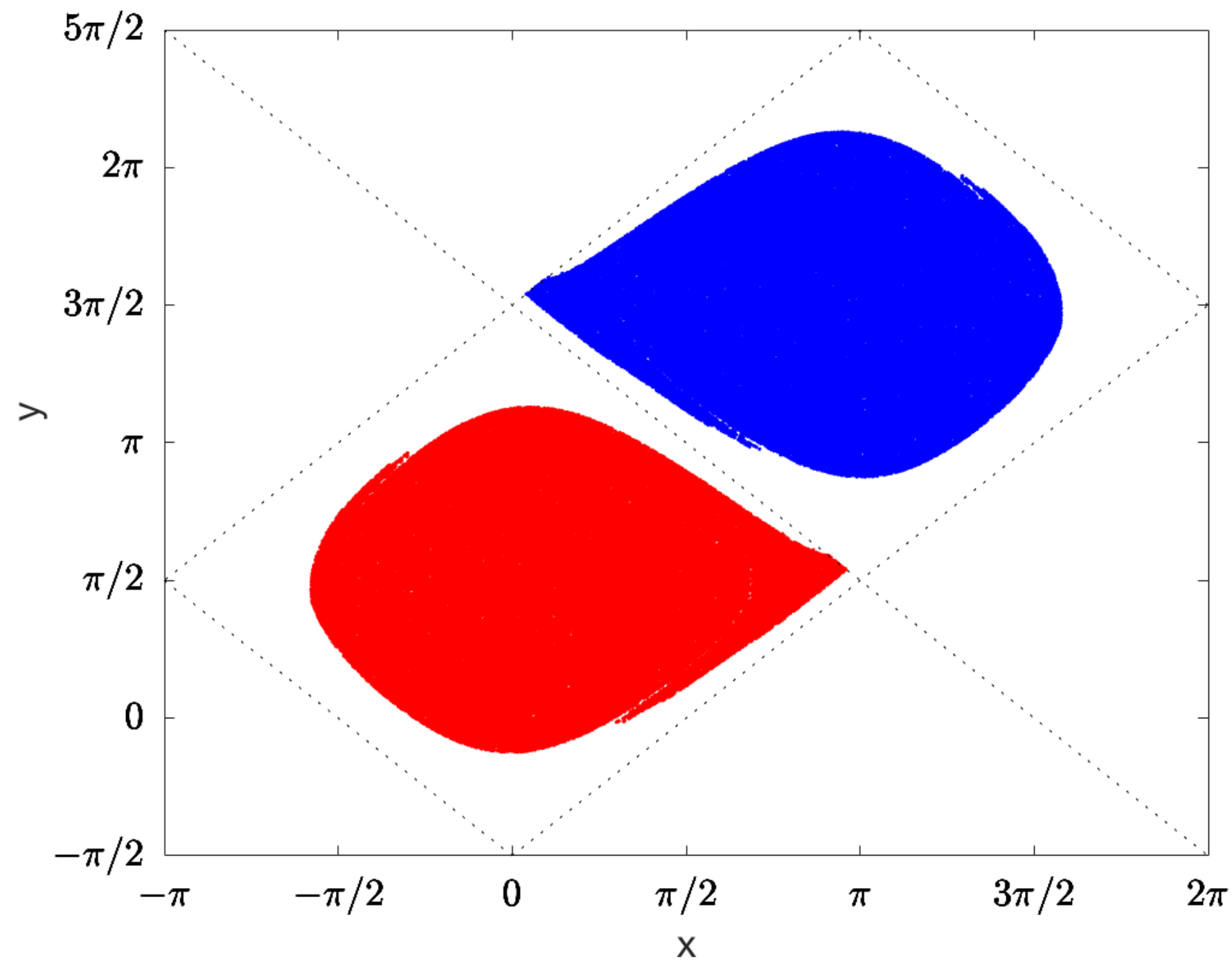}           
          
\end{center}
\caption{KAM regions of initial conditions of trajectories that do not leave a cell for $t\in [0, 50]$.  Left: $A=.05$; right: $A=.25$.  Top: $z(0)=0$; bottom: $z(0)=\pi$.}
\label{cellflow}
\end{figure}

Figure~\ref{trajectory1} also suggests that trajectories near the edges of the cells, starting outside the KAM regions from Figure~\ref{cellflow},  cross cell boundaries multiple (indeed, infinitely many) times.  
For these trajectories, two of which are seen in Figure~\ref{xzspiral}, $z$ is bounded while one (or both) of $x$ and $y$ grows.  They may be spirals (periodic in $z$ and periodic mod $2\pi$ in $x$ and $y$), quasi-periodic in $z$, or non-periodic.
  We will prove  existence of ballistic edge orbits, such as in the bottom part of Figure~\ref{xzspiral},  in \S\ref{sec.xygrowth} by exploiting the symmetries of the system \eqref{abc1}.

\begin{figure}[h!]
\begin{center}
          \includegraphics*[width=3.35in]{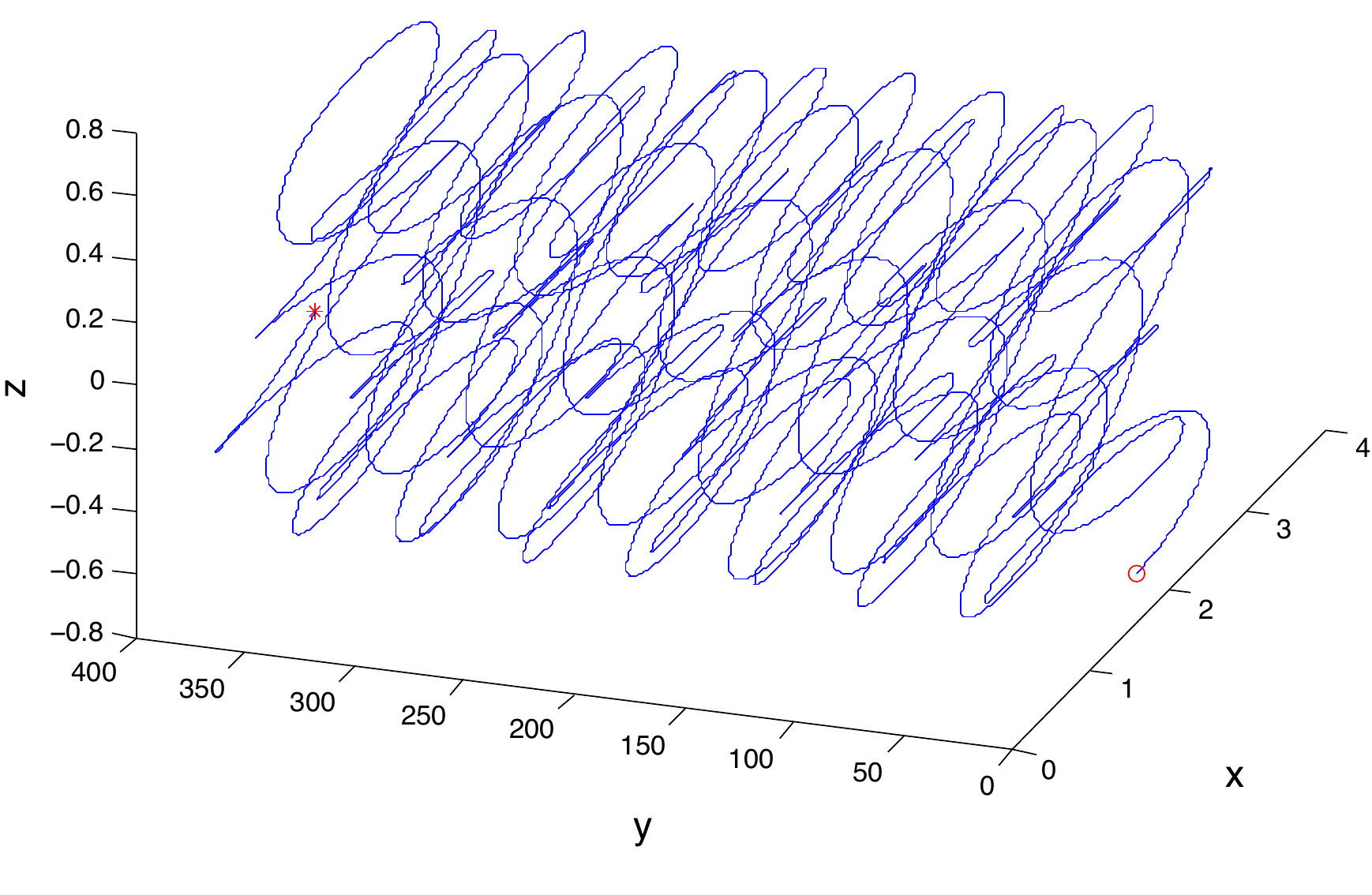}
          
           \includegraphics*[width=3.35in]{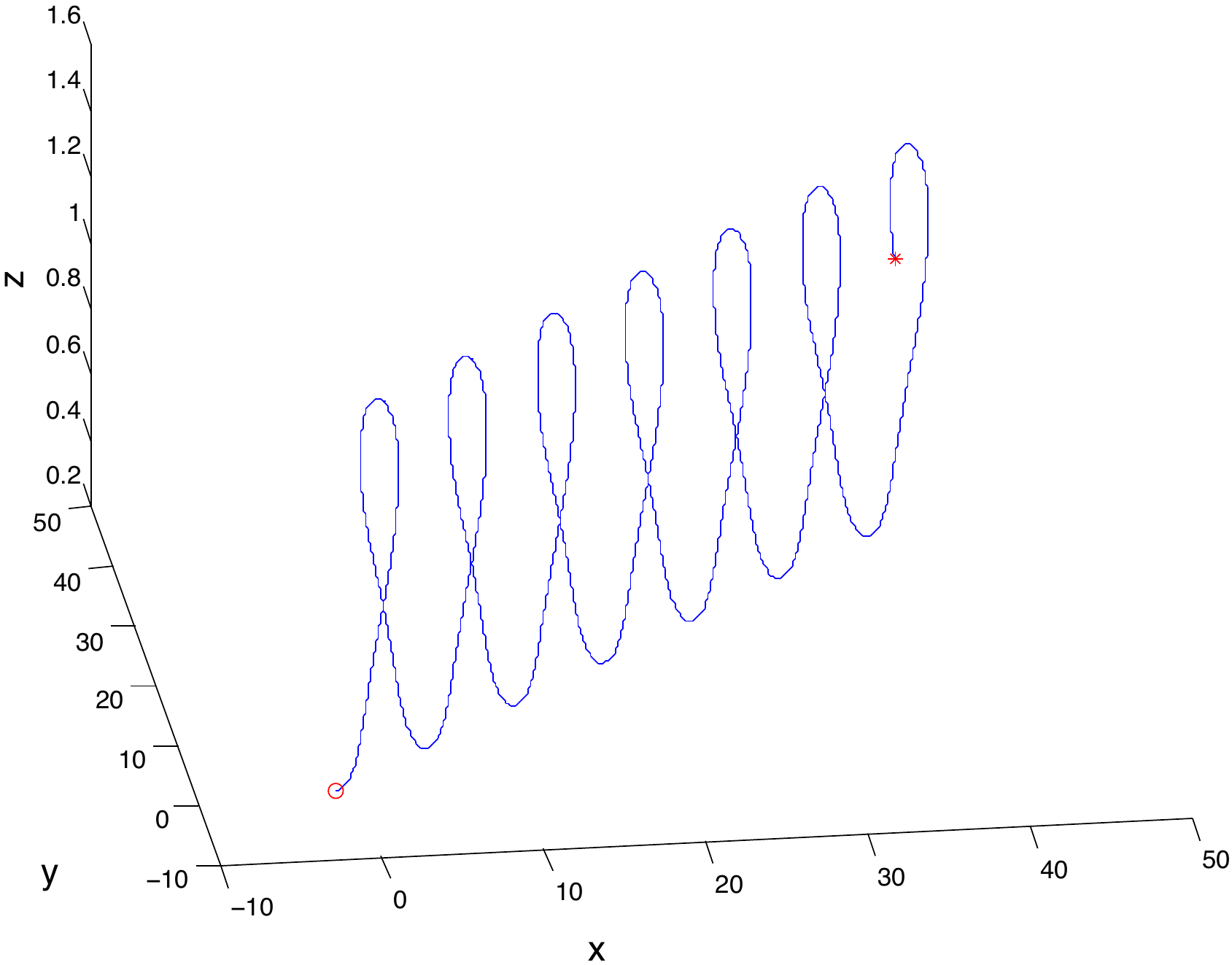}
\end{center}
\caption{Two trajectories for $(A,B,C)=(.1,1,1)$.  Top:  $(x(0),y(0),z(0))=(\pi/2 , 0, -.64)$;
bottom: $(x(0),y(0),z(0))=(-\pi/2,0, .2254)$.}
\label{xzspiral}
\end{figure}

The flow \eqref{abc1} may also be viewed as a flow on the torus $\mathbb{T}^3$.  In Figure~\ref{edges} we plot the trajectories from Figure~\ref{xzspiral} projected on the $xy$-plane, mod $2\pi$.
Two more trajectories are seen in Figure~\ref{edge4}, illustrating the variety of behaviors that can be obtained by varying  $z(0)$.  Note that the bottom one appears to be periodic.

\begin{figure}[h!]
\begin{center}
          \includegraphics*[width=3.6in]{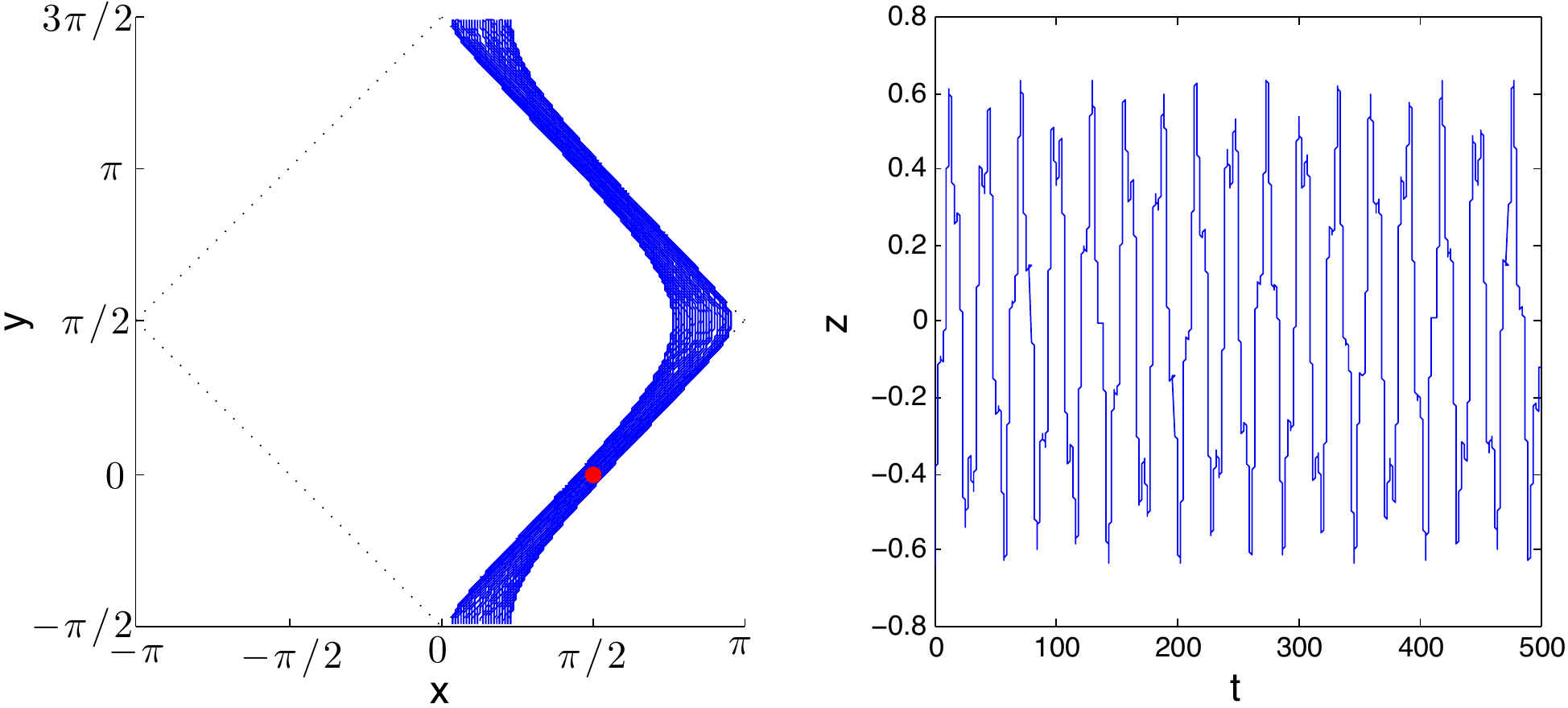}
          
           \includegraphics*[width=3.6in]{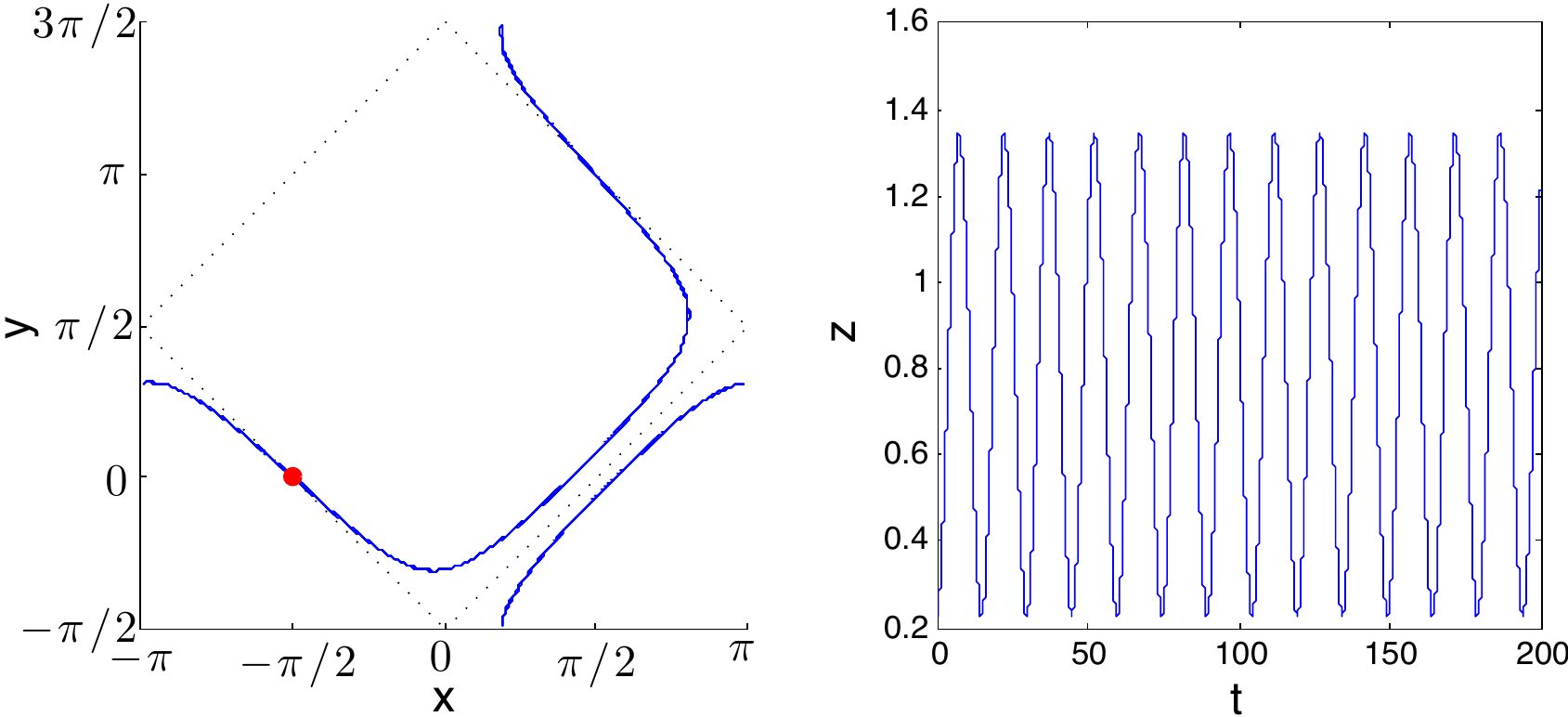}  
\end{center}
\caption{Projection of the orbits from Figure~\ref{xzspiral} on the torus in the $xy$-plane.}  
\label{edges}
\end{figure}

\begin{figure}[h!]
\begin{center}
	\includegraphics*[width=3.6in]{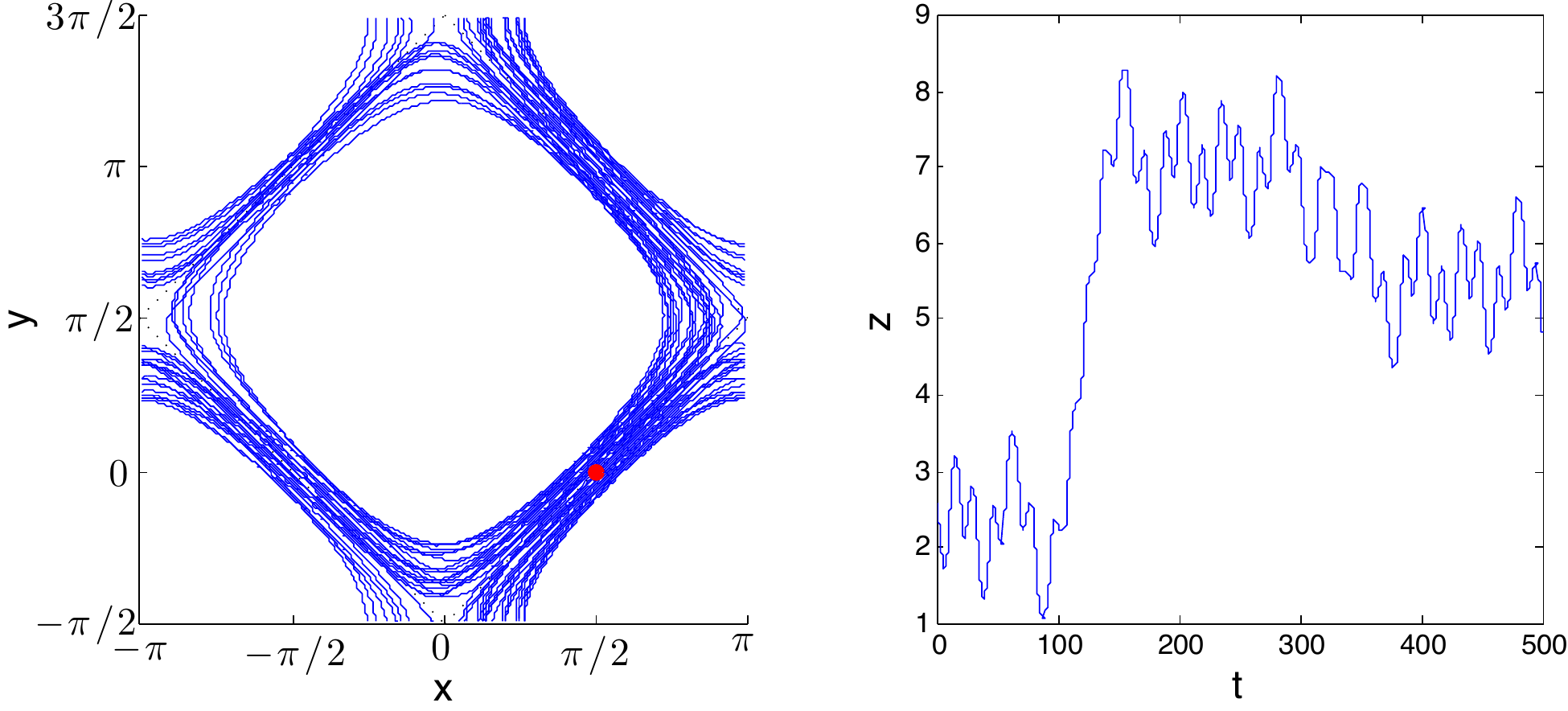}

          \includegraphics*[width=3.6in]{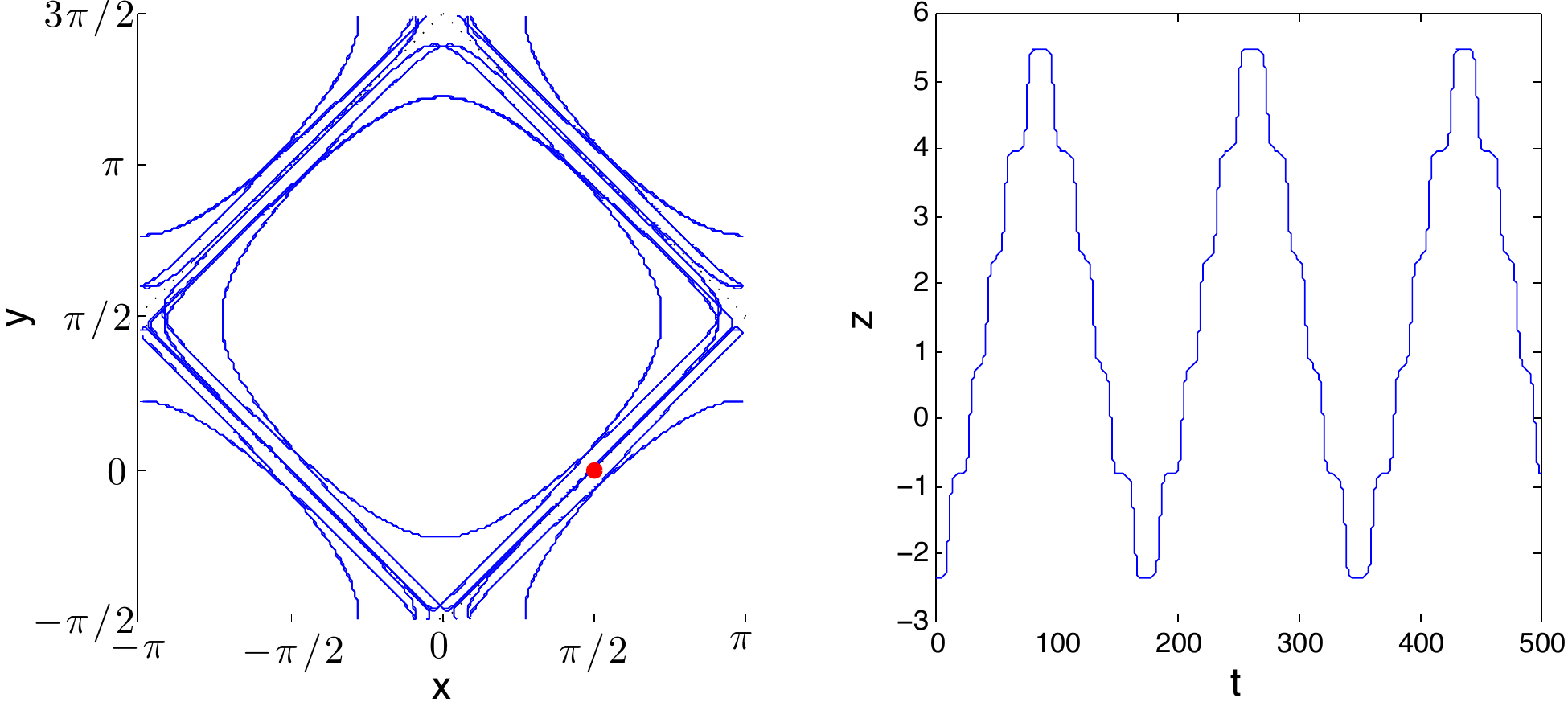}
\end{center}
\caption{Top: $z(0)= 2.6$; bottom:  $z(0) = -3\pi/4+.001$.}
\label{edge4}
\end{figure}

\subsection{Applications to front propagation}

Let us now discuss application of results regarding existence of ballistic (spiral or edge) orbits to front propagation.  

The G-equation is a well--known model in  turbulent combustion \cite{Pet00,W85}.  
Let the flame front be the zero level set of a reference function $G(x,t)$,
where the burnt and unburnt regions are $\{G(x,t) < 0\}$ and $\{G(x,t) > 0\}$, respectively.  See Figure~\ref{frontfig1}.   The propagation of the flame front  obeys the simple motion law  $ {v}_{n}=s_l+V(x)\cdot \vec{n}$, that is,  the normal velocity is the laminar flame speed $s_l$  plus the projection of fluid velocity  $V$ on the normal direction.    This leads to the  level-set PDE 
\begin{equation}
G_t + V(x)\cdot DG + s_l |DG|=0. \label{ge1}
\end{equation}

\begin{figure}[h!]
\begin{center}
         \includegraphics*[height=1.5in]{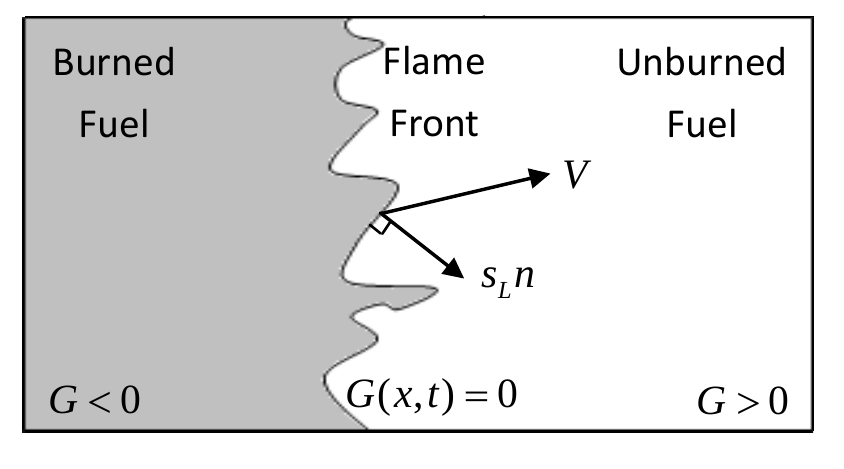}
\end{center}
\caption{A level-set model for flame propagation.}
\label{frontfig1}
\end{figure}

In this paper,  we only consider the simplest case of constant $s_l$, and without loss of generality we assume $s_l=1$.   
For a unit vector $p\in   \R ^n$,   let $G_p(x,t)$ be the viscosity solution to 
$$
\begin{cases}
G_t + V(x)\cdot \nabla G + |\nabla G|=0 \quad \text{in $\R^n\times (0,  \infty)$}\\
G(x,0)=p\cdot x.
\end{cases}
$$
When $V$ is periodic and incompressible (i.e.,  div$(V)=0$),   \cite{CNS,XY1} show that  the limit $s_T(p,V)=-\lim_{t\to  \infty}{G_p(x,t)\over t}$ exists and is at least 1.   Here  $s_T(p,V)$ represents the turbulent flame speed (or turbulent burning velocity) in the G-equation model. Roughly speaking, the turbulent flame speed is the averaged propagation velocity in the presence of the flow $V$.  A simple example is the spreading of a wildfire fanned by strong winds (see Figure~\ref{frontfig2}).   

\begin{figure}[h!]
\begin{center}
\includegraphics*[height=1.5in]{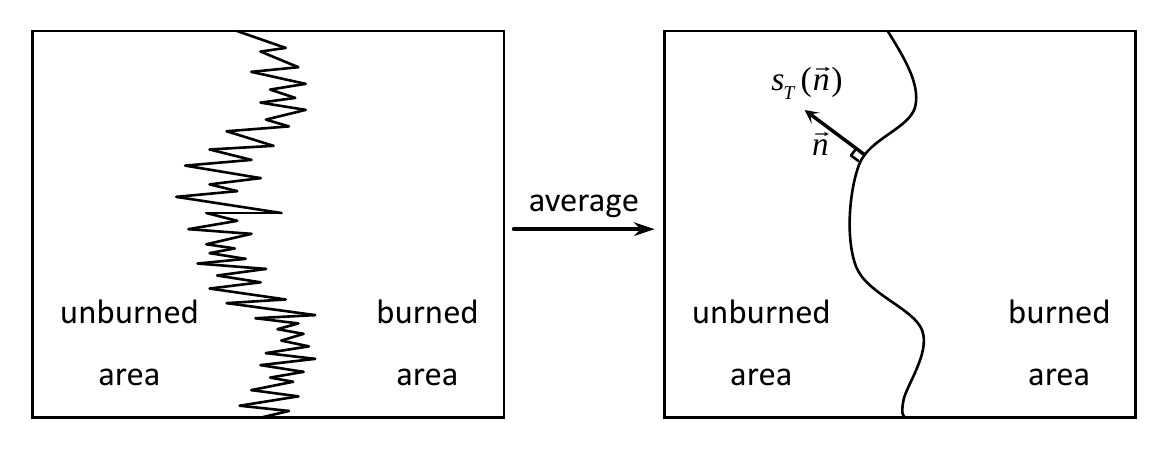}
\end{center}
\caption{Fluctuations along a flame front and their average.}
\label{frontfig2}
\end{figure}

To determine the turbulent flame speed is one of most important unsolved problems in turbulent combustion.  A basic question
is to understand, for physically meaningful and important
classes of flows,  the dependence of the limit $s_T$ on the flow
intensity.  That is, 
to identify the growth pattern of $s_T(p,AV)$ as the real parameter $A\to \infty$.

When $n=2$,  this problem has been thoroughly studied in   \cite{XY5}.    In particular,  when $V$ is a 2D cellular flow 
 it is known \cite{O,NXY,XY3} that 
$$
s_T(p,AV)=O\left({A\over \log A}\right)
$$
for all unit vectors $p$.   However,  the problem becomes much more challenging for  $n\geq 3$ due to the presence of chaotic structures.  As a first step,  one may want to investigate when
$$
\lim_{A\to \infty}{s_T(p,AV)\over A}>0,
$$
that is, when  the turbulent flame speed grows at least linearly in $A$ for a specific direction $p$.  This question was addressed in \cite{XY5}, where it was proved  that
\be\label{limitformula}
\lim_{A\to \infty}{s_T(p,AV)\over A}=\max_{\{\xi|\ \dot \xi=V(\xi)\}}\limsup_{t\to \infty}{p\cdot \xi(t)\over t }.
\ee

That is, $s_T(p,AV)$ grows at least linearly in $A$ precisely when there is an orbit $\dot \xi=V(\xi)$ whose dot product with $p$ diverges at least linearly to $\infty$.  In this paper we will prove the existence of such orbits for each $p$ when $V$ is an ABC flow with $0<A\ll 1$ and $B=C=1$, implying linear turbulent flame speed enhancement for these flows.

\medskip

Another well-known approach to modeling  front propagation is to study  traveling wave solutions to the reaction-diffusion-advection (RDA) equation
$$
T_t+V(x)\cdot \nabla T=d\Delta T+f(T).
$$
Here $T$ represents the reactant temperature, $V(x)$  is a prescribed fluid velocity, $d$ is the molecular diffusion constant, and $f$  is a nonlinear reaction function.  
The turbulent flame speed in this model is the minimal traveling wave speed  $c_{p}^{*}(V)$ in direction $p$ (see, e.g.,  \cite{Berrev,Xin_09}).   The dependence of $c_{p}^{*}(AV)$ on the flow intensity $A$  has also been studied extensively.  For example, when $V$ is a 2D cellular flow,   it was proved in  \cite{NR} that
$$
c_{p}^{*}(AV)=O\left(A^{1/4}\right).
$$
For general incompressible flows,  it was established in  \cite{Z_10} that
\be\label{limitformula2}
\lim_{A\to \infty}{c^{*}_p(AV)\over A}=\sup_{w\in\Gamma}\int_{\Bbb T^n}(V\cdot p) w^2dx,
\ee
where
$$
\Gamma= \left \{w\in  H^1(\Bbb T^n) \,\Big|\,   V\cdot \nabla w=0 \,\&\, ||w||_{L^2(\Bbb T^n)}=1 \,\&\, \|\nabla w||_{L^2(\Bbb T^n)}^{2}\leq f'(0) \right\}.
$$

Hence, this time we need a nice tube of orbits $\dot \xi=V(\xi)$ which all travel with a positive average speed in direction $p$  (as opposed to just a single orbit in the G-equation model) to conclude that $c_{p}^{*}(AV)$ grows at least linearly in $A$.  (This difference is non-trivial; see \cite{XY3} for an example of a 3D incompressible flow, the Robert cell flow, where $\lim_{A\to \infty}{c_{p}^{*}(AV)\over A}=0$ but $\lim_{A\to \infty}{s_T(p,AV)\over A}>0$ for $p=(0,0,1)$.)  
We are not able to prove the 
existence of such tubes for ABC flows $V$ with $0<A\ll 1$ and $B=C=1$ even in the KAM region since KAM type theorems do not provide any regularity of the set of quasi-periodic orbits.   However, we will present numerical evidence in \S\ref{prevalence} for the existence of such tubes.

\subsection{Organization of the paper}

The remainder of this paper is organized as follows.  
In \S\ref{sec.spiral} we will prove existence of ballistic spiral orbits in the 
KAM region, where the Hamiltonian function has a fixed sign.  
In \S\ref{sec.edgeflow} we will analyze the flow near the boundaries of the cells, 
where  the Hamiltonian changes sign, and prove existence of ballistic edge orbits there.  
We present both analytic and numerical arguments for existence 
of trajectories that cross edge boundaries infinitely many times, with the
analysis near the cell boundaries carried out in two different ways.  
We employ a perturbation method to approximate these trajectories 
based on the computable heteroclinic 
connections in the conservative case, and we also study computationally the percentage and persistence  
of the linearly growing non-KAM solutions.
All this analysis and results apply to the near-integrable regime of ABC flows.
We conclude with some remarks and conjectures in \S\ref{sec.conclusion}.

\newpage
\section{Spiral orbits in the KAM region}
\label{sec.spiral}
\setcounter{equation}{0}

When one of the parameters $A,B,C$ is small,  the ABC flow is perturbation of an integrable flow.  The standard KAM theorem can not be immediately applied due to oddness of the dimension, and various KAM-type theorems  have been developed to handle such cases (see, e.g., \cite{CS1990,MW1994}).   In this section we will prove existence of a special helical orbit (spiral orbit) such that $x,y$ are $2\pi$ periodic in $z$.   As is mentioned in the introduction and in the paragraph after  Theorem \ref{persol} below,  such an orbit  cannot be derived  from either KAM-type theorems or Melnikov methods.  Our strategy is instead to look at perturbations of ballistic linear solutions, such as the exact solution 
\be
(x,y,z) = (0,\pi/2, (B+C)t).
\label{int1}
\ee
of \eqref{abc1} with $A$=0.
Here $(x,y)$ is a fixed point of \eqref{e1}, the center of a cell in Figure~\ref{conservativeflow}, and
similar ballistic orbits 
exist when instead either $B=0$ or $C=0$.

In the near-integrable case $0<A = \epsilon \ll 1$ and  $B,C\sim 1$, one may be tempted to seek perturbative spiral orbit solutions of the form
\be
(x,y,z)= (u(t), \pi/2 + v(t),  (B+C)t + c^*t+w(t)),
\label{int2}
\ee
with a small constant $c^*$ and a small vector function $(u,v,w)(t)$.  Then \eqref{abc1} yields
\ba
u' & = & -C\sin(v) + \eps\, \sin((B+C)t+c^*t+w), \no \\
v' & = & B\sin(u) + \eps\, \cos((B+C)t+c^*t+w), \label{abc2} \\
w' & = & B (\cos(u)-1) + C (\cos(v)-1) - c^*, \no
\ea
and we may  also take, for instance, $(u,v,w)(0)=(0,0,0)$.  One then aims to construct a bounded global in time solution $(u,v,w,c^*)$ to \eqref{abc2}.  Consider the iteration scheme
\be
(u_n,v_n,w_n,c_n)(t)\rightarrow (u_{n+1},v_{n+1},w_{n+1},c_{n+1})(t)
\label{abc3}
\ee
with $n\geq 0$, given by
\ba
u_{n+1}' &=&-Cv_{n+1} - C(\sin(v_n)-v_n) + \eps\, \sin((B+C)t+c_nt+w_n), \no \\
v_{n+1}' & = & B u_{n+1} + B (\sin(u_n)-u_n) + \eps\, \cos((B+C)t+c_n t+w_n), \label{abc4} \\
w_{n+1}' & = & B (\cos(u_{n+1})-1) + C (\cos(v_{n+1})-1) - c_{n+1}, \no
\ea
with $(u_n,v_n,w_n)(0)=0$ for any $n\geq 0$, and $(u_0,v_0,w_0)(t)\equiv (0,0,0), \ c_0=0$.

One might hope that the mapping (\ref{abc4}) is a contraction if $\eps \ll 1$.
Of course, for $w_{n+1}$ to be uniformly bounded in time, 
$c_{n+1}$ must be the long time average of 
$B (\cos (u_{n+1}) -1) + C (\cos (v_{n+1}) -1)$.  This condition determines $c_{n+1}$.
However, the main difficulty is to show that $(u_{n+1},v_{n+1})$ can be obtained 
from $(u_{n},v_n,w_n,c_n)$ without encountering resonance or growth in time, 
and this is generally not the case. Consider $B=C=1$, when
the solution $(u_1,v_1)$ is a  
linear combination of $\sin t,\cos t, \sin 2 t,\cos 2t$. 
Since $\sin u - u$ contains odd nonlinearities, a term like  
$(\sin t)^3 (\sin 2t)^2$ arising in the quintic component 
of $\sin u_1 - u_1$ will generate 
$\sin 3t \sin 4t = (\cos t - \cos 7t)/2$ on the right hand side of (\ref{abc4}), 
where $\cos t $ is resonant!  Also, there are no additional parameters in the $(u,w)$ equations 
to zero out such modes.  Moreover, the $c_n$'s introduce additional frequencies 
besides $1$ (intrinsic frequency) and $2$ (initial driving frequency), 
causing $(u_n,v_n)$ to be at least quasi-periodic in $t$ and making 
solutions complicated.  
Similar problems occur with small $A$ or $B$.
\medskip

Hereafter, we shall work with the near-integrable case of a small $A=\epsilon$, and the values  of $B$ and $C$ of order one.

\subsection{A Hamiltonian form and an iteration scheme}

A better way to carry out the contraction mapping approach is 
to consider the Hamiltonian form of (\ref{abc1}), as in \cite{Zas_08,Tip_96}:
\be
dx/dz = \H_p,\;\; dp/dz = - \H_x, 
\label{abc5}
\ee
with the Hamiltonian
\be
\H = B\cos x + A (y \sin z - x\cos z) + C \, \sin y, \label{abc6}
\ee
where $y=y(x,p)$ is given implicitly by
\be 
p = B y\cos x + C(1-\cos y). \label{abc7}
\ee
We are interested in a periodic 
solution $(x,p)(z)$, treating $z$ as a time variable. Then by (\ref{abc7}),
$y$ becomes a periodic function of $z$, and the  
substitution of $(x,y)=(x,y)(z)$ into 
$z' = B \, \cos x + C \,\sin y$ recovers $z$ as a function of $t$.
\medskip

A derivation of (\ref{abc5})-(\ref{abc7}) is given here for the sake of  completeness. 
First, forming ratios of the equations in (\ref{abc1}) gives
\be
dx/dz = {1\over H} \, \H_y, \;\; dy/dz = -{1\over H}\, \H_x, \label{abc1a}
\ee
with
\ba
H & = & H(x,y) = B \cos x + C\sin y , \no \\
\H & =& \H(x,y,z) = H(x,y) + A(y \sin z - x \cos z). \label{abc1b}
\ea
Define
\be
p = p(x,y) = \int_{0}^{y} H(x,y') \, dy' =  B y\cos x + C(1-\cos y)
\label{abc1c}
\ee
and identify $\H(x,y)=\H(x,p)$.  The chain rule now gives
\be
 \H_y(x,y) = \H_p(x,p) p_y = \H_p(x,p) H (x,y), \label{abc1d}
\ee
which implies the first equation of (\ref{abc5}):
\[ dx/dz = {1\over {H(x,y)}} \, \H_y(x,y) = \H_p(x,p). \]
On the other hand, (\ref{abc1a}), (\ref{abc1d}), and
\[ \H_x(x,y) = \H_x(x,p) + \H_p(x,p) p_x \]
 imply the second equation of (\ref{abc6}):
\ba
dp/dz & = & p_x dx/dz + p_y dy/dz \no \\
 & = & p_x {1\over {H(x,y)}} \, \H_y(x,y) - p_y {1\over {H(x,y)}}\, \H_x(x,y) \no \\
& = & p_x \H_p(x,p) - \H_x(x,y) \no \\
& = & - \H_x (x,p). \label{abc1e}
\ea 
\medskip

In spite of the potentially complicated $t$-dependence of $z$, 
the $z$-dependence of $(x,p)$ is simpler. 
The unperturbed solution at $A=0$ is $(x_0,y_0,z_0)=(0,\pi/2, (B+C) t + c_0)$ for any real number $c_0$,  with
$p_0 = C+B\pi/2$. Write $y=\pi/2 + \hat{y}$ and $p = C+ B\pi/2 + \hat{p}$, with $\hat{y}, \hat{p}$ small.
Then (\ref{abc7}) gives
\ba
 \hat{p} & = & B \hat{y} + B \, (\pi/2 +\hat{y})(\cos x -1) + C \sin (\hat{y}) \label{abc9a} \\
& = & (B+C) \hat{y} + O( x^2) + O(\hat{y}x^2) + O(\hat{y}^3). \label{abc9b}
\ea
The Hamiltonian (\ref{abc6}) is written as (recall that $A=\epsilon$)
\ba
  \H  &=&  B \cos (x) + C \cos (\hat{y}) + \eps ((\pi/2 + \hat{y})\sin (z) - x \cos (z)) \label{abc10a}
\\
&=& B \cos (x) + C \cos {\hat{p}\over B+C} \no
\\
&& + \eps ((\pi/2 + \hat{p}/(B+C))\sin (z) - x \cos (z))+h.o.t. 
\label{abc10b}
\ea
In terms of the hat variables, (\ref{abc5}) reads (ignoring the higher order terms h.o.t. for the moment)
\ba
dx/dz & = & \H_{\hat{p}} = -{C\over B+C} \sin  {\hat{p}\over B+C} + {\eps \over B+C} \sin z, \no \\
d\hat{p}/dz & = & - \H_x= B \sin x + \eps \cos z. \label{abc11}
\ea
Separating the linear and nonlinear terms, we rewrite (\ref{abc11}) as
\ba
dx/dz + {C \hat{p} \over (B+C)^2} & = & -{C\over B+C} (- {\hat{p}\over B+C} +\sin  {\hat{p}\over B+C})+{\eps \over B+C} \sin z, \no \\[.1in]
d\hat{p}/dz - B x & = & B (-x + \sin x) + \eps \cos z. \label{abc12}
\ea 
The homogeneous linear part is:
\be  
dx/dz + {C \over (B+C)^2} \hat{p} = 0, \; d\hat{p}/dz - B x = 0, \label{abc13}
\ee
implying that the intrinsic (resonant) frequency is $\omega_o =\sqrt{BC}/(B+C)$. In other words, 
functions such as $\sin (\omega_0 z)$ and $\cos (\omega_0 z)$ cannot appear as forcing terms on the 
right hand side of (\ref{abc12}). 
The explicit forcing terms are $\sin z$ and $\cos z$, while 
$\omega_0 \in (0,1/2)$ (when $B,C>0$). The nonlinear terms of (\ref{abc12}) contain odd powers, hence generate only 
non-resonant integer frequencies in the following iteration scheme:
\ba
dx_{n+1}/dz + {C \over (B+C)^2} \hat{p}_{n+1} & = & -{C\over B+C} (- {\hat{p}_{n}\over B+C} 
+\sin  {\hat{p}_{n}\over B+C})+{\eps \over B+C} \sin z, \no \\
d\hat{p}_{n+1}/dz - B x_{n+1} & = & B (-x_n + \sin x_n) + \eps \cos z. \label{abc14}
\ea  
We then expect a $2\pi$-periodic solution $(x,p)$ (as a function of $z$) at small $\epsilon$, via 
establishing the contraction property of the map
\[ (x_n,\hat{p}_{n}) \ra (x_{n+1},\hat{p}_{n+1}) \]
in $L^2([0,2\pi])$. 

\subsection{Existence of ballistic spiral orbits}

We now study  invertibility of the linear operator in (\ref{abc14}).
\blem \label{lem1}
If $(f,g) \in (L^2([0,2\pi]))^2$, then there exist unique solution $(x,\hat{p}) \in (H^1([0,2\pi]))^2$
of the system
\ba
dx/dz + {C p \over (B+C)^2}  & = & f(z), \no \\
d p/dz - B x & = & g(z), \label{est1}
\ea 
satisfying the estimate (with $\|\cdot \|$ the $L^2$ norm and $\|\cdot \|_1$ the $H^1$ norm)
\be
\|(x,p)\|_{1} \leq \alpha(B,C) \|(f,g)\|, \label{est2}
\ee
for a constant $\alpha(B,C)$ depending only on $(B,C)$.
\elem

\begin{proof} 
The Fourier series representations $(f,g)=\sum_{j} (f_j,g_j) \, \exp\{i\, j z\}$ and 
$(x,p)=\sum_{j} (x_j,p_j)\,  \exp\{ i\, j z\}$ turn the system (\ref{est1}) into
\[ \left ( \begin{array}{ll}
          ij & C/(B+C)^2 \\
        -B & i j 
\end{array} \right )  \left (\begin{array}{l}
                               x_j \\
                                p_j 
                             \end{array} \right ) 
= \left (\begin{array}{l}
                               f_j \\
                                g_j 
                             \end{array} \right ). 
\]
The unique solution is
\[ \left (\begin{array}{l}
                               x_j \\
                                p_j 
                             \end{array} \right ) 
= (BC(B+C)^{-2} - j^2)^{-1} \left ( \begin{array}{ll}
          ij & -C/(B+C)^2 \\
        B & i j 
\end{array} \right )\left (\begin{array}{l}
                               f_j \\
                                g_j 
                             \end{array} \right ). 
\]
Clearly
\[ |j(x_j,p_j)| \leq \alpha_0 (B,C) |(f_j,g_j)|\]
with a $j$-independent constant $\alpha_0(B,C)>0$, 
implying the estimate (\ref{est2}). $\square$
\end{proof}

\medskip

Next, we turn to the higher order nonlinear terms ignored in (\ref{abc12}). It follows 
from taking gradient of (\ref{abc9a}) with respect to $(x,\hat{p})$ that
\be
B \hat{y}_{x} + B \hat{y}_{x} (\cos x -1) - B (\pi/2 + \hat{y})\sin x + C \hat{y}_{x} \cos \hat{y} = 0
\label{abc15}
\ee
and
 \be
1 = B \hat{y}_{p} + B (\cos x -1)\hat{y}_{p} + C \hat{y}_{p} \cos \hat{y}. \label{abc16}
\ee
Hence,
\ba
\hat{y}_{x} &=& {B(\pi/2 +\hat{y})\sin x \over B+C + C(\cos \hat{y} -1) +B(\cos x -1)}, \label{abc17}\\
\hat{y}_{\hat{p}} &=& {1\over B+C\cos \hat{y} + B(\cos x -1)}. \label{abc18}
\ea
We have from (\ref{abc10a}) that
\ba
\H_{\hat{p}} & = & ( - C \sin \hat{y} + \eps \, \sin z)\, \hat{y}_{\hat{p}}   \no \\
& = & { - C \sin \hat{y} + \eps \, \sin z \over C + B + C(-1 + \cos \hat{y})+ B(-1 +\cos x)} \no \\
& = & -{C\over C+B}\sin \hat{p} + {\eps \sin z \over B+C} + N_1 (x,\hat{p}), \label{abc19}
\ea
where $|N_1| \leq c_1 (x^2 + \hat{p}^2)$ for $(x,\hat{p})\leq \delta_1 =\delta_1(B,C) \ll 1$ and 
a positive constant $c_1 =c_1(B,C)$. Similarly,
\ba
\H_x & = & - B\, \sin x - C \hat{y}_{x} \, \sin \hat{y} +\eps \hat{y}_{x} \sin z -\eps \cos z, \no \\
& = & -(B \, \sin x +\eps \cos z) - \hat{y}_{x} (C \sin \hat{y} - \eps \sin z), \no \\
& = & -(B \, \sin x +\eps \cos z) - N_2 (x,\hat{p}) +\eps N_3(x,\hat{p}) \sin z, \label{abc20}
\ea
where $|N_2| \leq c_2 (x^2 + \hat{p}^2)$ and $N_3=N_3(x,\hat{p}) \leq c_3 |x|$ 
for $(x,\hat{p})\leq \delta_2 =\delta_2(B,C) \ll 1$ and 
positive constants $c_j =c_j(B,C)$ ($j=2,3$). Finally, all the $N_j$ are Lipschitz 
continuous with a uniform Lipschitz constant $L=L(B,C)$.

The mapping $T: \; (x_n,\hat{p}_{n}) \ra (x_{n+1},\hat{p}_{n+1})$ from $L^2([0, 2\pi])^2$ to 
$H^1([0, 2\pi])^2$ is given by
\ba
dx_{n+1}/dz + {C \over (B+C)^2} \hat{p}_{n+1} & = & -{C\over B+C} (- {\hat{p}_{n}\over B+C} 
 +\sin  {\hat{p}_{n}\over B+C}) \no \\
& + & {\eps \sin z \over B+C} + N_1(x_n,\hat{p}_{n}) \no \\
d\hat{p}_{n+1}/dz - B x_{n+1} & = & B (-x_n + \sin x_n) + \eps \cos (z) \no \\
& &  - N_2 (x_n,\hat{p}_{n}) 
+\eps N_3(x_n,\hat{p}_n)\, \sin z, \label{abc21}
\ea  
solution of which is ensured by Lemma \ref{lem1}. If $\eps$ is small enough, 
$T$ maps a small ball of radius $r=\sqrt{\eps}$ in $L^2([0, 2\pi])^2$ into a radius $O(\epsilon )$ ball in 
$H^1([0, 2\pi])^2$ which embeds continuously into $L^\infty([0, 2\pi])^2$. 
By Lipschitz continuity of the $N_j$ ($j=1,2,3$) 
and Lemma \ref{lem1}, we have for constants $c_4,c_5$, depending only on $(B,C)$, that
\ba
\| T (x_{n+1},\hat{p}_{n+1}) - T(x_n,\hat{p}_n))\| & \leq &  c_4 r^2 \|(x_{n},\hat{p}_{n}) - (x_{n-1},\hat{p}_{n-1})\| \no \\ 
& + & c_5 \eps  \|(x_{n},\hat{p}_{n}) - (x_{n-1},\hat{p}_{n-1})\| \no \\
& \leq & \beta \|(x_{n},\hat{p}_{n}) - (x_{n-1},\hat{p}_{n-1})\|, \label{abc22}
\ea
where $\beta =  c_4 r^2 +c_5 \eps \leq c_6 \eps < 1$. The mapping $T$ is a contraction for small enough 
$\eps =\eps (B,C)$, whose unique fixed point is a $2\pi $-periodic solution $(x,\hat{p})$ of
\be
dx/dz = \H_{\hat{p}}(x,\hat{p},z), \;\; d\hat{p}/dz = - \H_{x}(x,\hat{p},z). \label{abc23}
\ee
Recalling $p= C + B\pi/2+\hat{p}$ and using \eqref{abc7}, 
we obtain a $2\pi$ periodic solution $(x,y,p)(z)$ to (\ref{abc5})-(\ref{abc7}). 
Finally, $z=z(t)$ is found from
\be
dz/dt = B \cos (x(z)) + C\sin ( y(z))= B \cos (x(z)) + C \cos (\hat{y}(z)), 
\label{abc24}
\ee
with $\hat y=y-\pi/2$ and $z(0)=c_0$. Because $(x(z),\hat y(z))$ is $2\pi$ periodic in $z$ and small, $z(t)$ is globally defined 
and satisfies
\[
 \int_{c_0}^{z(t)} {dz' \over B \cos (x(z')) + C \cos (\hat{y}(z'))} = t.  \]
In particular, it follows that $z$ is asymptotically linear, with
\be
\lim_{t \ra \infty} \frac{z(t)}t = \left ( {1\over 2\pi}\int_{0}^{2\pi}\, {dz' \over B \cos (x(z')) + C \cos (\hat{y}(z'))} \right )^{-1}. 
\label{abc25}
\ee

To summarize, we proved
\begin{theorem}\label{persol}
There is a small positive number $A_0=A_0(B,C)$ such that for any $A \in[0, A_0]$ and any $z(0)\in\mathbb R$, 
there is a smooth solution 
$(x,y,z)(t)$ to the ABC flow system (\ref{abc1}) such that 
 $z$ is increasing in $t$, the limit $\lim_{t \ra \infty} z(t)/t$ exists and converges to $B+C$ as $A\to 0$, and $(x,y)$ is $2\pi$-periodic in $z$. 
Thus, the solution is a ballistic spiral orbit moving helically in the direction of the $z$ axis. 
\end{theorem}
\medskip

Obviously, a similar result holds for small $B$ (or $C$), with the spiral 
orbit moving helically along the $x$ (or $y$) axis.

\begin{figure}[h!]  
\begin{center}
\includegraphics[height=2in]{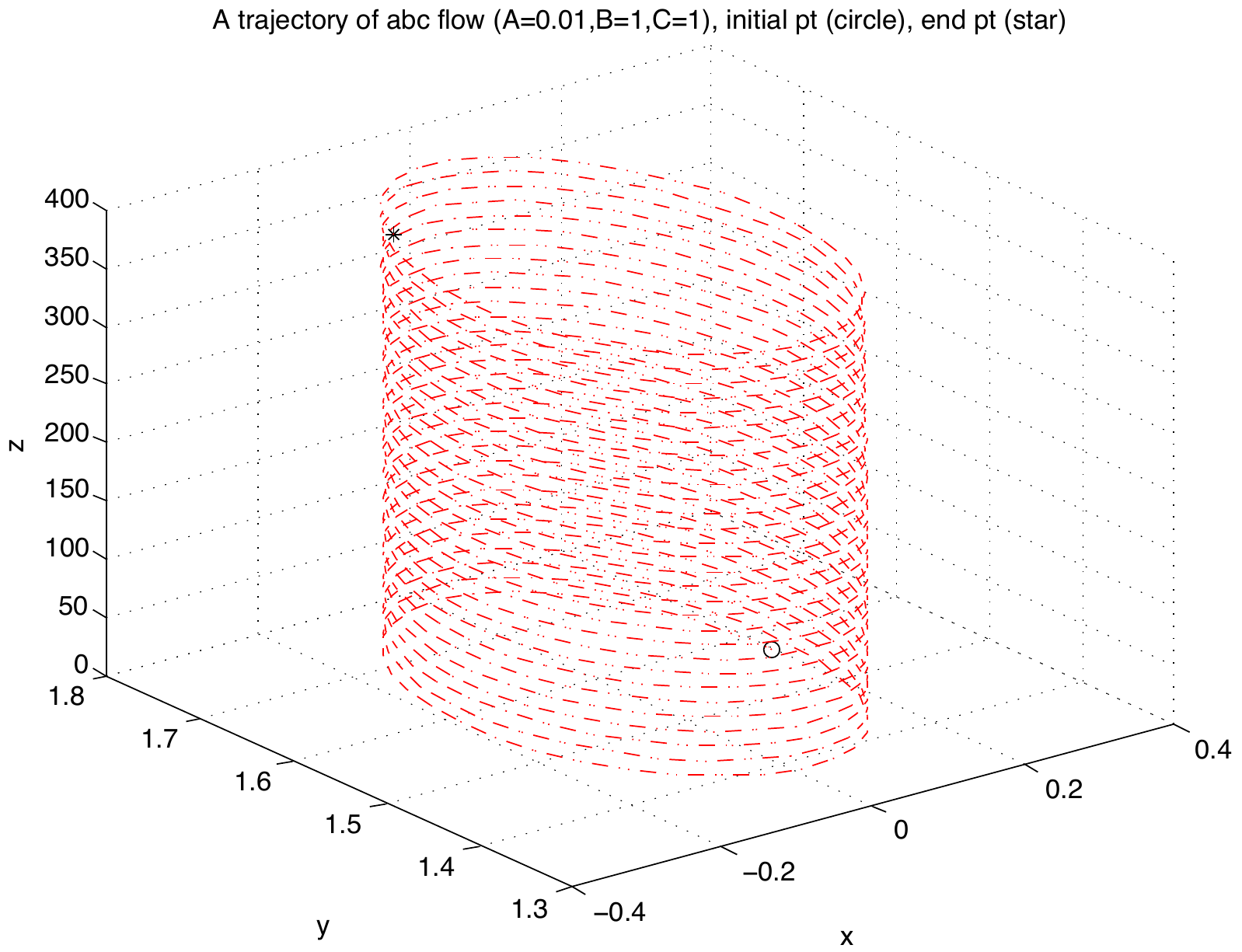}
\hskip .2in
\includegraphics[height=1.7in]{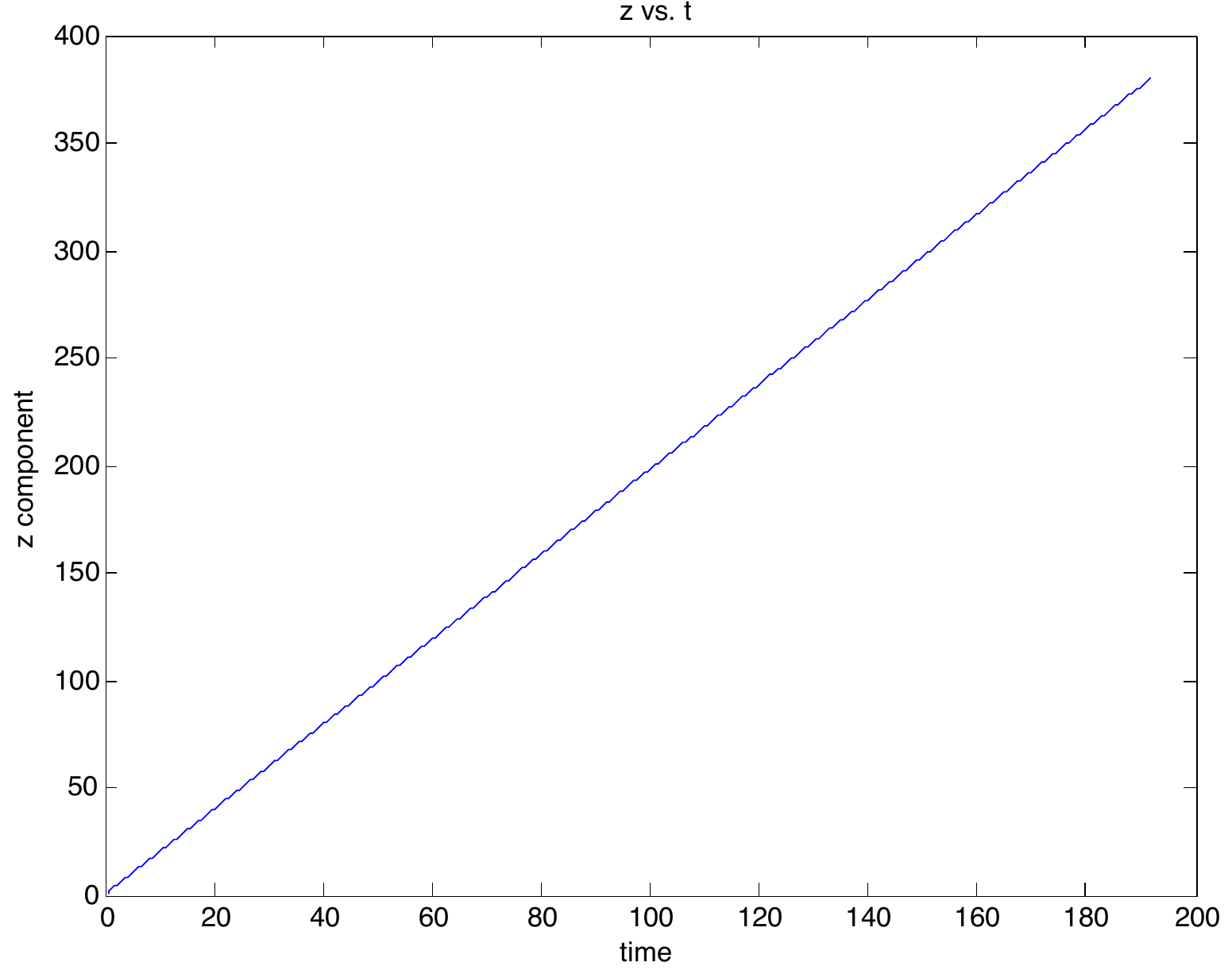}
\caption{Left: a trajectory of the ABC flow for $(A,B,C)=(0.01,1,1)$, with the initial point 
$(0.2,\pi/2,0)$ marked by a circle and the end point  marked by a star.  Right:  the $z$ component of trajectory.}
\label{KAM1}
\end{center}
\end{figure}

\begin{figure}[h!]
\begin{center}
\includegraphics[scale=0.32]{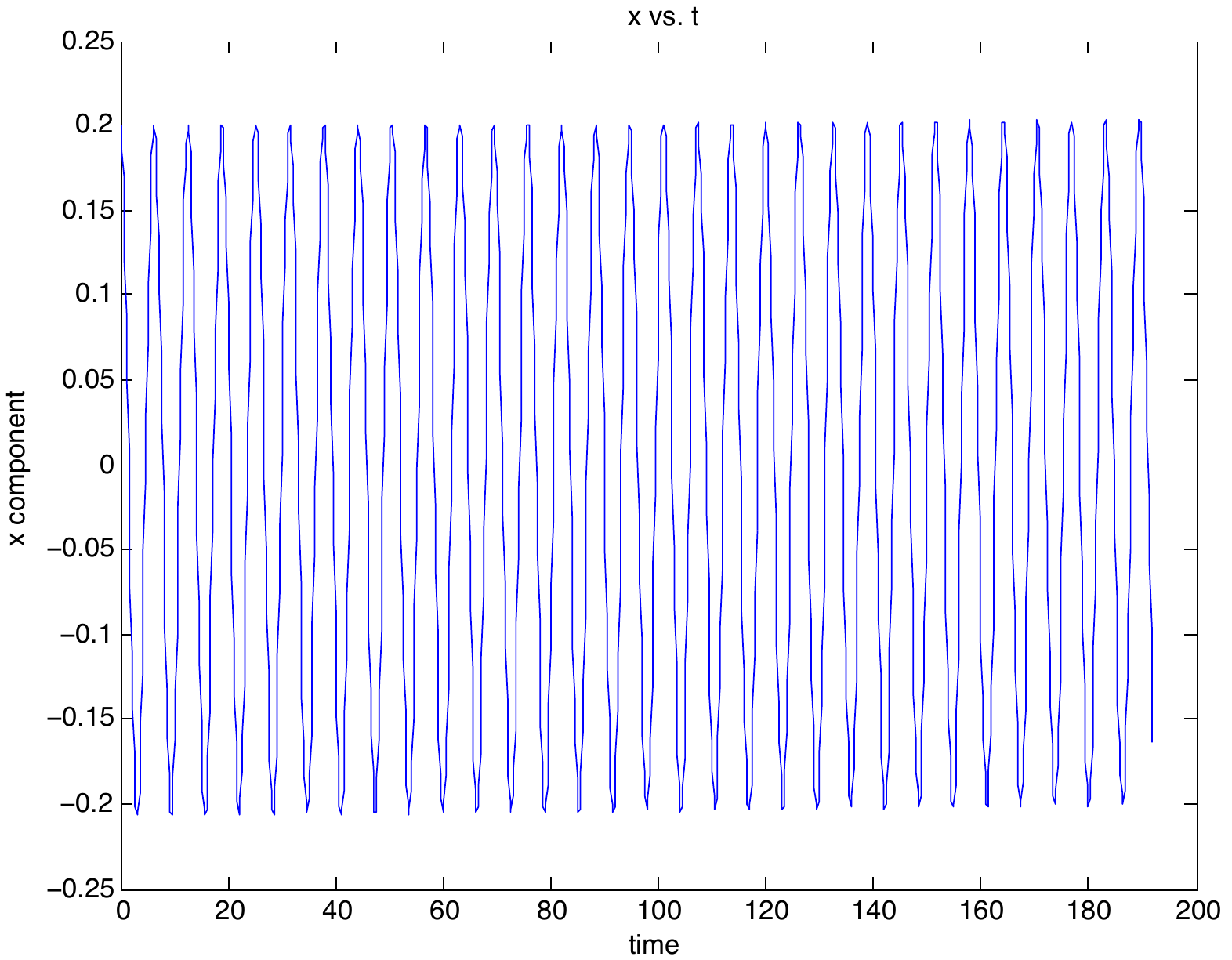}
\hskip .2in
\includegraphics[scale=0.32]{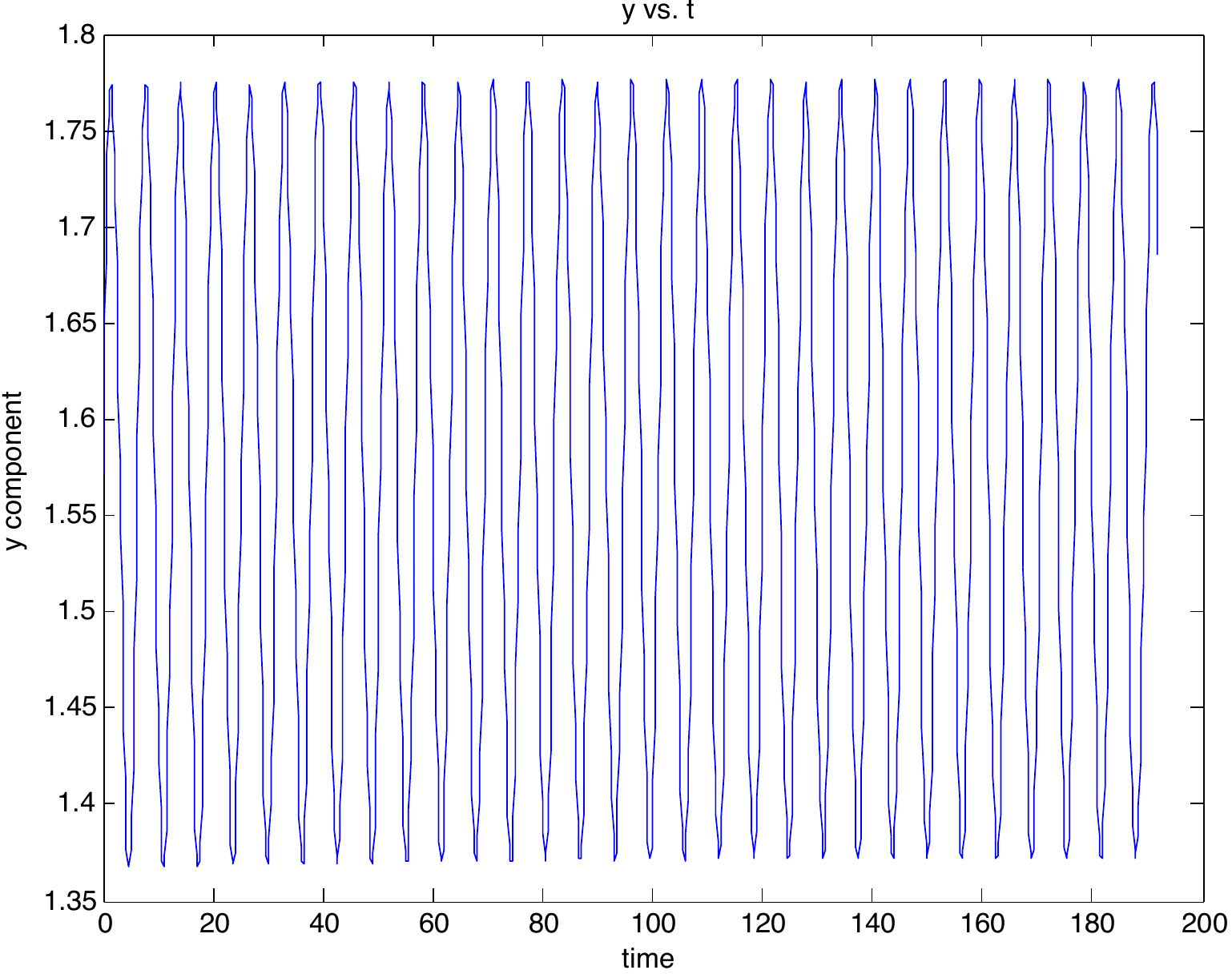}
\captionof{figure}{The $x$ (left) and $y$ (right) components of the trajectory in Figure \ref{KAM1}.}
\label{KAM3}
\end{center}
\end{figure}

We note that writing the ABC flow in the Hamiltonian form also facilitates  KAM-type analysis.  Let us introduce the usual action-angle variables  $(I, \phi)$ within the  cell in the $(x,y)$ plane centered at $(0,  {\pi\over 2})$.   Consider the case $B=C=1$ and $0<A=\epsilon \ll 1$.  The relation between $I$ and $H$ (i.e. $H=H(I)$) is given by $2\pi I=$ area of $\{ H\leq \cos x+\sin y \leq 2\}$,  so  $I\in  [0,\pi)$.  The ABC flow system (\ref{abc1})  can be written as the perturbed  action-angle-angle system
$$
\begin{cases}
\dot I=\epsilon F_0(I, \phi,  z)\\
\dot \phi=H'(I)+\epsilon F_1(I,  \phi, z)\\
\dot z=H(I)
\end{cases}
$$
for suitable smooth functions $F_0$ and $F_1$. According to  Theorem 5.1 in \cite{MW1994},  there exists a family of  perturbed invariant tori  parametrized by $\omega^{*}=H'(I^{*})$ for  $I^{*}\in J(\epsilon)\subset   [0,  \pi)$.  The measure of  $J(\epsilon)$ tends to $\pi$ as $\epsilon\to 0$.  Unlike the regular KAM theorem,  we do not know which invariant torus will survive after perturbation.  Also,  the frequency of a perturbed torus might not be the same as the unperturbed one.  For small $\epsilon$,   quasi-periodic orbits on invariant tori can also be parametrized by $z$ (i.e.  $I=I(z)$ and $\phi=\phi(z)$),  and they are usually not periodic.  If an orbit happens to be periodic in $z$, the period is close to $2\pi H(I)\over H'(I)$.  In particular,  near $H=2$ (or equivalently $I=0$),  we have that 
$$
I=2-H+C(2-H)^2+O(|2-H|^2)
$$
for some positive constant $C$.  Hence any periodic orbit  near the line $(0,  {\pi\over 2}, z)$  from KAM-type theorems has period $\approx 4\pi$ in $z$.

\begin{figure}[h!]  
\begin{center}
\includegraphics[height=2.5in]{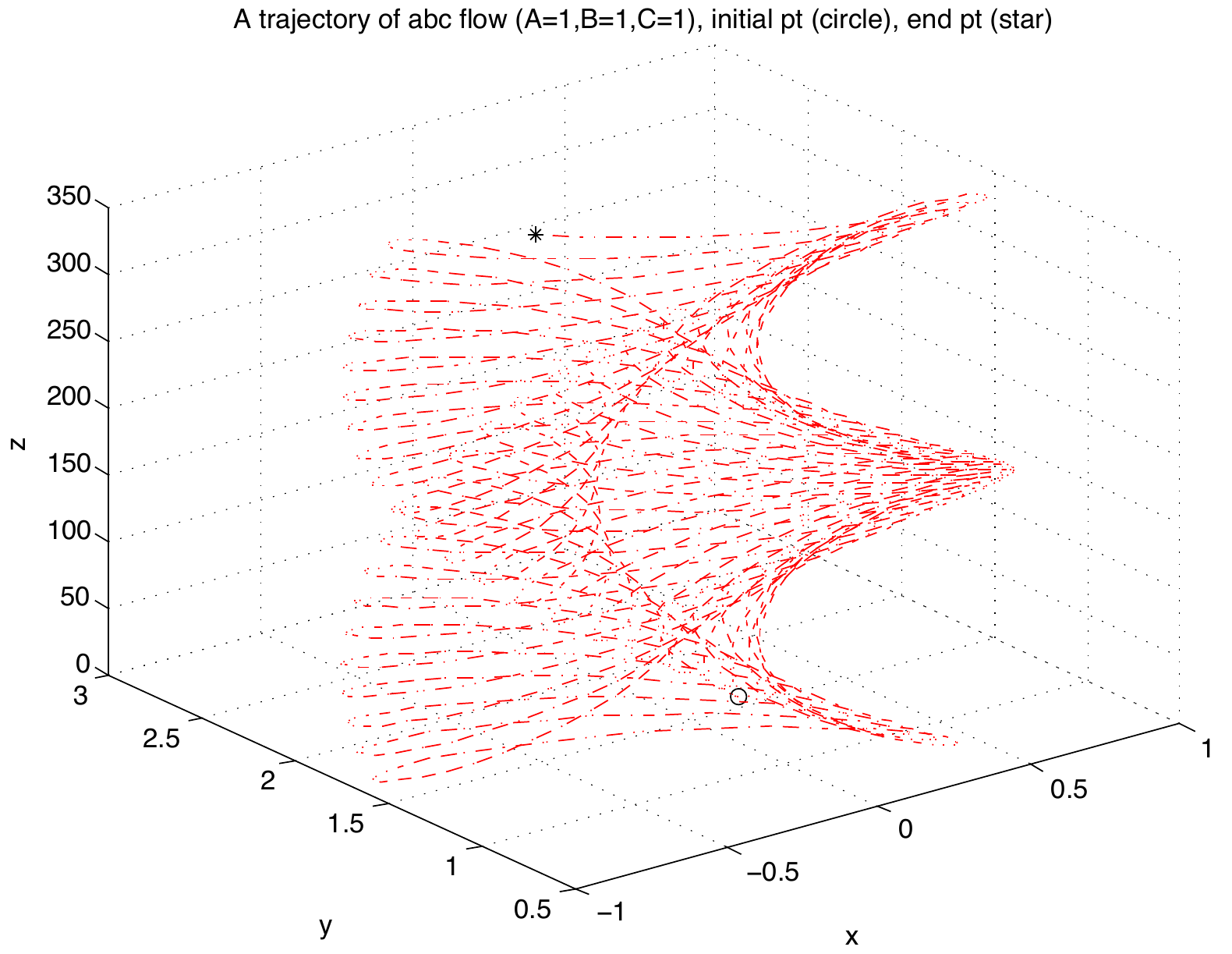}
\hskip .2in
\includegraphics[height=1.25in]{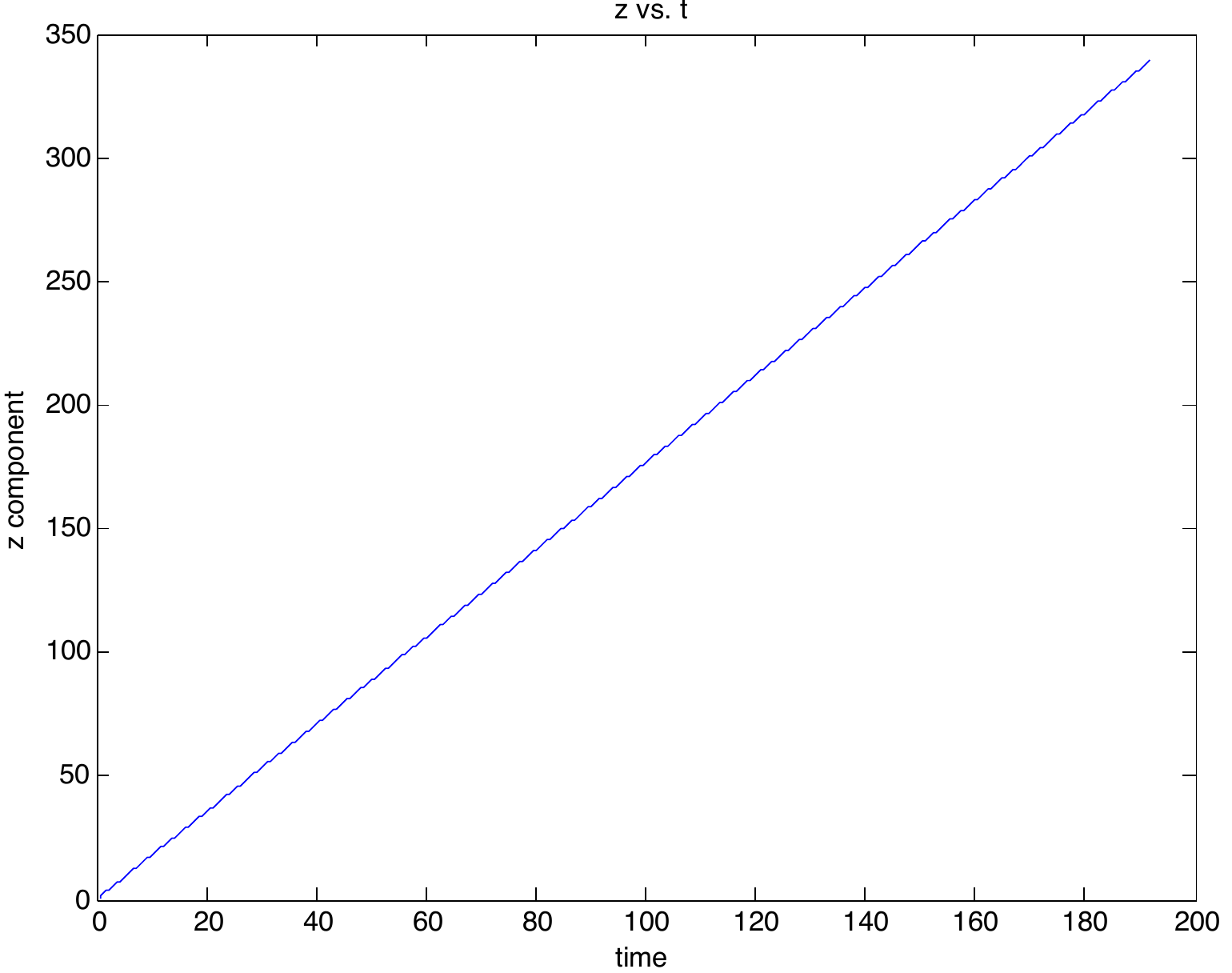}
\caption{Left: a trajectory of the ABC flow for $(A,B,C)=(1,1,1)$, with the initial point 
$(0.2,\pi/2,0)$ marked by a circle and the end point marked by star.  Right:  the $z$ component of the trajectory.}
\label{KAM4}
\end{center}
\end{figure}

\begin{figure}[h!]
\begin{center}
\includegraphics[height=1.5in]{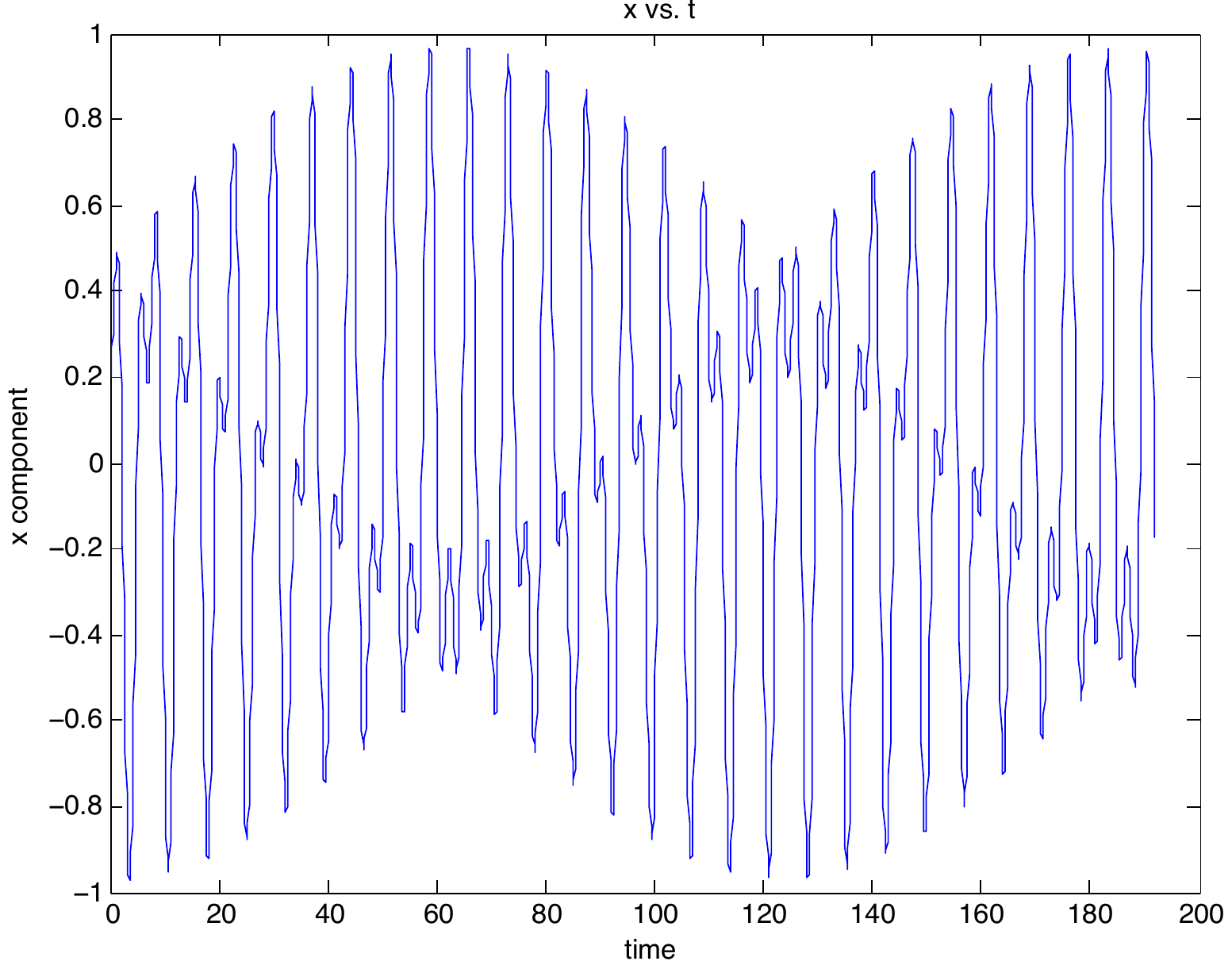}
\hskip .2in
\includegraphics[height=1.5in]{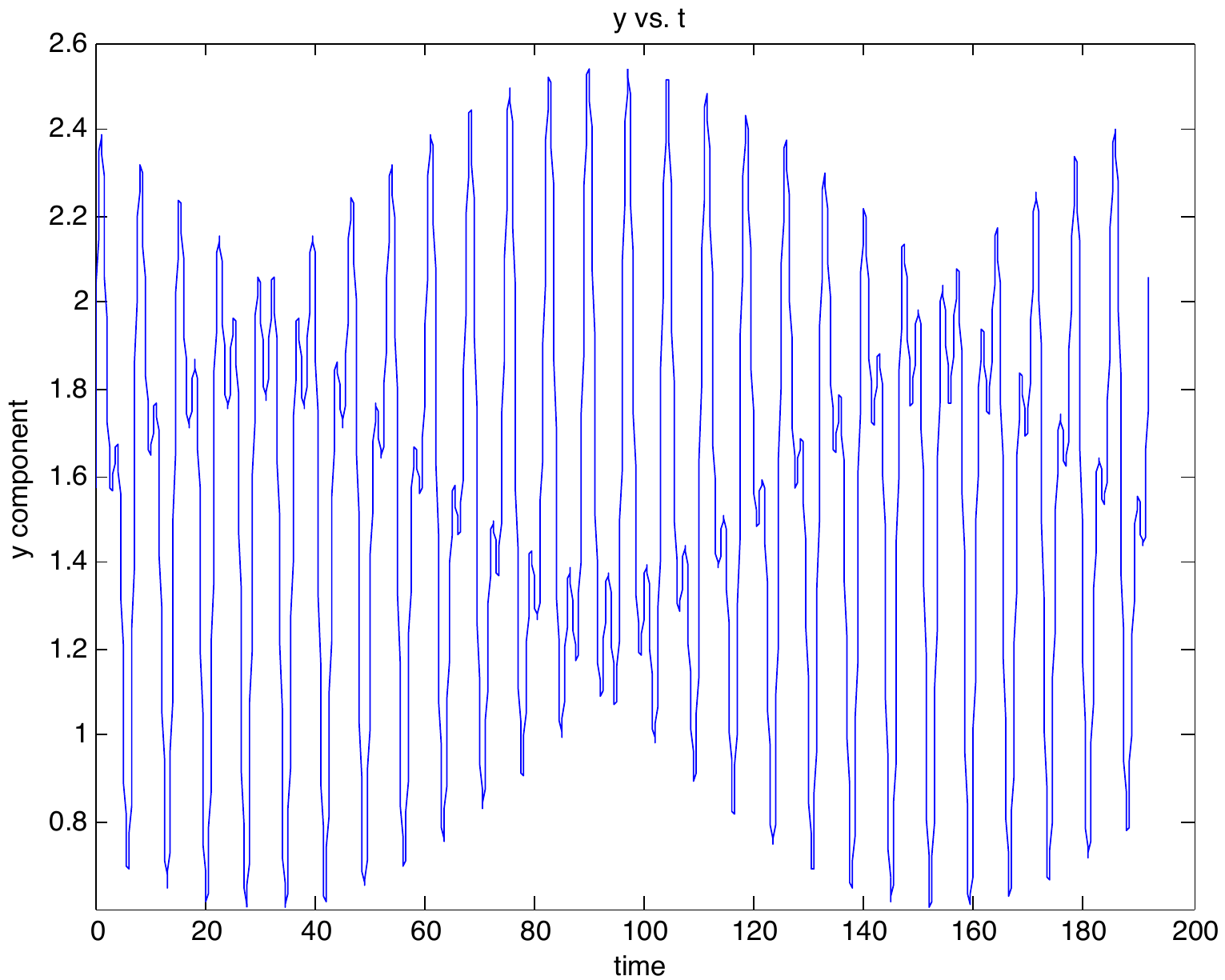}
\captionof{figure}{The $x$ (left) and $y$ (right) components of the trajectory in Figure~\ref{KAM4}.}
\label{KAM6}
\end{center}
\end{figure}

\newpage
\section{Edge orbits in the non-KAM region}
\label{sec.edgeflow}
\setcounter{equation}{0}

The results of the previous section hold for trajectories that remain inside a single cell in the $xy$-plane, where $H$ has the same sign.  Now we turn to the case in which trajectories cross cell boundaries, that is, the lines $H(x,y)=0$.  In this case, $z$ cannot be treated as a time variable.  In particular, \eqref{abc7} does not define $y$ uniquely, so the analysis of the previous section does not apply.  However, symmetries of the system do imply existence of spiral orbits such as that in Figure~\ref{xzspiral}, which we prove in the following subsection.  Afterwards we derive first order perturbation results and perform numerics to draw conclusions about  behavior of generic trajectories in the non-KAM region.

Throughout this section we assume $A=\epsilon > 0$ and $B=C=1$, so that
\begin{eqnarray}
x' &=& \cos y + \epsilon \sin z \nonumber
\\
y' & = & \sin x + \epsilon \cos z
\label{abceps}
\\
z' & = & \cos x + \sin y.
\nonumber
\end{eqnarray}
We denote solutions of this system $X(t)=(x(t),y(t),z(t))$.

%
%
\subsection{Existence of  ballistic edge orbits}
\label{sec.xygrowth}

The existence of ballistic edge orbits  follows from the following time-reversal symmetries of the system \eqref{abceps}:
\begin{eqnarray}
&& (t,x,y,z)\rightarrow (-t, -\pi - x,  - y,  z)
\label{sym1}
\\
&& (t,x,y,z)\rightarrow \left(- t,  \frac{\pi}{2} - y,  \frac{\pi}{2} - x, \frac{\pi}{2} - z \right)
\label{sym2}
\\
&& (t,x,y,z)\rightarrow (- t,  -x, y,  \pi -z)
\label{sym3}
\end{eqnarray}

\begin{theorem}
\label{zperiodic}
For all small enough $\epsilon>0$, the following hold.

(a)  There exists $T>0$ and four trajectories $X(t)$ of \eqref{abceps} that satisfy
\begin{eqnarray*} 
X(t+T) &=& X(t) + (2\pi, 2\pi,0)
\\
X(t+T) & = & X(t) - (2\pi, 2\pi,0)
\\
X(t+T) & = & X(t) + (2\pi, -2\pi,0)
\\
X(t+T) & = & X(t) - (2\pi,  -2\pi,0)
\end{eqnarray*}

(b)  There exists $T>0$ and four trajectories $X(t)$ of \eqref{abceps} that satisfy
\begin{eqnarray*} 
X(t+T) &=& X(t) + (2\pi, 0, 0)
\\
X(t+T) & = & X(t) - (2\pi, 0, 0)
\\
X(t+T) & = & X(t) + (0, 2\pi,0)
\\
X(t+T) & = & X(t) - (0, 2\pi,0)
\end{eqnarray*}

\end{theorem}

\begin{proof}
Let $X_a(t) = (x_a(t), y_a(t), z_a(t))$ be the solution of \eqref{abceps} with initial condition
\beq 
X_a(0) = \left(-\frac{\pi}{2}, 0, a\right).
\label{ainit}
\eeq
For $a\in[\frac\pi 6,\frac \pi 4)$, let $t_a > 0$ be the first time such  that $X_a(t_a)\in \partial D$, where $ D = R\times (0,\frac{\pi}{4})$, and $R$ is the open rectangle in the $xy$-plane with vertices $(0,-\frac{\pi}{2}), \ (\frac{\pi}{2},0)$,  $(-\frac{\pi}{2},\pi), \ (-\pi,\frac{\pi}{2})$ (see Figure~\ref{Rboundary}).  If no such time exists, let $t_a=\infty$.  Note that for $t\in(0,t_a)$ we have $(x_a(t), y_a(t))\in R$, so $z_a'(t) > 0$.  This, together with the vector $(\cos y,\sin x)$ being tangential to $\partial R \backslash \{x+y=\frac{\pi}{2}\}$ and $\cos z > |\sin z|$ for $z\in[\frac{\pi}{6},\frac{\pi}{4})$, show that either
\[ z_a(t_a) = \frac{\pi}{4} \quad \mbox{or} \quad
y_a(t_a) + x_a(t_a) = \frac{\pi}{2} 
\quad \mbox{or} \quad
y_a(t_a) - x_a(t_a) = \frac{3\pi}{2}
\quad \mbox{or} \quad
t_a = \infty.
\]

Now let $\epsilon>0$ be sufficiently small.  Then  $\{(x_a(t),y_a(t))\}_{t\in[0,t_a]}$  stays close to the part of $\partial R$ between $(-\frac \pi 2,0), (0,-\frac \pi 2), (\frac \pi 2,0)$ (let us call it $V$) because $(\cos y, \sin x)$ is tangential to $V$, continuous, and non-zero near $V$ except at $(0,-\frac\pi 2)$.  More specifically, the trajectory $\{(x_a(t),y_a(t))\}_{t\in[0,t_a]}$ would have to hit the line $\{y + x = \frac{\pi}{2}\}$ before it can depart from $V$, and we in fact also have
\beq \label{az2}
\lim_{\epsilon\to 0} \sup_{\substack{a\in[\pi/6,\pi/4]\,\&\\ t\in[0,t_a]}} {\rm dist}((x_a(t),y_a(t)),V)=0.
\eeq
This means that the third alternative above cannot happen.  Since $\cos y + \epsilon \sin z  > \frac\epsilon 2$ for $(x,y)$ near $V$ (and inside $R$) and $z\in [\frac\pi 6,\frac \pi 4)$, we have $x_a'(t)>\frac\epsilon 2$ for $t\in(0,t_a)$. Hence $t_a<\frac {2\pi}\epsilon$ due to $x_a(t_a)-x_a(0)\le \frac \pi 2-(-\frac \pi 2)= \pi$, so one of the first two alternatives must happen.

It is also easy to see that $z_a(t_a)=\frac{\pi}{4}$ when $a$ is close enough (depending on $\epsilon$) to $\frac{\pi}{4}$ because $\cos x + \sin y>0$ in $R$.  Moreover, we also have $z_a(t_a)<\frac{\pi}{4}$ when $a=\frac \pi 6$ and $\epsilon>0$ is small enough.  This holds because $x_a'(t)\ge \cos y_a(t)$ for $t\in[0,t_a)$ and 
\begin{eqnarray*}
\lim_{\substack{(x,y)\in R\;\&\\ {\rm dist}((x,y),V)\to 0}} \frac {\cos x + \sin y}{\cos y} &=& \lim_{\substack{(x,y)\in R\,\;\&\\ {\rm dist}((x,y),V)\to 0}} \frac {2 \sin \frac{y+x+\pi/2}2 \sin \frac{y-x+\pi/2}2} {\sin(y+\pi/2)} 
\\
&=&0
\end{eqnarray*}
(the latter due to $\max\{\frac{y+x+\pi/2}2, \frac{y-x+\pi/2}2 \}\le y+\pi/2$ for $(x,y)\in R$),
 which together with \eqref{az2} show that
 \[
 \lim_{\epsilon\to 0} \sup_{t\in(0,t_a)} \frac{z_a'(t)}{x_a'(t)}=0.
 \]
 Since $x_a(t_a)-x_a(0)\le \pi$, it follows that $z_a(t_a)< \frac\pi 4$ for $a=\frac\pi 6$ and any small enough $\epsilon>0$.

We thus obtain that for any small $\epsilon>0$, there is $a\in(\frac{\pi}{6},\frac{\pi}{4})$ such that 
\beq 
z_a(t_a) = \frac{\pi}{4} \quad \mbox{ and } \quad 
x_a(t_a) + y_a(t_a) = \frac{\pi}{2}
\label{acrit}
\eeq
(this also uses that the two sets of $a\in [\frac{\pi}{6},\frac{\pi}{4})$ where one of these claims holds but not the other are both relatively open in $[\frac{\pi}{6},\frac{\pi}{4})$, which is due to the vector field on the right-hand side of \eqref{abceps} being transversal to $\partial D$ at $R\times\{\frac\pi 4\}$ and at the points of $L\times(\frac\pi 6,\frac\pi 4)$ that lie near $V\times(\frac\pi 6,\frac\pi 4)$, with $L$ the open segment connecting $(\frac{\pi}{2},0)$ and $(-\frac{\pi}{2},\pi)$ in the $xy$-plane).   Figures~\ref{Rboundary} and \ref{Dintersect} illustrate this for $\epsilon=0.1$ (which is too large for the critical $a$ to be greater than $\frac\pi 6$).

\begin{figure}[h!]
\begin{center}
          \includegraphics*[width=5in]{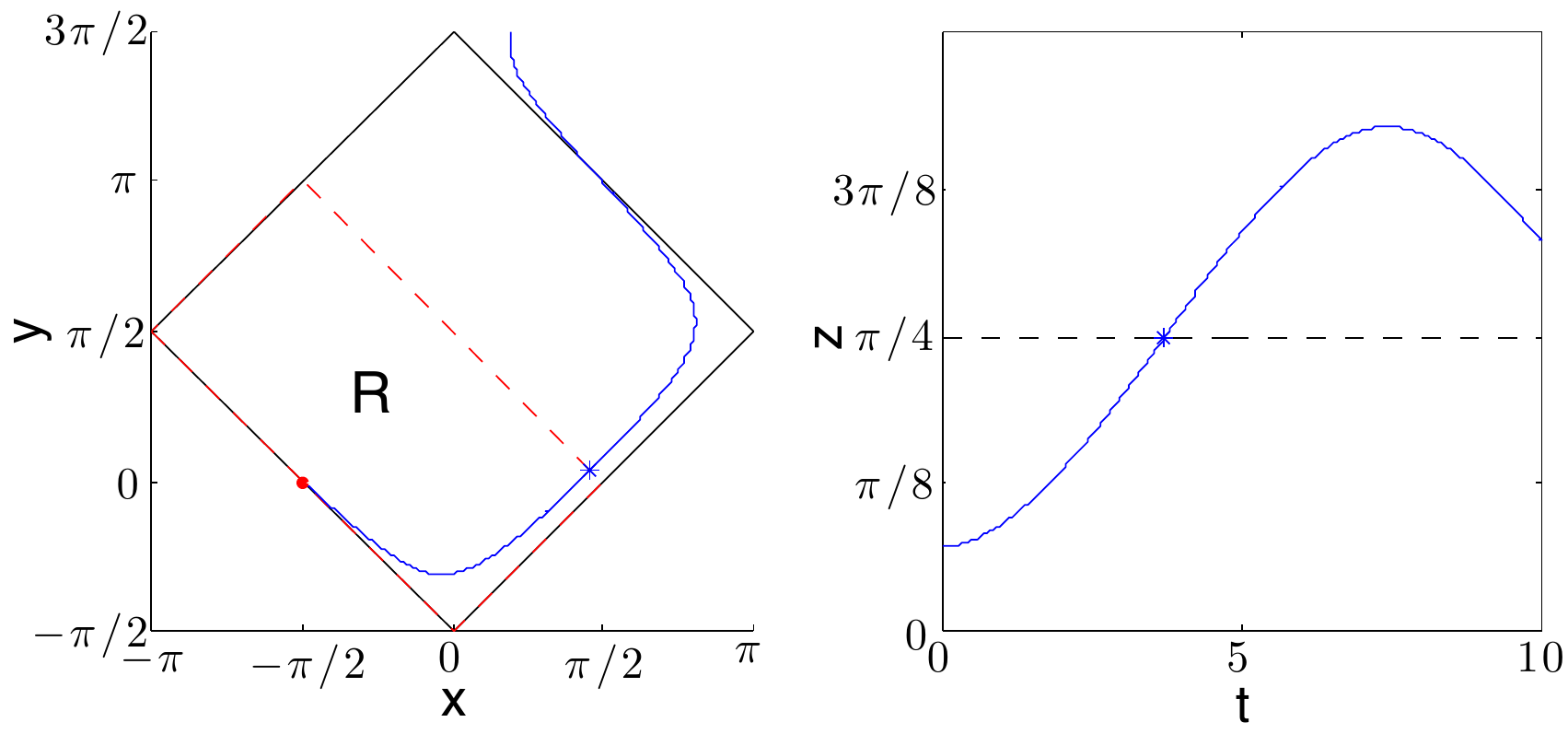}
          
\end{center}
\caption{The trajectory for \eqref{abceps} with $\epsilon =.1$ and  $X(0)=(-\frac{\pi}{2},0,0.2254)$, the star marking the point where it hits the intersection of the planes $\{z=\frac \pi 4\}$ and $\{x+y=\frac \pi 2\}$.  The rectangle $R$ is bounded by dashed red lines.}
\label{Rboundary}
\end{figure}

\begin{figure}[h!]
\begin{center}
          \includegraphics*[width=3.5in]{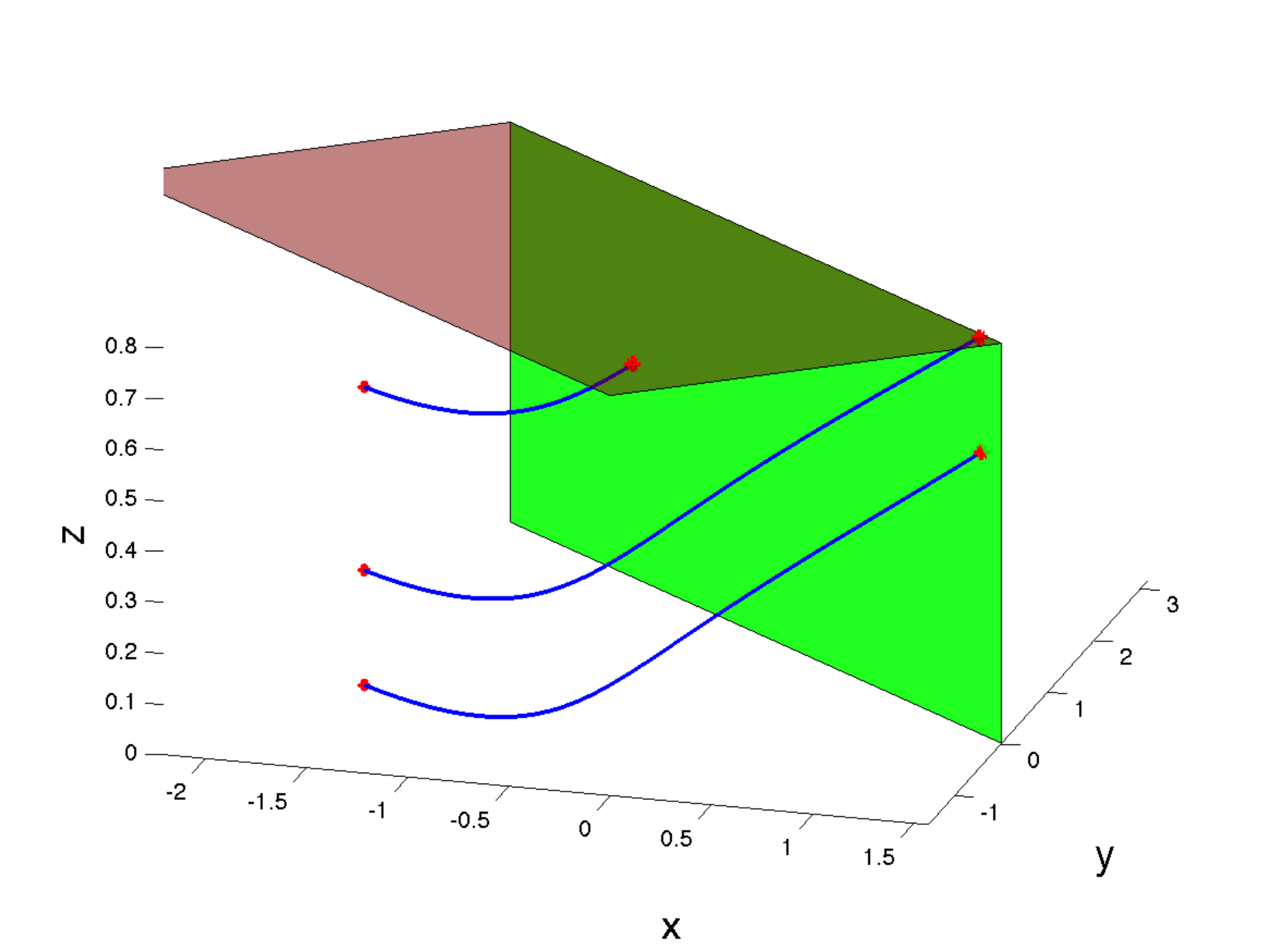}
          
\end{center}
\caption{Trajectories $X_a(t)$ with $X_a(0)=(-\frac \pi 2,0,a)$ and $a=0, 0.2254, 0.5854$ (again $\epsilon=.1$).   The middle one hits the intersection of $\{z=\frac \pi 4\}$ and $\{x+y=\frac \pi 2\}$.
}
\label{Dintersect}
\end{figure}

For any $a$,  the symmetry \eqref{sym1} yields
\[ (x_a(t), y_a(t), z_a(t)) = (-\pi - x_a(-t), -y_a(-t), z_a(-t)).
\]
For  $a$ satisfying \eqref{acrit}, symmetry \eqref{sym2} also yields 
\[ (x_a(t), y_a(t), z_a(t)) = \left(\frac{\pi}{2} - y_a(2t_a-t),\frac{\pi}{2} -x_a(2t_a-t), \frac{\pi}{2} - z_a(2t_a-t)\right).
\]
Thus
\[ X_a(-t_a) = \left( - \pi - x_a(t_a), - y_a(t_a), \frac{\pi}{4}\right),
\]
and then
\[ X_a(3t_a) = \left( \frac{\pi}{2} + y_a(t_a), \frac{3\pi}{2} + x_a(t_a), \frac{\pi}{4}\right)
=X_a(-t_a) + (2\pi, 2\pi, 0),
\]
where we used  $x_a(t_a)+y_a(t_a) = \frac{\pi}{2}$ in the last equality.
Hence for all $t\in\mathbb R$,
\beq 
X_a(t+4t_a) = X_a(t_a) + (2\pi, 2\pi, 0).
\eeq
Finally, 
\[
\tilde X_a(t)= \left(\frac\pi 2- y_a(t), \frac\pi 2+x_a(t),  z_a(t)-\frac\pi 2 \right)
\]
is a trajectory of \eqref{abceps} satisfying  $X(t+4t_a)=X(t)+(-2\pi,2\pi,0)$, and $X_a(-t)-(\pi,\pi,\pi)$ and $\tilde X_a(-t)-(\pi,\pi,\pi)$ are the remaining two trajectories from (a).

The proof of (b) is identical, this time considering $a\in[\frac \pi 6,\frac \pi 2)$ and letting $t_a>0$ be the first time such that $X_a(t_a) \in \partial D$, where now $D=R\times (0, \frac{\pi}{2})$ and $R$ is the triangle in the $xy$-plane with vertices $(0,-\frac{\pi}{2}), (0, \frac{3\pi}{2}), (-\pi, \frac{\pi}{2})$.  As in (a), we now obtain $a\in(\frac{\pi}{6}, \frac{\pi}{2})$ such that 
\beq
z_a(t_a) = \frac{\pi}{2} \quad \mbox{ and } \quad
x_a(t_a) = 0.
\label{acrit2}
\eeq
Then symmetries \eqref{sym1} and  \eqref{sym3}, together with \eqref{acrit2}, yield
\[ 
X_a(3t_a) = \left( \pi, - y_a(t_a), \frac{\pi}{2}\right) = X_a(-t_a) + (2\pi, 0, 0),
\]
and the rest follows as in (a).
$\square$
\end{proof}

\medskip 

\begin{remark}
Note that since $(\arcsin \frac \eps{\sqrt 2}, \arcsin \frac \eps{\sqrt 2} - \frac\pi 2,\frac{5\pi} 4)$ is a stationary point of \eqref{abceps}, we have proved that for any small $\epsilon>0$ and any $\alpha,\beta\in\{-2\pi,0,2\pi\}$, the system \eqref{abceps} has a solution satisfying $X(t+T)=X(t)+(\alpha,\beta,0)$ for some $T>0$ and each $t\in\mathbb R$.
Figure~\ref{xyperiodic} shows two such solutions, with $X(t+T) = X(t)+(2\pi, 2\pi,0)$ and $X(t+T)=X(t) + (2\pi, 0,0)$.
We conjecture  that 
such solutions exist for any $\epsilon >0$.  As this paper is mainly concerned with the near-integrable case of small $\epsilon$, we  will investigate large $\epsilon$ in a future work.
\end{remark}

\begin{figure}[h!]
\begin{center}
          \includegraphics*[width=2.7in]{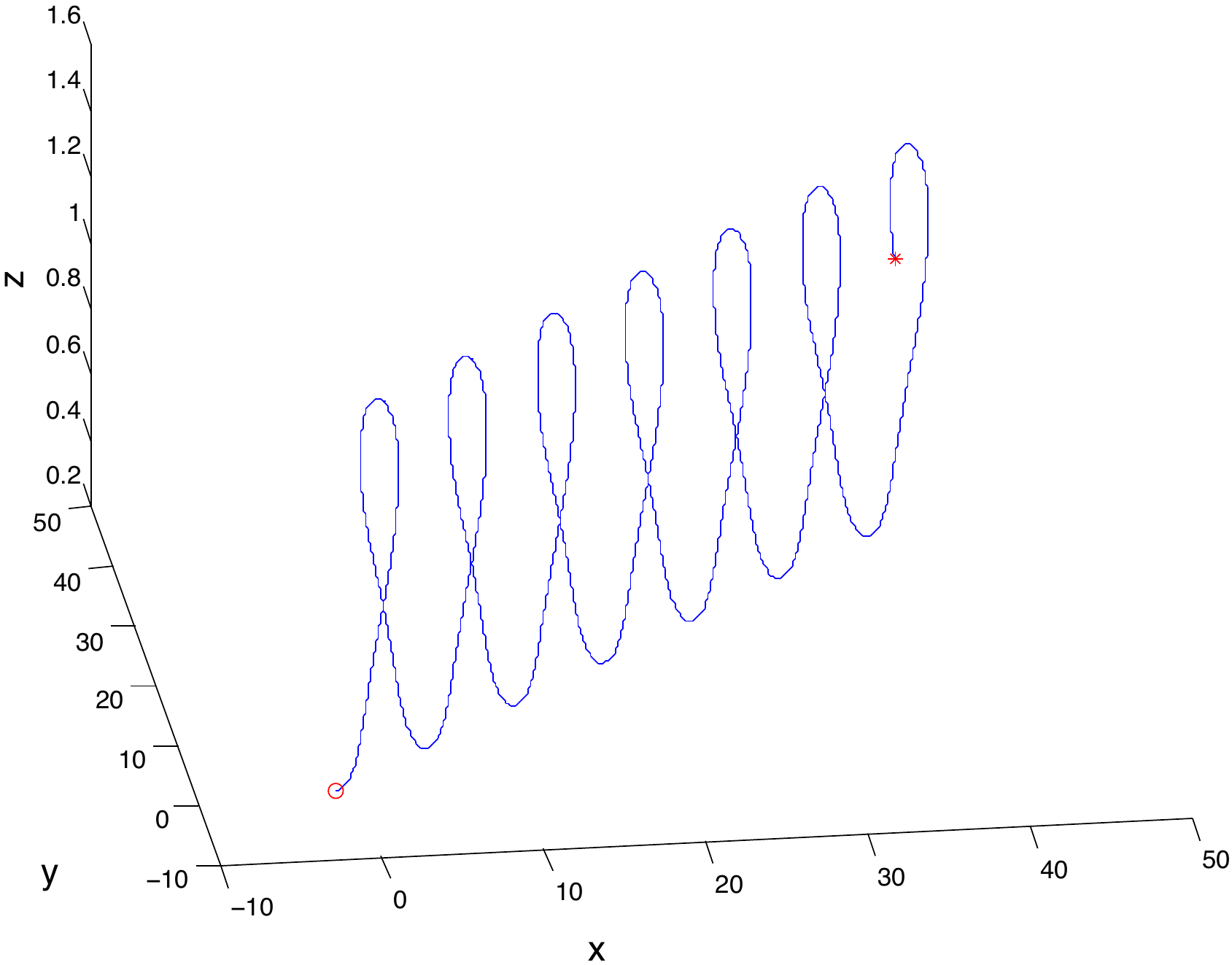}
          \includegraphics*[width=2.7in]{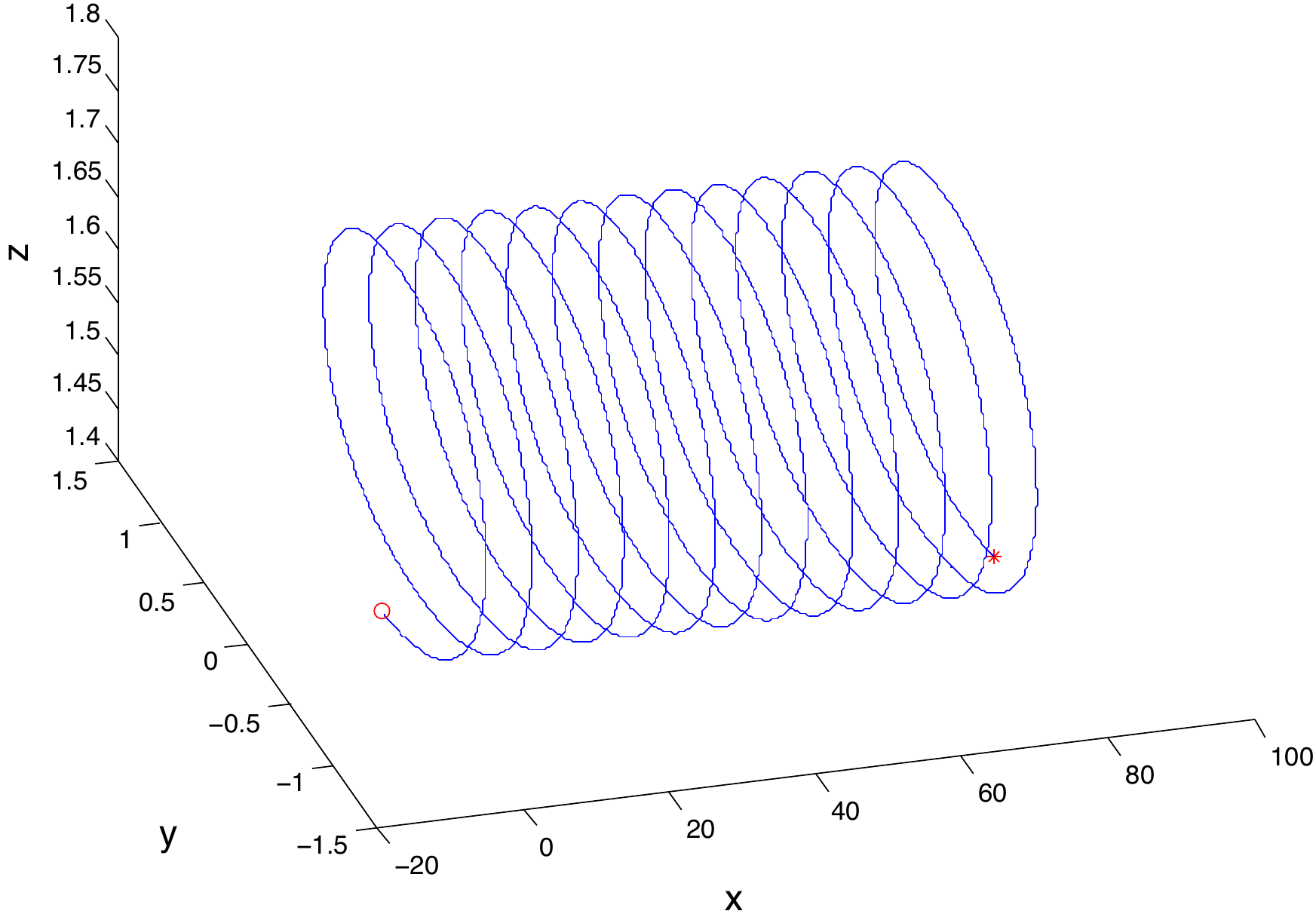}
          
\end{center}
\caption{Trajectories for \eqref{abceps} with $\epsilon=.1$ and $X(0)=(-\frac \pi 2,0,a)$, where  $a=0.2254, 1.4148$.}
\label{xyperiodic}
\end{figure}

\begin{figure}[h!]
\begin{center}
          \includegraphics*[width=2.5in]{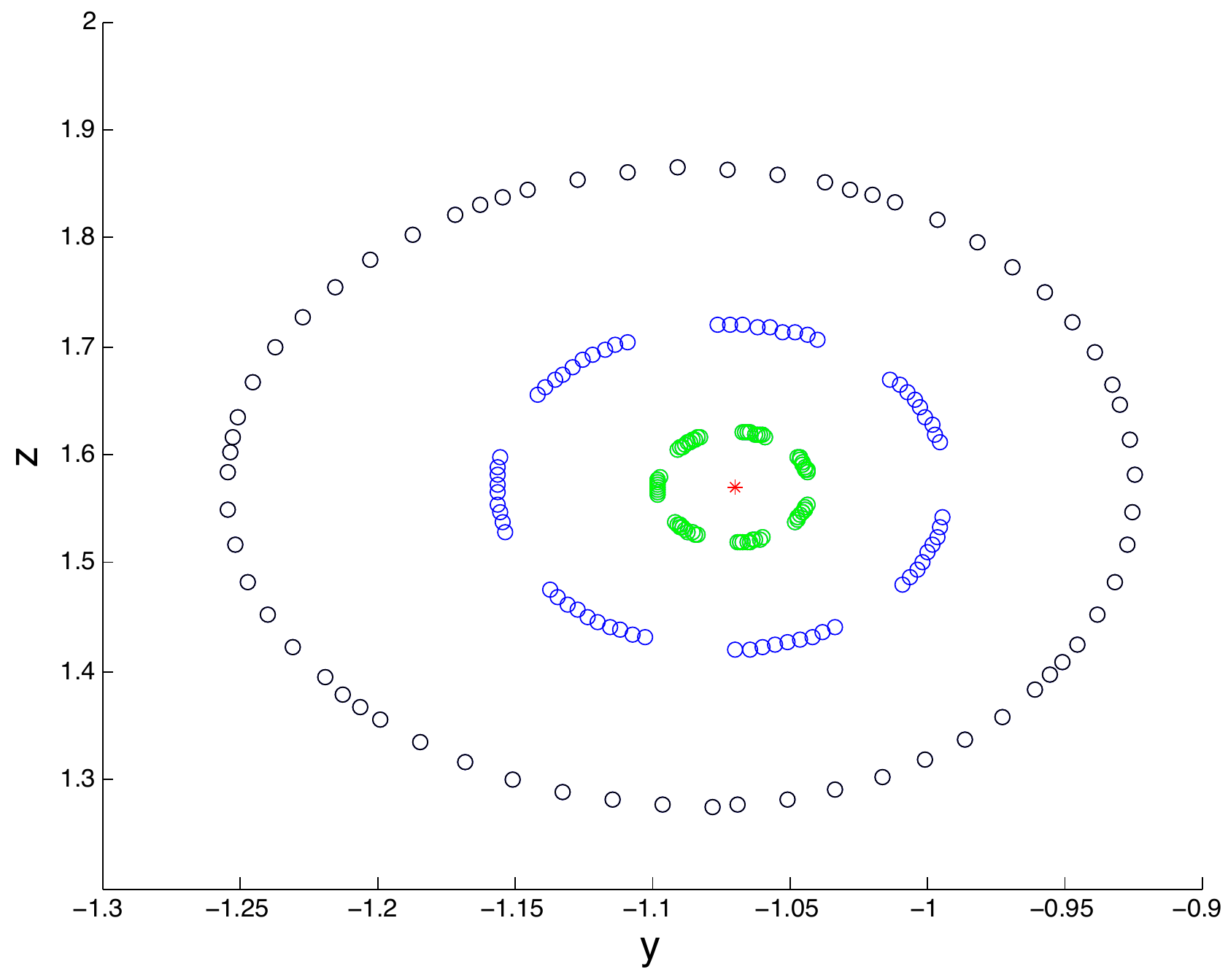}
\end{center}
\caption{Poincar\' e sections for starting points with  $z(0) = a_c + .05, .15, .3$ are marked in green, blue, and black, respectively.}
\label{poincare}
\end{figure}

\begin{remark}
Numerics suggest that the $z$-periodic solutions of Theorem~\ref{zperiodic} are neutrally stable.  In Figure~\ref{poincare} we show Poincar\'{e} sections at $x=0\,\,{ \rm mod}\,\, 2\pi$.  The $(y,z)$ coordinates are plotted at these sections for trajectories starting near the $z$-periodic trajectory crossing through $(-\frac \pi 2,0,a_c)$, where $a_c$ is the value for which $X(t+T)=X(t)+(2\pi,0,0)$.  Near the fixed point of the Poincar\'{e} map, points appear to be mapped onto closed curves surrounding the fixed point.
\end{remark}

%
%
\subsection{Perturbation analysis}
\label{sec.perturb}

A standard perturbation analysis can provide further information about orbits close to the boundaries of the cells, that is, the lines $H(x,y)=0$.  

The heteroclinic orbits  of the unperturbed $\epsilon = 0$ system can be calculated by elementary means.  For example, consider the cycle connecting the fixed points $(0,-\pi/2), (\pi,\pi/2), (0,3\pi/2)$, $(-\pi,\pi/2)$ (see Figure~\ref{heteroclinic}).  
\begin{figure}[h!]
\begin{center}
          \includegraphics*[width=2.5in]{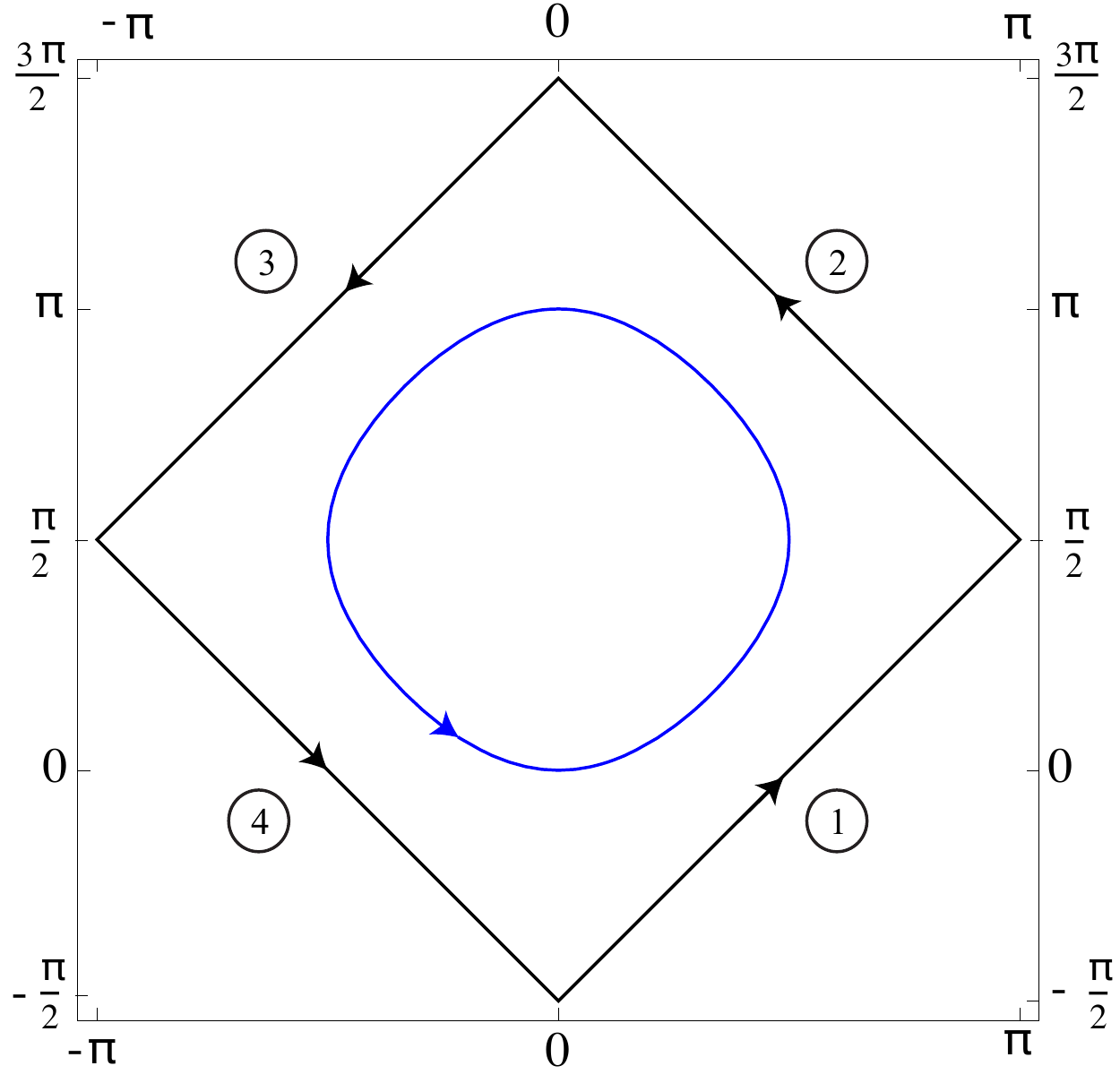}
\end{center}
\caption{A heteroclinic cycle in the conservative $\epsilon=0$ system.}
\label{heteroclinic}
\end{figure}

We label the heteroclinic orbits from 1 to 4 counterclockwise.  On orbit 1, $y=x-\pi/2$, so $x'=\sin x$.  This has the solution $x(t) = \gd(t)+\pi/2$, where $\gd(t)$ is the Gudermannian function 
\[
\gd(t) = 2\tan^{-1}\left(\tanh\left(\frac{t}{2}\right)\right).     
\]
Thus, the heteroclinic orbit 1 is
\[ (x(t),y(t)) = \left(\gd(t) + \frac{\pi}{2}, \gd(t)\right)
\]
The other heteroclinic orbits can be found similarly.

\begin{remark}
The above heteroclinic orbits exist for each $z$ when $\epsilon=0$, but  there are special values of $z$ for which straight-line orbits that are confined to the boundaries of the cells exist for each $\epsilon$.  Specifically, this happens when $(\sin z, \cos z)$ is parallel to a boundary of the cells.  For instance, if $\hat{x}(t)$ solves $\hat{x}'(t) = \sin \hat{x} + \epsilon/\sqrt{2}$, then
\beq
(x(t), y(t), z(t)) = \left( \hat{x}(t), \hat{x}(t) - \frac{\pi}{2}, \frac{\pi}{4}\right)
\label{special}
\eeq
solves \eqref{abceps}.  The same is true when $z=5\pi/4$ and $\hat{x}$ solves $\hat{x}'(t) = \sin \hat{x} - \epsilon/\sqrt{2}$, and similar solutions parallel to $(1,-1)$ exist
for $z=3\pi/4$ and $z=7\pi/4$.
\end{remark}

\medskip
Consider a heteroclinic orbit $(x_0(t), y_0(t), z_0)$ of the unperturbed system. 
We assume that $\epsilon >0$ is small and expand in powers of $\epsilon$:
\[ x = x_0 + \epsilon x_1 + \epsilon^2 x_2 + \cdots 
\]
and similarly for $y$ and $z$.
Substituting this into \eqref{abc1}
 and collecting  terms with the same powers of $\epsilon$ yields
\begin{eqnarray}
x_1 ' & = & -\sin(y_0) y_1 + \sin z_0  \label{x1}
\\
y_1 ' & = & \cos(x_0) x_1 + \cos z_0 \label{y1}
\\
z_1' &=& -\sin(x_0) x_1 + \cos(y_0) y_1. \label{z1}
\end{eqnarray}
Since $H(x_0,y_0)=0$, we have $-\sin y_0 =  \cos x_0$, so we can add and subtract \eqref{x1} and \eqref{y1} to get the decoupled system in the tangential and orthogonal directions to the unperturbed flow, which have easily obtainable closed form solutions.
For example, if we take $(x_0(t), y_0(t))$ to be orbit~4 from Figure~\ref{heteroclinic} with $(x_0(0), y_0(0))=(-\pi/2,0)$, then $\cos x_0(t) = \tanh t$, and we obtain 
\begin{eqnarray}
x_1 (t)+ y_1(t) & = & c_1\, \mbox{cosh}(t) + 
\sqrt{2}\sin(z_0+\pi/4) \cosh(t)  \gd(t)  
\label{xy1}
\\
x_1 (t)- y_1(t) &=& c_2 \, \sech (t) + 
\sqrt{2} \sin(z_0 - \pi/4) \tanh(t).
\label{xy2}
\end{eqnarray}
(Then $c_1$ and $c_2$ are determined from the initial condition for $x_1,y_1$.  For instance, taking $x_1(0)= y_1(0)=0$ 
yields $c_1=c_2=0$.)

Solutions obtained in this way have good agreement with numerical solutions until the trajectory has traversed approximately $1/4$ of the boundary of the cell (see Figure~\ref{fourtrajectories}), which takes a long time if $\epsilon$ is small because the trajectory comes close to the stationary point $(0,-\pi/2)$.
From (\ref{xy1})-(\ref{xy2}) we see that $x_1 + y_1$ grows without bound positive or negative, depending on the sign of $\sin(z_0+\pi/4)$, while $x_1- y_1$ asymptotes at $\sqrt{2}\sin(z_0-\pi/4)$.  Thus, for small $\epsilon > 0$, if $z_0$ is in
$(-\pi/4, \pi/4)$, $(\pi/4, 3\pi/4)$, $(3\pi/4, 5\pi/4)$, or $(5\pi/4, 7\pi/4)$, but not too close to the endpoints of these intervals, then the trajectory will have entered the cell marked in Figure \ref{fourtrajectories} by a, b, c, d,
respectively, by the time it traversed $1/4$ of the boundary of the cell.

\begin{figure}[h!]
\begin{center}
          \includegraphics*[width=4in]{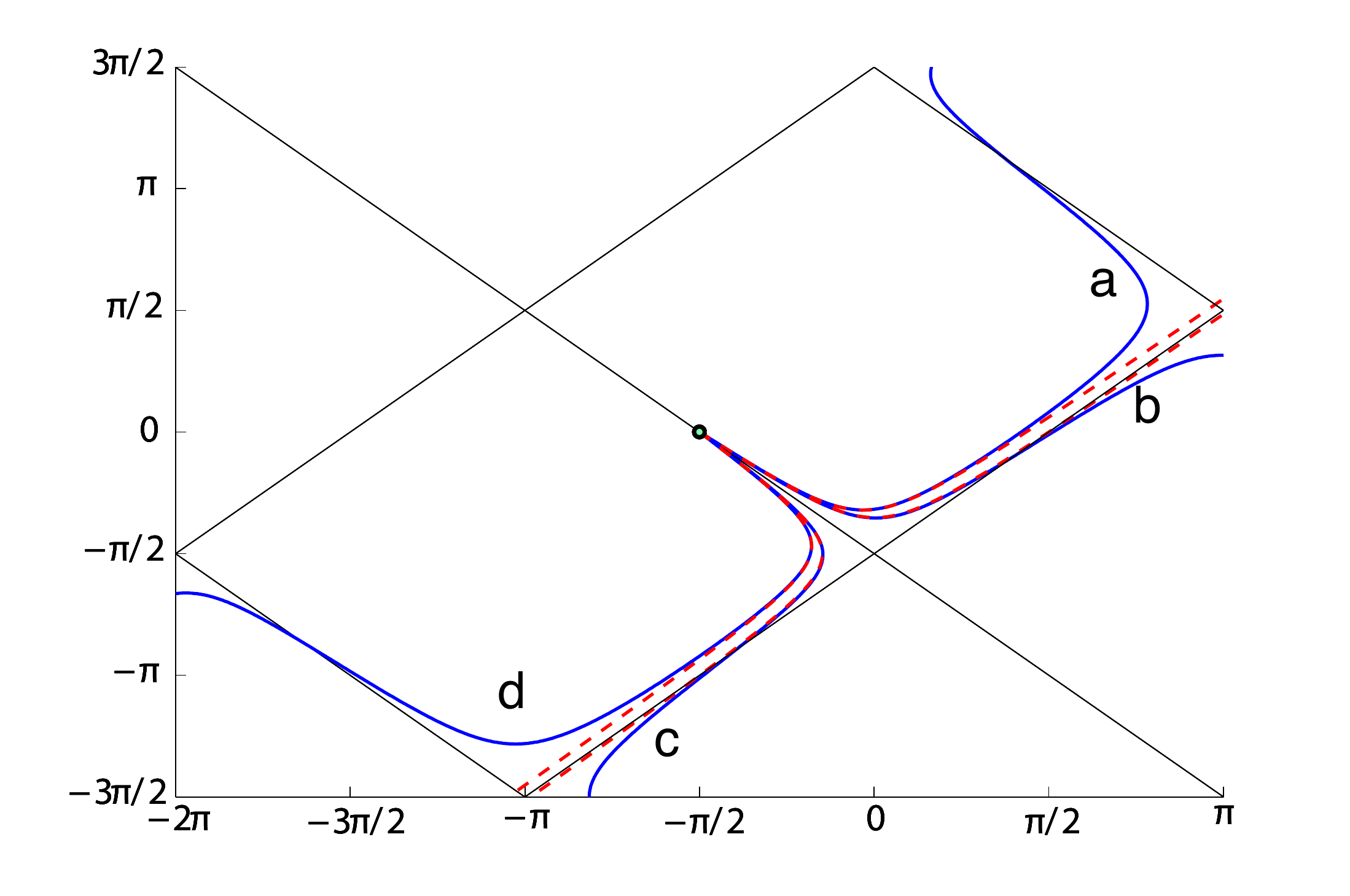}
\end{center}
\caption{Four trajectories with  $(x(0),y(0)) = (-\pi/2,0)$.   (a) $z(0)=0$;  (b) $z(0)=\pi/2$; (c) $z(0)=\pi$; (d) $z(0)=3\pi/2$.
Numerical solution (blue) and first order approximation (red dash).}
\label{fourtrajectories}
\end{figure}

%
%
\subsubsection*{Estimating $z(0)$ for the solutions with periodic $z$ components}

In \S\ref{sec.xygrowth} we showed that there are values $a$ such that the $z$ component of the trajectory  with $X_a(0)=(-\pi/2,0,a)$ is periodic.  Let us now approximate such $a$ for small $\epsilon>0$ by using perturbation analysis.  We will do this for the trajectory from Theorem \ref{zperiodic} satisfying $X_a(t+4t_a) = X_a(t) + (2\pi, 2\pi,0)$, for which $(z_0=)$ $a\in(-\pi/4,\pi/4)$.

Then $X_a(t_a)$ is the point where the trajectory crosses the plane $x+y=\pi/2$ (marked $b$ in Figure~\ref{approx}), and thus has traversed $1/4$ of the cell boundary.  Thus $z(t_a)$ (which equals $\pi/4$) can be estimated via \eqref{z1}.  Using the above solutions for $x_0, y_0, x_1, y_1$ (with $x_1(0)=y_1(0)=0$) yields
\[ z_1'(t) = \sqrt{2}\sin(z_0 + \pi/4)\gd(t),
\]
and so
\[ \frac\pi 4 = z(t_a) \approx a + \epsilon \sqrt{2} \sin(a+\pi/4) \int_0^{t_a} \gd(s)\, ds.
\]
Therefore $a$ and $t_a$ can be estimated by solving the system
\begin{eqnarray*}
\epsilon \sqrt{2} \sin(a+\pi/4) \cosh(t_a) \gd(t_a) & = & \pi
\\ 
a + \epsilon \sqrt{2} \sin(a+\pi/4) \int_0^{t_a} \gd(s)\, ds & = & \frac{\pi}{4}.
\end{eqnarray*}

\begin{figure}[h!]
\begin{center}
	 \includegraphics*[width=2.7in]{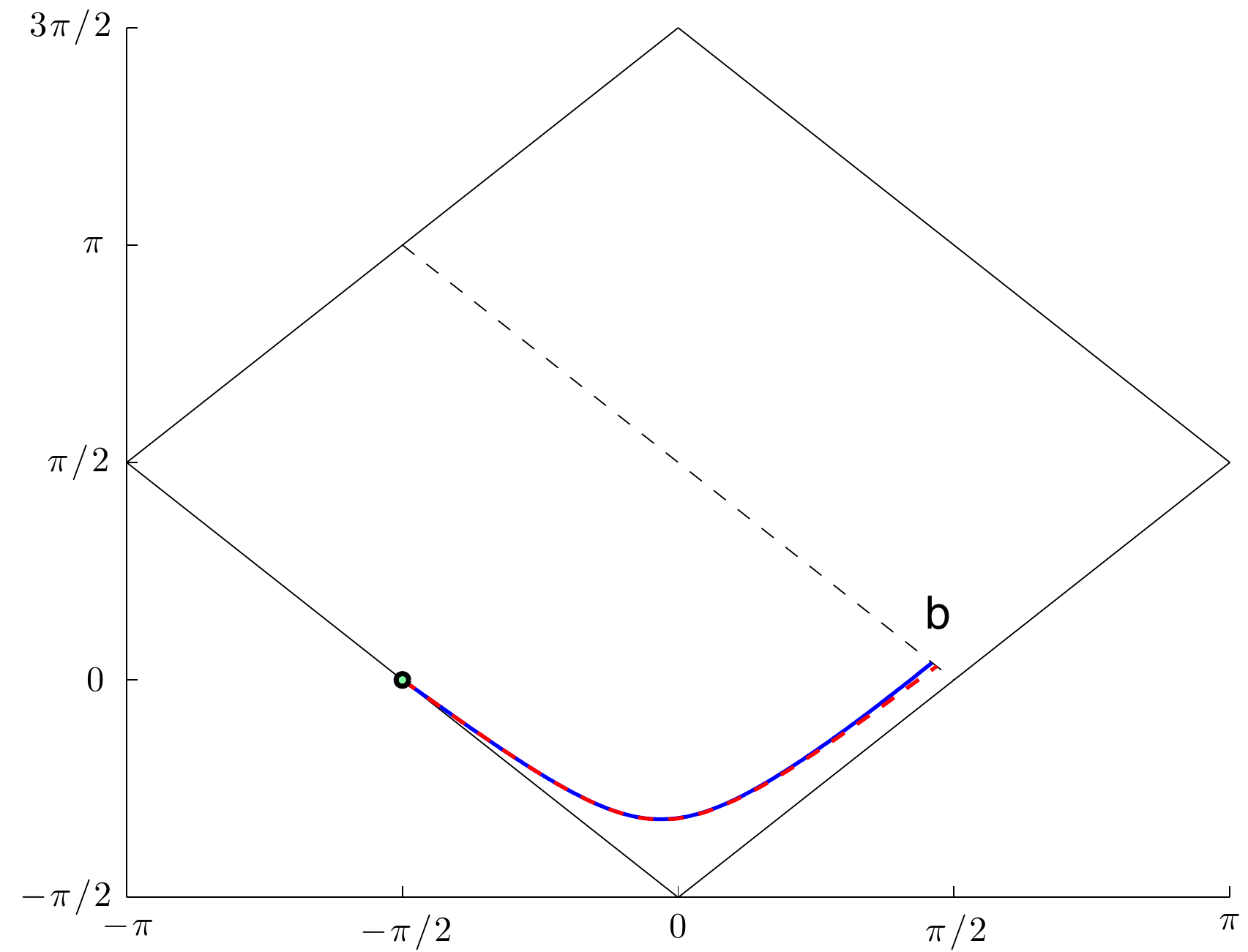}
          \includegraphics*[width=2.7in]{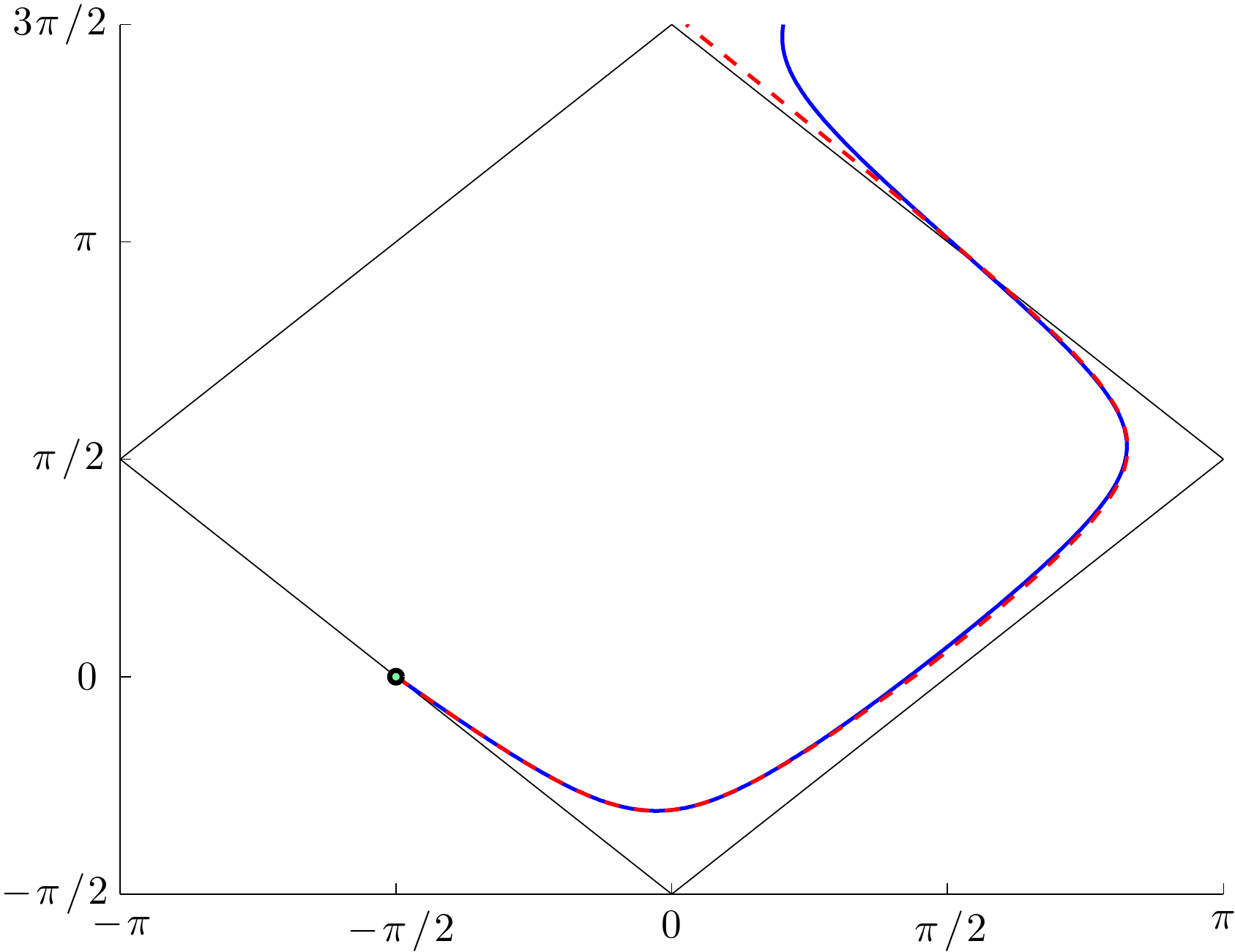}
\end{center}
\caption{Left: the point $b$ is where the trajectory crosses $x+y=\pi/2$.  Right:  Approximation (red dash)  and numerical solution (blue) for the trajectory with $\epsilon = .1$ and $z_0=0$.}
\label{approx}
\end{figure}

%
%
\subsubsection*{Crossing the cell boundary}

The first order approximation derived above breaks down, as noted, after the trajectory traverses $1/4$ of the cell boundary.  If $z_0\in (\pi/4,3\pi/4)$ and not close to the endpoints, then the trajectory has already crossed the cell boundary through orbit 1 by this point (the situation is similar, only reflected across the line $x+y=-\pi/2$, for $z_0\in(3\pi/4, 5\pi/4)$).  
For $z_0\in(-\pi/4, \pi/4)$ and not close to the endpoints, 
numerical results show that one can continue the approximation by taking an appropriate backwards-in-time perturbation from orbit 2 on Figure~\ref{heteroclinic} (see Figure~\ref{approx}; the situation is  similar for $z_0\in(5\pi/4, 7\pi/4)$).  In particular, the trajectory will cross the cell boundary through orbit 2.  For $z_0$ close to $\pi/4$ the crossing will also happen through one of these two orbits.

However, if $z_0$ is close to $-\pi/4$, 
then the trajectory may miss both these orbits.  In fact, for $z_0$ very close to $-\pi/4$, the trajectory may orbit  the cell several times before exiting.  Two such trajectories are shown in Figure~\ref{rots}.  It is not surprising that such a trajectory will eventually exit the cell, \emph{as long as it stays out of the KAM region}, since $z$ is increasing while the trajectory remains inside the cell.  It is interesting, however, that even those trajectories that orbit  the cell numerous times never seem to get caught in the KAM region (and hence seem to always exit the cell eventually).

\begin{figure}[h!]
\begin{center}
          \includegraphics*[width=2.7in]{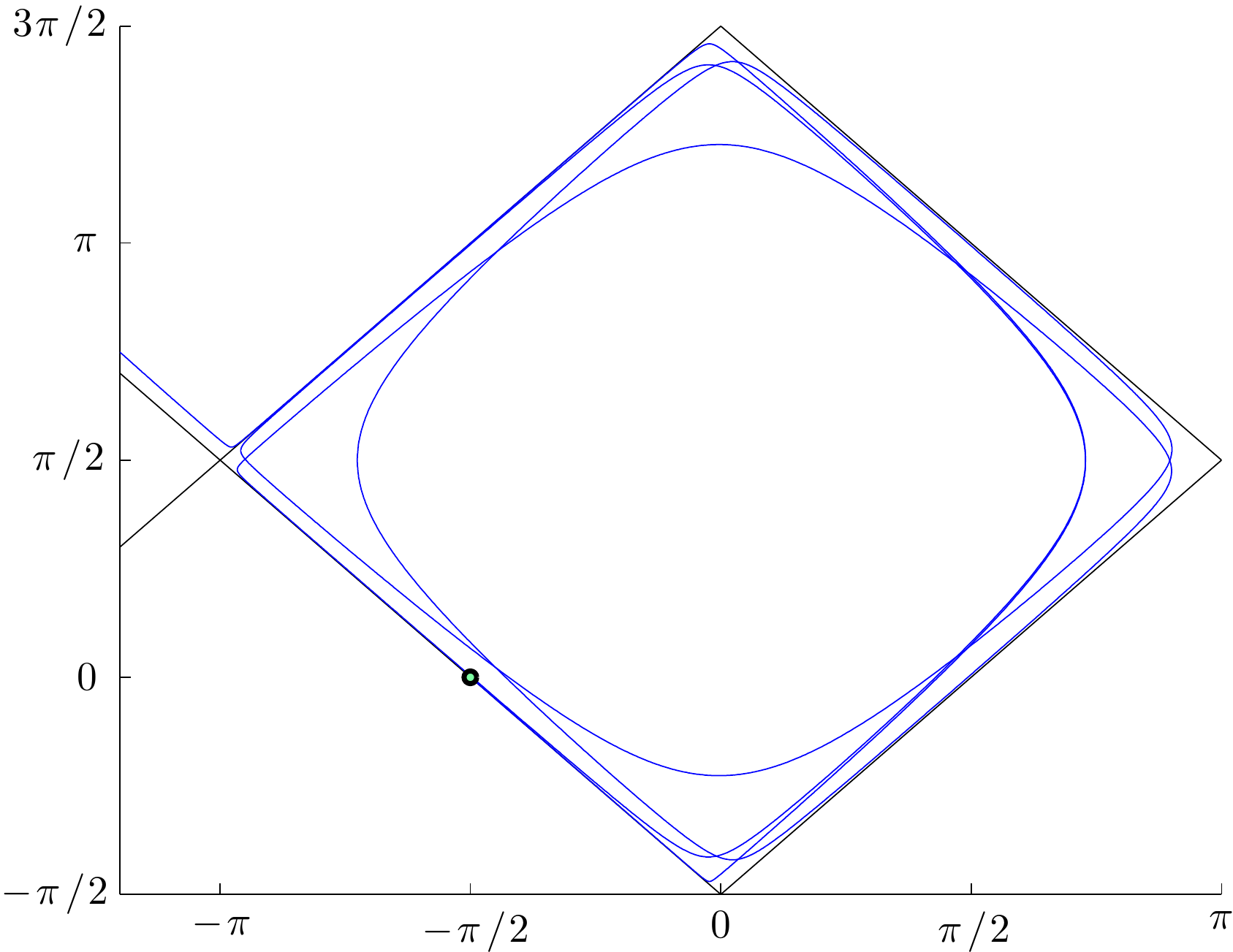}
          \includegraphics*[width=2.7in]{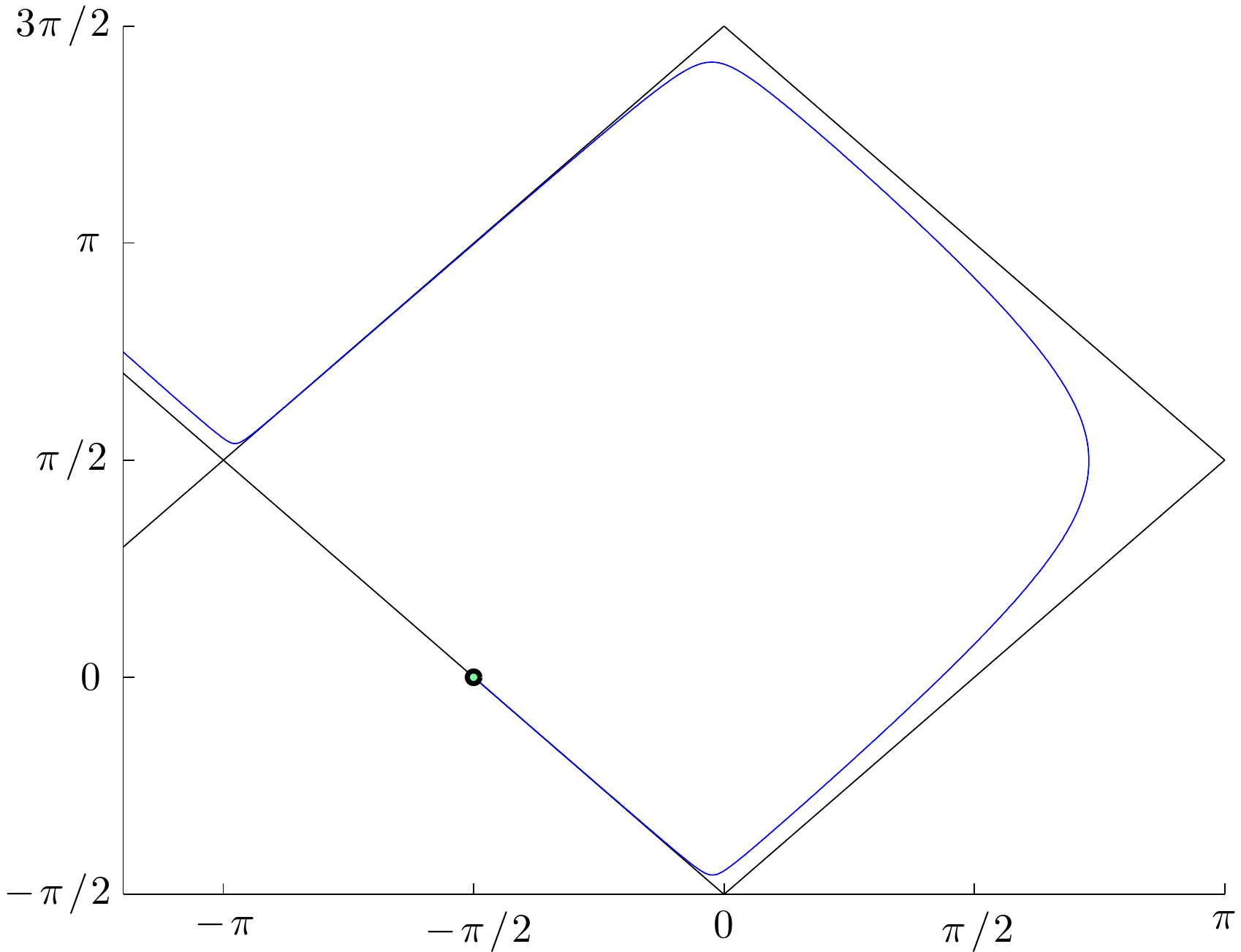}
\end{center}
\caption{Trajectories $X_a$ for $\epsilon = .1$ and $a = -\pi/4 + .01$ (left) or $a = -\pi/4 + .001$ (right).}
\label{rots}
\end{figure}

%
%
\subsection{Prevalence of linear growth in the non-KAM region}
\label{prevalence}

Trajectories that begin in the KAM region will grow linearly in $z$ (meaning that $z$ grows linearly in $t$).
In this section we present numerical results regarding the likelihood that a trajectory starting in the non-KAM region will grow linearly.  As we saw in \S\ref{sec.xygrowth}, there exist $z$-periodic trajectories that grow linearly in $x$, $y$, or both.  Are these trajectories typical for the non-KAM region, or are they exceptional?  We will present numerical evidence suggesting that the answer depends on $\epsilon$. 

\begin{figure}[h!]
\begin{center}
          \includegraphics*[width=2.7in]{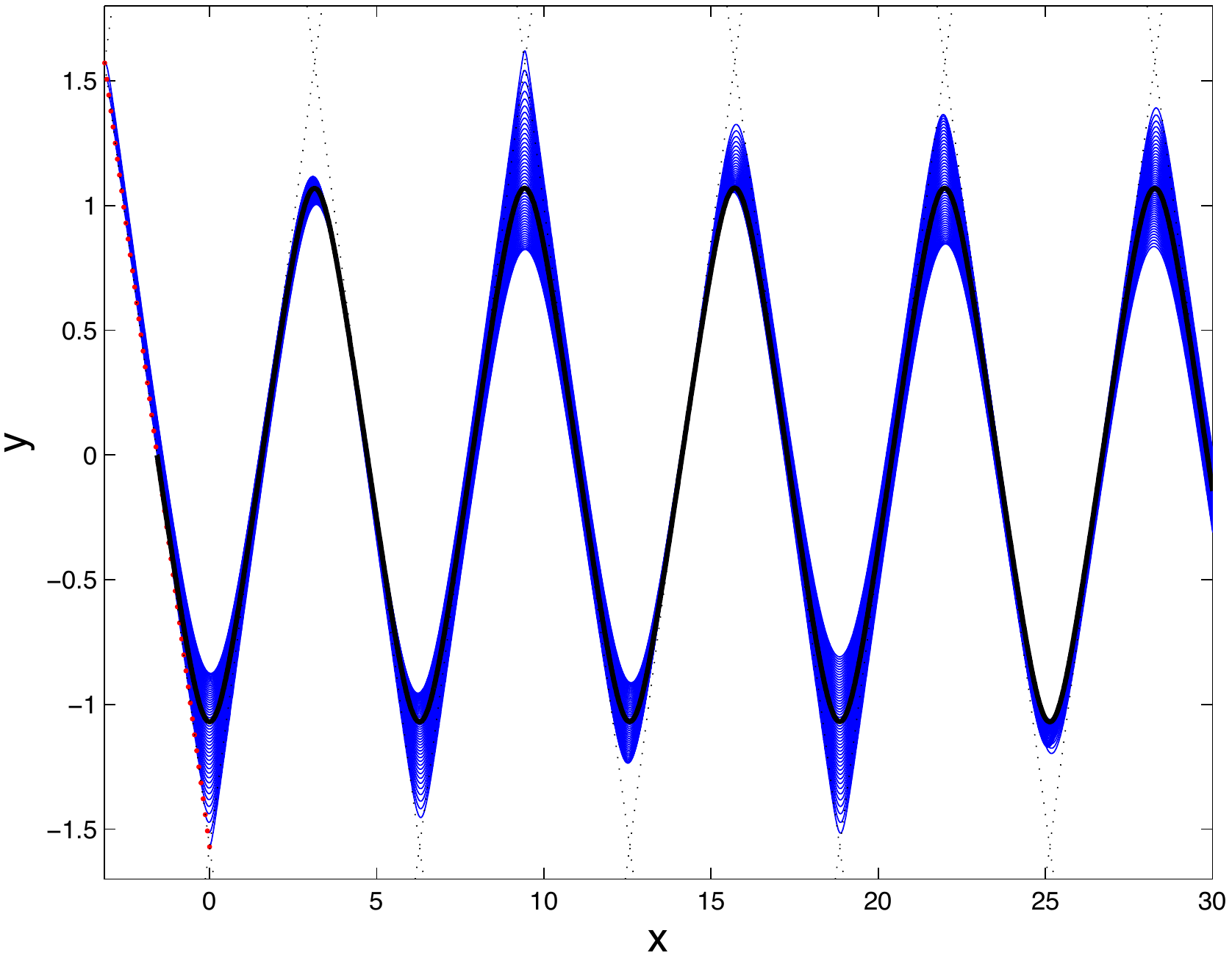}
          \includegraphics*[width=2.7in]{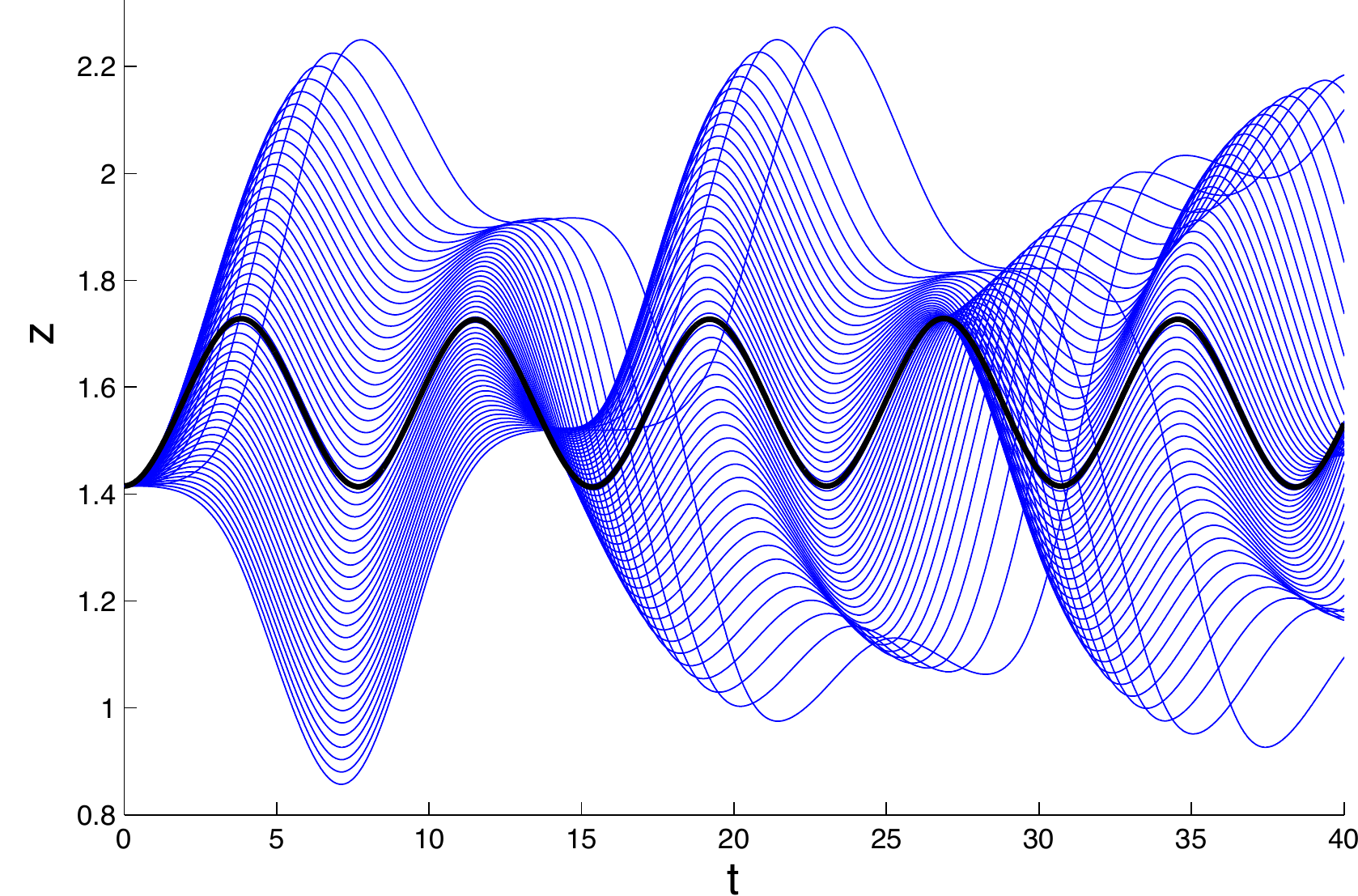}
\end{center}
\caption{Trajectories starting on orbit 4, with $\epsilon = .1$ and $z(0) = 1.4148$. }
\label{segsamez0}
\end{figure}

In Figure~\ref{segsamez0} we show trajectories with initial conditions $(x(0), y(0))$ evenly distributed on the line segment from $(-\pi,\pi/2)$ to $(0,-\pi/2)$ (i.e., orbit 4 in Figure~\ref{heteroclinic}).  The initial value $z(0)=a_c$ is the same for all trajectories, and it is that for which $X_{a_c}(t+4t_{a_c}) = X_{a_c}(t) + (2\pi, 0, 0)$ in Theorem \ref{zperiodic}.  The thick black curve is the trajectory $X_{a_c}$ (also shown on the right side of Figure~\ref{xyperiodic}), and we see that all the 
others have asymptotically linear growth in $x$.

Next we examine the effect of the initial value $z(0)$.  In Figure~\ref{segsamexy} we show 200 trajectories which all have the same $(x(0),y(0)) = (-\pi/2, 0)$ and different $z(0)$, evenly distributed from $-\pi/4$ to $7\pi/4$.  The trajectories are color coded according to which interval $z(0)$ is in.  It appears that only some of these trajectories  grow linearly in $x$, $y$ or both.

\begin{figure}[h!]
\begin{center}
          \includegraphics*[width=4in]{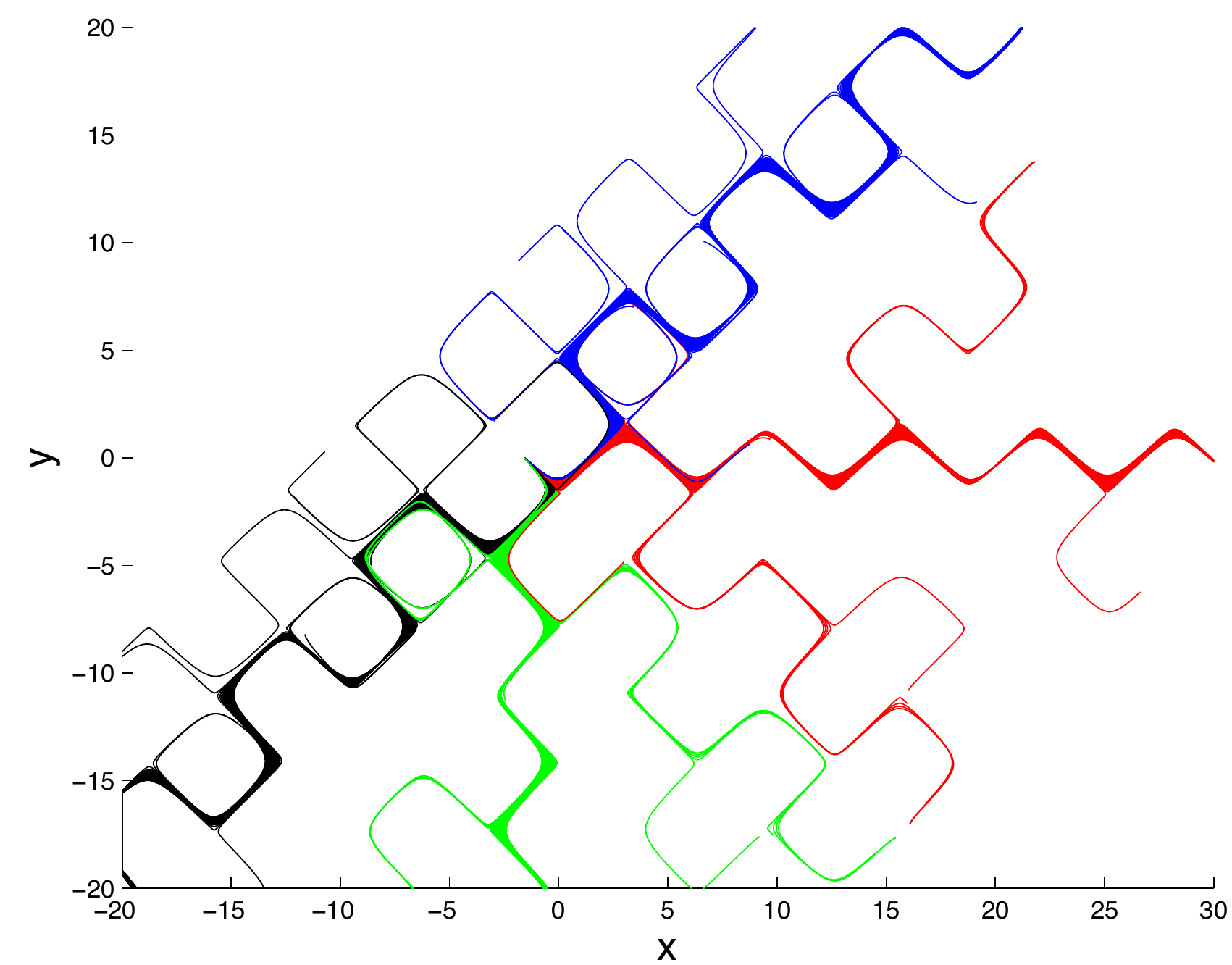}
\end{center}
\caption{Trajectories with initial conditions $(x(0),y(0)) = (-\frac \pi 2, 0)$ and  $\epsilon = 0.1$,  where $z(0)\in (-\frac{\pi}{4}, \frac{\pi}{4})$ (blue),  $z(0)\in (\frac{\pi}{4}, \frac{3\pi}{4})$ (red),   $z(0)\in (\frac{3\pi}{4}, \frac{5\pi}{4})$ (green),  and $z(0) \in (\frac{5\pi}{4}, \frac{7\pi}{4})$ (black).
}
\label{segsamexy}
\end{figure}

Based on these results, one may conjecture that trajectories will have linear growth in $x$, and be quasi-periodic in $y$ and $z$, if the initial condition is close to $(\frac{\pi}{2}, 0, a_c)$, where $a_c$ is the critical value above.
To determine how close one must start, consider the rectangle $R$ in the plane $x+y = -\frac{\pi}{2}$, centered at $(-\frac{\pi}{2}, 0, a_c)$ and with width $\sqrt{2}\pi r$ and height $\frac{\pi}2r$.  The perturbation analysis suggests that trajectories  with $z(0)\in (\frac \pi 4, \frac{3\pi} 4)$ cross through segment 1, so these are candidates for linear growth in $x$.  Consider the rectangle $R'=\{x+y = -\frac{\pi}{2} \,\&\, x\in (-\pi,0) 
\,\&\, z\in (\frac{\pi}{4}, \frac{3\pi}{4})\}$.  When  $r=1$, then  $R$ is  $R'$ shifted slightly down; and when $r$ is small, then  $R\subseteq R'$ (see Figure~\ref{Rcross}).  

\begin{figure}[h!]
\begin{center}
          \includegraphics*[width=4in]{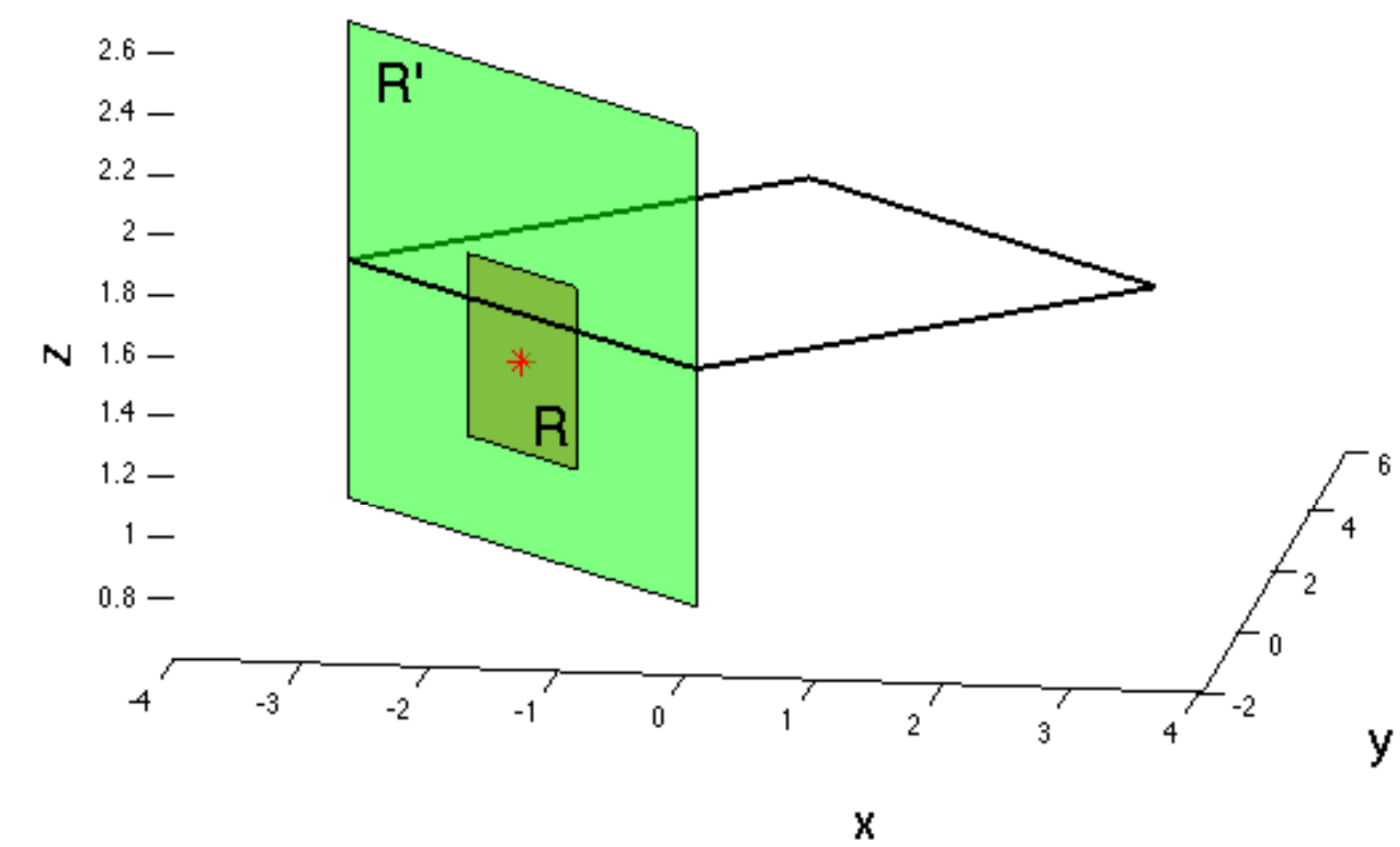}
\end{center}
\caption{$R$ is the rectangle centered at $(\frac{\pi}{2}, 0, a_c)$, shown as a star, in the plane $x+y=-\frac{\pi}{2}$.  The black square in the plane $z=\frac{\pi}{2}$ is the boundary of a cell.  $R'$ is the rectangle in the plane $x+y=-\frac{\pi}{2}$ from $z=\frac{\pi}{4}$ to $z=\frac{3\pi}{4}$.
}
\label{Rcross}
\end{figure}

Now we examine the likelihood that a trajectory crossing through $R$ will grow linearly in $x$.   For this, we take 400 points evenly distributed in the rectangle $R$ and run the simulation until time $t=50$.  We then calculate how many of these have linear growth in $x$.  The fraction is plotted in Figure~\ref{xgrowthfrac} for $\epsilon = 0.1$.  Somewhat surprisingly, this ratio remains 1 (all trajectories have linear growth) for $r$ up to about 0.5.  That is, until $R$ is about half the size of $R'$.  Then this ratio declines, but not to zero.  Even at $r=1$, the ratio is still larger than $0.5$, suggesting that at $\epsilon=0.1$, about half of trajectories in the KAM region have linear growth.

\begin{figure}[h!]
\begin{center}
          \includegraphics*[width=2.5in]{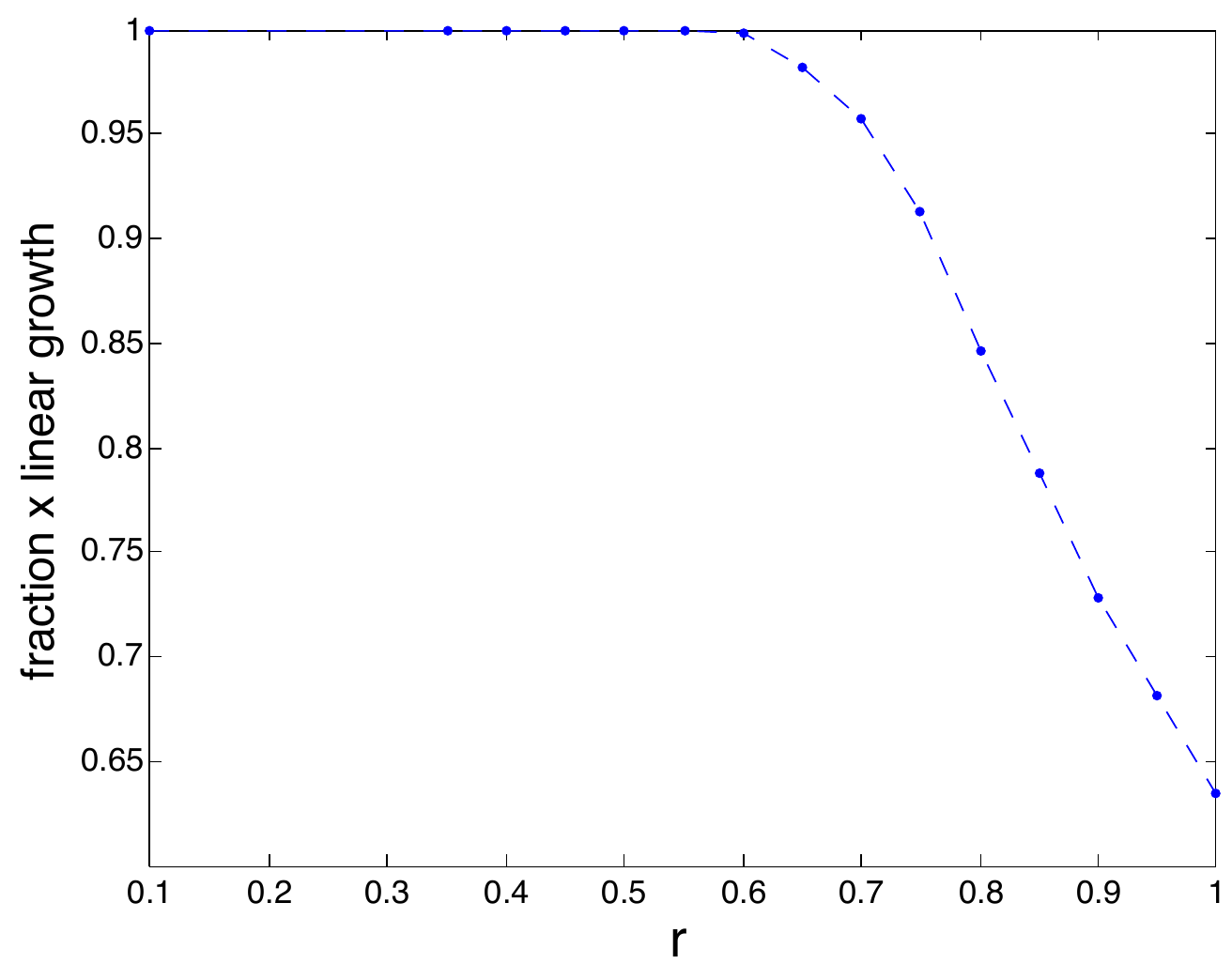}
\end{center}
\caption{Fraction of trajectories with linear growth in $x$ as a function of the size of the rectangle $R$, represented by $r$ (here $\epsilon = 0.1$).
}
\label{xgrowthfrac}
\end{figure}

Next we examine the effect of $\epsilon$ on the prevalence of linear growth.  For this we take the full rectangle $R'$, and 1000 points evenly distributed on this rectangle and run the simulation again until $t=50$.  Then we count how many trajectories have linear growth in $x$.  The fraction of trajectories with linear growth in $x$ is shown in Figure~\ref{xgrowth_eps}.  Interestingly, the fraction grows with $\epsilon$.  While we were able to prove the existence of trajectories that grow linearly in $x$ or $y$ when $\epsilon$ is small, it seems that such trajectories are actually more common when $\epsilon$ is large.  This is perhaps not surprising: when $\epsilon$ is small, we are close to the integrable case, where the non-KAM region is very small.  As $\epsilon$ increases, the KAM region shrinks and transport between cells increases.

\begin{figure}[h!]
\begin{center}
          \includegraphics*[width=2.5in]{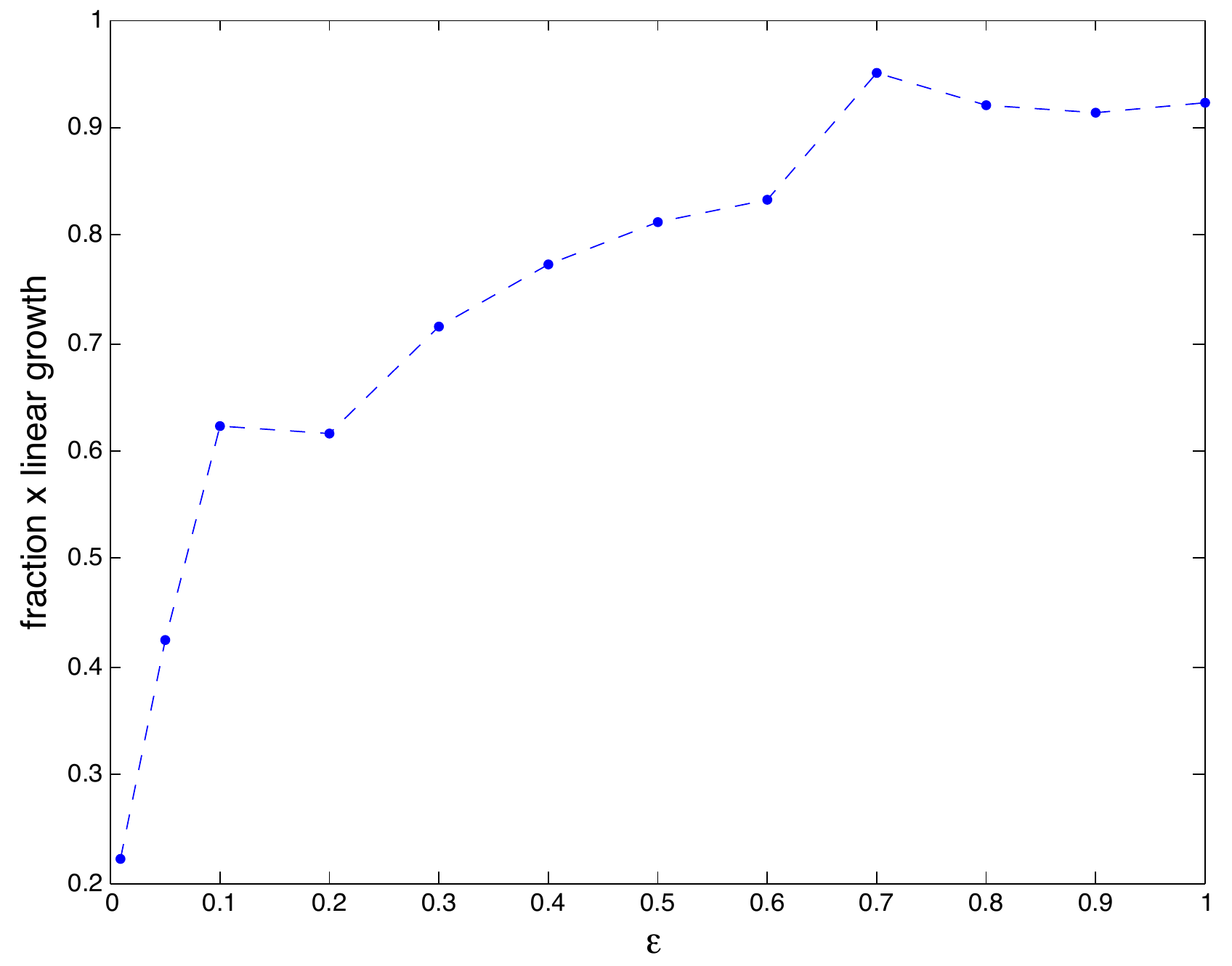}
\end{center}
\caption{Fraction of trajectories starting in the rectangle $R'$ with linear growth in $x$, as a function of $\epsilon$.
}
\label{xgrowth_eps}
\end{figure}

%
%
%
\section{Concluding remarks}
\label{sec.conclusion}
\setcounter{equation}{0}

In this paper we study solutions of the ABC flow in the near-integrable case $0<A\ll 1$ and $B=C=1$.  We are particularly interested in solutions that grow linearly in time, since directions in which such growth occurs correspond to enhanced speed of front propagation.  We find such growth in two distinct regions:  KAM and non-KAM.
Near the centers of cells, in the KAM regions, there are solutions whose $z$-components grow asymptotically linearly. Near the boundaries of cells, in the non-KAM regions, there are solutions whose $x$- and/or $y$-components grow asymptotically linearly.  These include special trajectories for which the $x$- and $y$-components are periodic in the KAM case, and the $z$-component is periodic in the non-KAM case.  Numerical evidence suggests that the asymptotically linear growth is common.

Numerical experiments also support the following conjectures.

\begin{conjecture}
Almost all trajectories that cross the boundary of a cell will cross cell boundaries infinitely many times.
\end{conjecture}

\begin{conjecture}
There exists an open set of positive measure of initial conditions such that $y$ (or $x$)  grows monotonically and $x$ (or $y$) and $z$ are bounded (quasi-periodic).  There also exists such an open set such that $x$ and $y$ grow monotonically and $z$ is bounded.
\end{conjecture}

Note that \cite{Dom_86} suggested that all ABC flows have so-called principal vortices, that is, tubes of (periodic and quasi-periodic) orbits that travel predominantly in one direction (along the axes $x$, $y$ or $z$).  Thus, the above conjecture is only a modest strengthening.

\begin{conjecture}
For any fixed initial value $z(0)=z_0$, the $xy$-plane is divided into disjoint KAM and non-KAM regions.  Any trajectory that begins in the KAM region will remain trapped in the cell it begins in (but not necessarily in the KAM region).  Any trajectory that begins outside a KAM region will cross cell boundaries infinitely many times.
\end{conjecture}

Although the picture of the dynamics of ABC flows in the near integrable case is becoming clearer, there remain unanswered questions.  These include the behavior of trajectories in the non-KAM regions near those horizontal planes  where such trajectories can rotate around a cell several times before exiting.  The dynamics in these regions appears to have some chaotic traits.
 Sander and Yorke \cite{sy15} point out several metrics for chaos, including fractal dimension of the attractor, broad power spectrum, Lyapunov exponent, positive entropy, while admitting that 
chaos defies any single definition.  We propose another metric: complex route of escape to infinity.

The results of this paper rely heavily on the exact form of the equations.  Recalling that the ABC flow can be written   in the form 
\renewcommand{\arraystretch}{1}
\ba
\frac{d}{dt}\bpm x \\ y \epm & = & \bpm 0 & 1 \\ -1 & 0 \epm \nabla H + \epsilon f(z), \label{last1}
\\[.1in]
z' & = & H(x,y), \label{last2}
\ea
it would be interesting to see whether the results of this paper can be extended to more general systems of this type.  That is,  for which types of  Hamiltonians $H$ and periodic functions $f$ does the system (\ref{last1})-(\ref{last2}) admit ballistic spiral and edge orbits?

\vskip .2in
\noindent
\textbf{Acknowledgements:} \  JX and YY would like to thank Institut Mittag-Leffler 
for its hospitality during the Fall 2014 homogenization program where some of their work on ABC flows was in progress.
JX was partially supported by NSF grants
DMS-0911277 and DMS-1211179. YY was partially supported by DMS-0901460 and
NSF CAREER award DMS-1151919. AZ was partially supported by NSF CAREER grant DMS-1056327 and NSF grant DMS-1600641.


\end{document}